\documentclass[a4paper,11pt,twoside]{amsart}
\usepackage{amssymb,amsthm,amsmath,amstext}
\usepackage{mathrsfs}  
\usepackage{bm}        
\usepackage{multirow}
\usepackage[top=1in,bottom=1in,left=1in,right=1in]{geometry}

\usepackage[all]{xy}
\usepackage{rotating}

\numberwithin{equation}{section}
\theoremstyle{plain}
\newtheorem{theorem}{Theorem}[section]
\newtheorem{proposition}[theorem]{Proposition}
\newtheorem{lemma}[theorem]{Lemma}
\newtheorem{corollary}[theorem]{Corollary}
\newtheorem{conjecture}[theorem]{Conjecture}

\newtheorem{theoremintro}{Theorem}

\newtheorem{questionintro}{Question}

\newtheorem{corollaryintro}{Corollary}

\newtheorem{definintro}{Definition}

\theoremstyle{definition}
\newtheorem{definition}[theorem]{Definition}

\theoremstyle{remark}
\newtheorem{remark}[theorem]{Remark}
\newtheorem{example}[theorem]{Example}

\numberwithin{equation}{section}

\newcommand{\sheaf}[1]{\mathscr{#1}}
\newcommand{\stack}[1]{\bm{\mathsf{#1}}}
\newcommand{\cat}[1]{\mathsf{#1}}
\newcommand{\dgcat}[1]{\mathscr{#1}}

\newcommand{\op}{^{\text{op}}}
\newcommand{\Coh}{\cat{Coh}}

\newcommand{\sig}{\sigma}
\newcommand{\HHom}{\mathrm{Hom}}
\newcommand{\EEnd}{\mathrm{End}}
\newcommand{\pullback}{^{*}}
\newcommand{\pushforward}{_{*}}
\newcommand{\rad}{\mathrm{rad}}
\newcommand{\tr}{\mathrm{tr}}
\newcommand{\ol}[1]{\overline{#1}}
\newcommand{\SB}{\mathrm{SB}}

\newcommand{\isom}{\cong}
\newcommand{\isomto}{\isom}
\newcommand{\equi}{\simeq}

\newcommand{\ind}{\operatorname{ind}}
\newcommand{\per}{\operatorname{per}}

\newcommand{\OO}{\sheaf{O}}
\newcommand{\D}{\cat{D}}
\newcommand{\Db}{\D^{\mathrm{b}}}
\newcommand{\Dbdg}{\dgcat{D}^{\mathrm{b}}}

\newcommand{\Aut}{\mathrm{Aut}}
\newcommand{\Gal}{\mathrm{Gal}}

\newcommand{\Br}{\mathrm{Br}}
\newcommand{\NS}{\mathrm{NS}}
\newcommand{\Pic}{\mathrm{Pic}}

\newcommand{\rk}{{\rm rk}}

\newcommand{\End}{\mathrm{End}}
\newcommand{\Hom}{\mathrm{Hom}}
\newcommand{\Spec}{\mathrm{Spec}}

\renewcommand{\dim}{{\rm dim}\,}

\newcommand{\sep}{{}^{\mathit{s}}}
\newcommand{\ksep}{{k\sep}}
\newcommand{\kalg}{\bar{k}}
\newcommand{\tensor}{\otimes}
\newcommand{\PGL}{\mathrm{PGL}}
\newcommand{\GL}{\mathrm{GL}}

\newcommand{\CH}{\mathrm{CH}}
\newcommand{\Forms}{\stack{Forms}}
\newcommand{\Gm}{\mathbb{G}_\mathrm{m}}
\newcommand{\mapto}[1]{\xrightarrow{#1}}

\newcommand{\sod}[1]{\langle #1 \rangle}

\newcommand{\cor}{\mathrm{cor}}

\newcommand{\dual}{^{\vee}}

\newcommand{\Ext}{{\rm Ext}}

\newcommand{\cal}{\mathscr}
\newcommand{\ka}{{\cal A}}

\newcommand{\kc}{{\cal C}}

\newcommand{\kf}{{\cal F}}

\newcommand{\ki}{{\cal I}}

\newcommand{\ko}{{\cal O}}

\newcommand{\NN}{\mathbb{N}}
\newcommand{\ZZ}{\mathbb{Z}}
\newcommand{\QQ}{\mathbb{Q}}

\newcommand{\FF}{\mathbb{F}}
\newcommand{\PP}{\mathbb{P}}

\newcommand{\RHom}{\mathbf{R}\Hom}

\def\inv{^{-1}}

\newcommand{\linedef}[1]{\textit{#1}}

\usepackage[backref=page]{hyperref}
\hypersetup{pdftitle={Semiorthogonal decompositions and birational geometry of del Pezzo surfaces over arbitrary fields},pdfauthor={Asher Auel and Marcello Bernardara}} 
\hypersetup{colorlinks=true,linkcolor=blue,anchorcolor=blue,citecolor=blue}

\begin{document}

\title[del Pezzo surfaces]{Semiorthogonal decompositions
and birational geometry of del Pezzo surfaces over arbitrary fields}

\begin{abstract} 
We study the birational properties of geometrically rational surfaces
from a derived categorical perspective.  In particular, we give a
criterion for the rationality of a del Pezzo surface $S$ over an
arbitrary field, namely, that its derived category decomposes into
zero-dimensional components. For $\deg(S) \geq 5$, we construct
explicit semiorthogonal decompositions by subcategories of modules
over semisimple algebras arising as endomorphism algebras of vector
bundles and we show how to retrieve information about the index of $S$
from Brauer classes and Chern classes associated to these vector
bundles.
\end{abstract}

\author{Asher Auel}
\address{Department of Mathematics\\ %
Yale University \\ %
10 Hillhouse Avenue \\ %
New Haven, CT 06511}
\email{asher.auel@yale.edu}

\author{Marcello Bernardara}
\address{Institut de Math\'ematiques de Toulouse \\ %
Universit\'e Paul Sabatier \\ %
118 route de Narbonne \\ %
31062 Toulouse Cedex 9\\ %
France}
\email{marcello.bernardara@math.univ-toulouse.fr}

\maketitle

\section*{Introduction}

In their address to the 2002 International Congress of Mathematicians
in Beijing, Bondal and Orlov \cite{bondal_orlov:ICM2002} suggest that
the bounded derived category $\Db(S)$ of coherent sheaves on a smooth
projective variety $S$ could provide new tools to explore the
birational geometry of $S$, in particular via semiorthogonal
decompositions.  The work of many authors \cite{kuznetsov:v14},
\cite{kuz:4fold}, \cite{bolo_berna:conic}, \cite{auel-berna-bolo},
\cite{at12} now provides evidence for the usefulness of semiorthogonal
decompositions in the birational study of complex projective varieties
of dimension at most $4$. A survey can be found in
\cite{kuz:rationality-report}.  At the same time, the relevance of
semiorthogonal decompositions to other areas of algebraic and
noncommutative geometry has been growing rapidly, as Kuznetsov
\cite{kuz:ICM2014} points out in his address to the 2014 ICM in Seoul.

In this context, based on the classical notion of representability for
Chow groups and motives, Bolognesi and the second author
\cite{bolognesi_bernardara:representability} proposed the notion of
\linedef{categorical representability} (see Definition
\ref{def-cat-rep}), and formulated the following question.

\begin{questionintro}
\label{mainquestion-repre}
Is a rational variety always categorically representable in codimension 2?
\end{questionintro}

When the base field $k$ is not algebraically closed, the existence of
$k$-rational points on $S$ (being a necessary condition for
rationality) is a major open question in arithmetic geometry. We are
indebted to H.~Esnault, who posed the following question to us in
2009.

\begin{questionintro}\label{mainquestion-points}
Let $S$ be a smooth projective variety over a field $k$.  Can the
bounded derived category $\Db(S)$ of coherent sheaves detect the
existence of a $k$-rational point on $S$?
\end{questionintro}

This question, which was the original motivation for the current work,
is now central to a growing body of research into arithmetic aspects
of the theory of derived categories, see \cite{an-au-ga-za},
\cite{antieau_krashen_ward}, \cite{auel-berna-bolo}, \cite{ABBV},
\cite{lieblich-olsson:FM}, \cite{hassett_tschinkel:derived},
\cite{vial-exceptional}, \cite{ascher-dasaratha-perry-zhou}.  As an example, if $S$ is a smooth projective
surface, Hassett and Tschinkel
\cite[Lemma~8]{hassett_tschinkel:derived} prove that the index of $S$
can be recovered from $\Db(S)$ as the greatest common divisor of the
second Chern classes of objects. Recall that the \linedef{index}
$\ind(S)$ of a variety $S$ over $k$ is the greatest common
divisor of the degrees of closed points of $S$. 
 
These questions also intertwine with concurrent developments in equivariant
and descent aspects of triangulated and dg categories in the context
of noncommutative geometry, see \cite{antieau-gepner},
\cite{elagin:descent-semi}, \cite{kuznetbasechange},
\cite{sosna:scalar}, \cite{toen:derived-azumaya}.

In the present text, we study these two questions for geometrically
rational surfaces over an arbitrary field $k$, with special attention
paid to the case of del Pezzo surfaces of Picard rank 1. The
$k$-birational classification of such surfaces was achieved by
Enriques, Manin \cite{manin-book}, and Iskovskikh
\cite{iskovskih:minimal_models}; the study of arithmetic properties
(e.g., existence of rational points, stable rationality, Hasse
principle, Chow groups) has been an active area of research since the
1970s, see \cite{colliot-thelene:ICM86} for a survey.

One of the most important invariants of a geometrically rational
surface $S$ over $k$ is the N\'eron--Severi lattice $\NS(S_{\ksep}) =
\Pic(S_{\ksep})$ of $S_{\ksep} = S \times_k \ksep$ as a module over
the absolute Galois group $G_k = \Gal(\ksep/k)$, where $\ksep$ is a
fixed separable closure of $k$.  The structure of this Galois lattice
controls the exceptional lines on $S$, and hence the possible
birational contractions.  As the canonical bundle $\omega_{S}$ is
always defined over $k$, one often considers the orthogonal complement
$\omega_{S}^{\perp}$.  For del Pezzo surfaces, this lattice is a twist
of a semisimple root lattice with the Galois action factoring through
the Weyl group, see \cite[Thm.~2.12]{kunyav_skoro_tsfas}.  C.\
Vial~\cite{vial-exceptional} has recently studied how one can, given a
$k$-exceptional collection on a surface $S$, obtain information on the
N\'eron--Severi lattice.  In particular, he shows that a geometrically
rational surface with an exceptional collection of maximal length is
rational. In this paper, we deal with the opposite extreme, of
surfaces of small Picard rank.

Even when $S$ has Picard rank 1, and there are no exceptional
curves on $S$ defined over $k$, our perspective is that the missing
information concerning the birational geometry of $S$ can be filled
in, to some extent, by considering higher rank vector bundles on $S$
that are generators of canonical components of $\Db(S)$.
  
For example, let $S$ be a quadric surface with $\Pic(S)$ generated by
the hyperplane section $\ko_S(1)$ of degree 2.  It is well known that
the rationality of $S$ is equivalent to the existence of a
rational point.  By a result in the theory of quadratic forms, a
quadric surface having a rational point is equivalent to the
corresponding even Clifford algebra $C_0$ (a quaternion algebra over
the quadratic extension defining the rulings of the quadric) being
split over its center.  Given a rational point $x \in S(k)$, the Serre
construction yields a rank 2 vector bundle $W$ as an extension of
$\ki_x$ by $\ko_S(-1)$.  On the other hand, the surface
$S_{\ksep}=\PP^1_{\ksep} \times \PP^1_{\ksep}$ carries two rank one
\linedef{spinor bundles}, corresponding to $\ko(1,0)$ and $\ko(0,1)$
(see, e.g., \cite{ottaviani}).  The vector bundle $W_{\ksep}$ is
isomorphic to $\ko(1,0) \oplus \ko(0,1)$, as the point $x$ is the
complete intersection of lines in two different rulings.  Not assuming
the existence of a rational point, such rank 2 vector bundles are only
defined \'etale locally; the obstruction to their existence is the
Brauer class of $C_0$.
There do, however, exist naturally defined vector bundles $V$ of
rank 4 on $S$, via descent of $W \oplus W$, and satisfying $\End(V) =
C_0$.  One could say that $V$ ``controls'' the birational geometry of
$S$ in a noncommutative way.

Derived categories and their semiorthogonal decompositions provide a
natural setting for a \linedef{noncommutative} counterpart of the
N\'eron--Severi lattice with its Galois action.  One can consider the
Euler form $\chi(A,B)= \sum_{i=0}^2 \dim \Ext^i(A,B)$ on the derived
category, and the $\chi$-semiorthogonal systems of simple generators,
the so-called \linedef{exceptional collections}.  Over $\ksep$, these
are known to exist and have interesting properties for rational
surfaces. In this paper, we consider the descent properties of such
collections and show how they indeed give a complete way to
interpret the link between vector bundles, semisimple algebras,
rationality, and rational points. In particular, on a geometrically
rational surface $S$, the canonical line bundle $\omega_S$ is
naturally an element of such a system. Then one can consider the
subcategory $\langle\omega_S\rangle^{\perp}$, whose object are
$\chi$-semiorthogonal to $\omega_S$. In practice, it can be more
natural to consider the category $\cat{A}_S
=\langle\ko_S\rangle^\perp$, which is equivalent to
$\langle\omega_S\rangle^\perp$.  We describe decompositions (or
indecomposability) of this subcategory by explicit descent methods for
vector bundles.  In this context, certain semisimple $k$-algebras
naturally arise and control the birational geometry of $S$. These
algebras also provide a presentation of the $K$-theory of $S$.

A sample result of this kind, that does not seem to be contained in
the literature, is as follows.  Let $S$ be a del Pezzo surface of
degree 5 over a field $k$.  It is known that $S$ is uniquely
determined by an \'etale algebra $l$ of degree 5 over $k$, see
\cite[Thm.\ 3.1.3]{skorobogatov:torsors_rational_points}.  We prove a
$k$-equivalence $\Db(S) = \Db(A)$, where $A$ is a finite dimensional
algebra whose semisimplification is $k \times k \times l$.  In
particular, there is an isomorphism
$$
K_i(S) \isom K_i(k) \times K_i(k) \times K_i(l)
$$ 
in algebraic $K$-theory, for all $i \geq 0$.  In the context of
computing the algebraic $K$-theory of geometrically rational surfaces,
this adds to the work of Quillen \cite{quillen:higher_K-theory} for
Severi--Brauer surfaces (i.e., del Pezzo surfaces of degree 9), Swan
\cite{swan:quadric_hypersurfaces} and Panin
\cite{panin:twisted_flag_varieties} for quadric surfaces and
involution surfaces (i.e., minimal del Pezzo surfaces of degree 8),
and Blunk \cite{blunk-dp6} for del Pezzo surfaces of degree 6.  Our
method provides a uniform way to compute the algebraic $K$-theory of
all del Pezzo surfaces degree $\geq 5$.

S.~Gille \cite{gille} recently studied the Chow motive with integer
coefficients of a geometrically rational surface, showing that they
are zero dimensional if and only if the surface has a zero cycle of
degree one and the Chow group of zero-cycles is torsion free after
base changing to the function field.  Tabuada (cf.\
\cite{tabuada-survey}) has shown how the Morita equivalence class of
the derived category of a smooth projective variety (with its
canonical dg-enhancement) gives the \linedef{noncommutative motive} of
the variety. We refrain here from defining the additive, monoidal, and
idempotent complete category of noncommutative motives, whose objects
are smooth and proper dg categories up to Morita equivalence; the
interested reader can refer to \cite{tabuada-survey}. We recall only
that any semiorthogonal decomposition gives a splitting of the
noncommutative motive (see \cite{tabuada-universal}) and that, as well
as Chow motives, noncommutative motives can be defined with
coefficients in any ring $R$. Working with $\QQ$-coefficients, there
is a well established correspondence between noncommutative and Chow
motives, see \cite{tabuada-CvsNC}.  Roughly speaking, the category of
Chow motives embeds fully and faithfully into the category of
noncommutative motives, if one ``forgets'' the Tate motive. As pointed
out in \cite{berna-tabuada-decos}, this relation does not hold in
general for integer coefficients. In particular, for surfaces, one has
to invert $2$ and $3$ in the coefficient ring to have such a direct
comparison (see \cite[Cor.~1.6]{berna-tabuada-decos}).  We think that
a comparison between Gille's results and those of the present paper
could lead to a deep understanding of the different types of geometric
information encoded by Chow and noncommutative motives.

\medskip

By a del Pezzo surface $S$ over a field $k$, we mean a smooth,
projective, and geometrically integral surface over $k$ with ample
anticanonical class.  The degree of $S$ is the self-intersection
number of the canonical class $\omega_S$. Our main result is the
following.

\begin{theoremintro}
\label{thm:main}
A del Pezzo surface $S$ of Picard rank one over a field $k$ is
$k$-rational if and only if $S$ is categorically representable in
dimension 0.

In particular, if $S$ has degree $d \geq 5$ and any Picard rank, the
following are equivalent:
\begin{enumerate}
 \item\label{thm:1} $S$ is $k$-rational.
 \item\label{thm:2} $S$ has a $k$-rational point.
 \item\label{thm:3} $S$ is categorically representable in dimension $0$.
\item\label{thm:4} $\cat{A}_S \simeq \langle \omega_S \rangle^{\perp}$ is representable in dimension $0$.
\end{enumerate}
\end{theoremintro}

The equivalence of \ref{thm:1} and \ref{thm:2} for del Pezzo surfaces
of degree $\geq 5$ is a result of Manin~\cite[Thm.~29.4]{manin-book}.
Our work is a detailed analysis of how birational properties interact
with certain semiorthogonal decompositions.  In particular, we study
collections (so-called \linedef{blocks}) of completely orthogonal
exceptional vector bundles on $S_{\ksep}$ and their descent to $S$.
The theory of 3-block decompositions due to Karpov and
Nogin~\cite{karpov-nogin} is indispensable for our work.

In our results for del Pezzo surfaces of higher Picard rank, the bound
on the degree is quite important.  This is not surprising, as del
Pezzo surfaces of degree $d \geq 5$ have a much simpler geometry and
arithmetic than those of lower degree.  For example, in smaller degree
the existence of a rational point only implies unirationality
\cite[Thm. 29.4]{manin-book}.  On the other hand, a succinct corollary
of Theorem~\ref{thm:main} gives a positive answer to Question
\ref{mainquestion-repre} for all del Pezzo surfaces.

\begin{corollaryintro}\label{cor:main}
Any rational del Pezzo surface $S$ is categorically representable in
dimension 0.
\end{corollaryintro}

This relies on the fact that there is no minimal $k$-rational del
Pezzo surface of degree $d \leq 4$ (see
\cite[Thm. 3.3.1]{manin-tsfasman}).  So Question
\ref{mainquestion-points} in smaller degrees, where the existence of a
$k$-point is weaker than categorical representability in dimension 0,
remains open.

\medskip

One of the ingredients in the proof of Theorem \ref{thm:main} is the
description of semiorthogonal decompositions of minimal del Pezzo
surfaces of degree $\geq 5$.  This relies on, and extends, the special
cases established by Kuznetsov \cite{kuznetquadrics}, Blunk
\cite{blunk-gen-bs}, Blunk, Sierra, and Smith
\cite{blunk-sierra-smith}, and the second author
\cite{bernardara:brauer_severi}.  In these cases, it always turns out
that there are semisimple $k$-algebras $A_1$ and $A_2$, Azumaya over
their centers, and a semiorthogonal decomposition
\begin{equation}
\label{semiorth-for-dps}
\cat{A}_S \simeq \langle \omega_S \rangle^{\perp} = \langle \Db(k,A_1), \Db(k,A_2) \rangle
\end{equation}
over $k$. The algebras $A_i$ arise as endomorphism algebras of vector
bundles with special homological features. In particular, over
$\ksep$, these bundles split into a sum of completely orthogonal
exceptional bundles, whose existence was proved by Rudakov and
Gorodentsev \cite{gorodentsev-rudakov}, \cite{rudakov-quadric}, and by
Karpov and Nogin \cite{karpov-nogin}.  This generalizes the example of
quadric surfaces explained above.

\begin{theoremintro}\label{prop:2nd-part-of-main1}
Let $S$ be a del Pezzo surface of degree $d \geq 5$.
There exist vector
bundles $V_1$ and $V_2$ and a semiorthogonal
decomposition
\begin{equation}\label{eq:the-dec-of-AS}
\cat{A}_S = \langle V_1, V_2 \rangle = \langle \Db(k,A_1), \Db(k,A_2)
\rangle = \langle \Db(l_1/k,\alpha_1), \Db(l_2/k,\alpha_2) \rangle
\end{equation}
where $A_i = \End(V_i)$ are semisimple $k$-algebras with centers
$l_i$, and $\alpha_i$ is the class of $A_i$ in $\Br(l_i)$.  Finally,
$S$ has Picard rank 1 if and only if the vector bundles $V_i$ are
indecomposable or, equivalently, the algebras $A_i$ are simple.
\end{theoremintro}

One of our main technical results is that such a decomposition of
$\cat{A}_S \simeq \langle \omega_S \rangle^\perp$ can never occur if
$S$ has Picard rank 1 and degree $\leq 4$, see
Theorem~\ref{thm:decos-of-minimal-dps}.  In fact, in this case one
should expect $\cat{A}_S$ to have a very complicated algebraic
description, but not much is known (see Theorem \ref{thm:A-for-dp4}
and Theorem \ref{thm:A-for-dp3}).

A feature of these semiorthogonal decompositions is that the
$k$-birational class of $S$ is determined by the unordered pair of
semisimple algebras $(A_1,A_2)$ up to pairwise Morita equivalence.  As
pointed out by Kuznetsov~\cite[\S 3]{kuz:rationality-report}, one of
most tempting ideas in the theory of semiorthogonal decompositions
with a view toward rationality is to define, in any dimension and
independently of the base field, a categorical analog of the Griffiths
component of the intermediate Jacobian of a complex threefold.  Such
an analog would be the best candidate for a birational invariant. It
turns out that such a component is not well defined in general.
However, this is not the case for del Pezzo surfaces.  Indeed, we can
define the \linedef{Griffiths--Kuznetsov component} $\cat{GK}_S$ of a
del Pezzo surface $S$ in the case when $S$ has degree $d \geq 5$ or
Picard rank $1$ as follows.

\begin{definintro}\label{def:grif-kuz-comp}
Let $S$ be a minimal del Pezzo surface of degree $d$ over $k$.  We
define the \linedef{Griffiths--Kuznetsov component} $\cat{GK}_S$ of
$S$ as follows: if $\cat{A}_S$ is representable in dimension 0, set
$\cat{GK}_S=0$.  If not, $\cat{GK}_S$ is either the product of all
indecomposable components of $\cat{A}_S$ of the form $\Db(l,\alpha)$
with $l/k$ a field extension and $\alpha \in \Br(l)$ nontrivial or
else $\cat{GK}_S=\cat{A}_{S}$.

If $S$ is not minimal, then we set $\cat{GK}_S=\cat{GK}_{S'}$ for a minimal 
model $S \to S'$.
\end{definintro}

Definition \ref{def:grif-kuz-comp} may appear ad hoc, but it can 
roughly be rephrased by saying that the Griffiths--Kuznetsov component is
the product of all components of $\cat{A}_S$ which are not
representable in dimension 0.  If $\deg(S) \geq 5$, then $\cat{GK}_S$
is determined up to equivalence by Brauer classes related to the
algebras $A_1$ and $A_2$. If $\deg(S) \leq 4$ and $S$ has Picard rank
$1$, we conjecture its indecomposability (see Conjecture
\ref{conj:indecomp-fot-low-deg}).
\begin{theoremintro}
\label{thm:birat-class-of-dpsmall}
Let $S$ be a del Pezzo surface of degree $d$ over $k$. The
Griffiths-Kuznetsov component $\cat{GK}_S$ is well-defined, unless
$\deg(S)=4$ and $\ind(S)=1$, or $\deg(S)=8$ and $\ind(S)=2$.
Moreover, if $S' \dashrightarrow S$ is a birational map, then
$\cat{GK}_S = \cat{GK}_{S'}$.
\end{theoremintro}

We notice that the two ``bad'' cases of degree 4 and index 1 or degree
8 and index 2 are birational to conic bundles which are not forms of
Hirzebruch surfaces. In the Appendix, we will show how $\cat{A}_S$ can
be interpreted as categories naturally attached to conic bundles.

Theorem \ref{thm:birat-class-of-dpsmall} could be considered as an
analogue of Amitsur's theorem that
$k$-birational Severi--Brauer varieties have Brauer classes that
generate the same subgroup of $\Br(k)$.

In the case where $d \geq 5$, we push this further to analyze how the
algebras $A_i$ obstruct the existence of closed points of given
degree.  The \linedef{index} $\ind(A)$ of a central simple $k$-algebra
$A$ is the greatest common divisor of the degrees of all central
simple $k$-algebras Morita equivalent to $A$.  We extend this
definition to finite dimensional semisimple algebras by taking the
least common multiple of the indices of simple components considered
over their centers.

\begin{theoremintro}\label{thm:obstruction-to-low-deg-points}
Let $S$ be a non-$k$-rational del Pezzo surface of degree $d \geq 5$,
and $A_1$, $A_2$ the associated semisimple algebras.  If $A_i$ has
index $d_i$ then $\ind(S) = d_1 d_2/\!\gcd(d_1,d_2)$.  In particular:
\begin{itemize}
 \item If $d=9$, then $d_1=d_2$ is either $1$ if $S \simeq \PP^2_k$ or $3$
 if $S(k) \neq \varnothing$.

\item If $d=8$ and $S$ is an involution surface, then (up to
renumbering) $A_1$ is a central simple $k$-algebra, $d_1|4$ and $d_2|2$, 
and $d_1=1$ if and only if $S$ is a quadric in $\PP^3_k$ while $d_2=1$
if and only $S$ is rational.

\item If $d=6$, then (up to renumbering) $d_1|3$ and $d_2|2$, with $S$ birationally rigid if $d_i > 1$.

\item In all other cases, $d_1=d_2=1$ and $S$ is rational.
 \end{itemize}
\end{theoremintro}

The key of the proof of Theorem
\ref{thm:obstruction-to-low-deg-points} is the explicit description of
the vector bundles generating the exceptional blocks of $S$, together
with an analysis of all possible Sarkisov links,
cf.\ \cite{isko-sarkisov-complete}.  

For del Pezzo surfaces of Picard rank one and degree at least $5$, a
consequence of our results is that the index of $S$ can be calculated
only in terms of the second Chern classes of generators of the 3
blocks, improving upon, in this case, a result of Hassett and
Tschinkel \cite[Lemma 8]{hassett_tschinkel:derived}.

\begin{theoremintro}\label{thm:index-and-c2}
Let $S$ be a del Pezzo surface of degree $d \geq 5$.
There exist generators $V_i$ for the three blocks such that
$$
\ind(S) = \gcd \{ c_2(V_0), c_2(V_1), c_2(V_2) \}.
$$
Moreover, unless $d=5$, $d=\ind(S)=6$, or $S$ is an anisotropic
quadric surface in $\PP^3$ then
$$
\ind(S) = \gcd \{ c_2(V_1), c_2(V_2)\}
$$
for indecomposable generators of the blocks of $\cat{A}_S$.
\end{theoremintro}
Notice that more precise and detailed statements can be given in each
specific case. Consult Sections \ref{subsec:dP9}--\ref{subsec:dP5} and
Tables \ref{table:dp9}--\ref{table:dp5} for details.

In summary, our results establish a complete understanding of the
relationship between birational geometry, derived categories, and
vector bundles on del Pezzo surfaces of degree $\geq 5$.

{\small
\subsection*{Structure.}
The paper is organized
as follows. Part 1 organizes the necessary algebraic and geometric
background. In Section \ref{sect:semideco} we treat semiorthogonal
decompositions and a generalized notion of exceptional objects, whose
descent is treated in Section \ref{sec:descent_sodec}. Section
\ref{surfaces} collects useful results on geometrically rational
surfaces.  Part 2 is dedicated to a detailed statement and proof of
the main results. Before proceeding to the proofs, we recall the
exceptional block theory in Section \ref{sect:karpovnogin}.  Section
\ref{sect:decos-of-minimal-dps} proves the main results for surfaces
of low degree, while the higher degree cases are treated separately in
Sections \ref{subsec:dP9}--\ref{subsec:dP5}. Part 3 consists of three
appendices containing useful calculations related to elementary
links.
}

{\small
\subsection*{Acknowledgements.} 
The authors are grateful to H.~Esnault, whose initial questions formed
one of the most important inspirations for this work. They would also
like to acknowledge B.~Antieau, P.~Belmans, J.-L.~Colliot-Th\'el\`ene,
S.~Gille, N.~Karpenko, S.~Lamy, R.~Parimala, A.~Qu\'eguiner-Mathieu,
T.~V\'arilly-Alvarado, B.~Viray, and Y.~Tschinkel for useful
conversations and remarks.  This project was started the week
following hurricane Sandy, during which time the Courant Institute at
New York University graciously hosted the authors in one of the few
buildings in lower Manhattan not completely without electricity.  Yale
University and the American Institute of Mathematics also provided
stimulating working conditions for both authors to meet and are warmly
acknowledged.  The first author was partially sponsored by NSF grant
DMS-0903039 and a Young Investigator grant from the NSA.
}

{\small
\subsection*{Notations.}
Fix an arbitrary field $k$.  If $X$ is a $k$-scheme, we will denote by
$\Db(X)$ the bounded derived category of complexes of coherent
sheaves on $X$, considered as a $k$-linear category.   
If $B$ is an $\OO_X$-algebra, we will denote by $\Db(X,B)$ the
bounded derived category of complexes of $B$-modules, considered as a
$k$-linear category.  If $X=\Spec(K)$ is an affine $k$-scheme and $B$
is associated to a $K$-algebra, we will write $\Db(K/k,B)$ as
shorthand.  Also $\Db(K,A)$ is shorthand for $\Db(K/K,A)$.  If
$A$ is the restriction of scalars of $B$ down to $k$, we remark that
$\Db(K/k,B)$ is $k$-equivalent to $\Db(k,A)$.  Finally, if $B$ is
Azumaya with class $\beta$ in $\Br(X)$, by abuse of notation, we will
write $\Db(X,\beta)$ for $\Db(X,B)$.

Triangulated categories will be commonly denoted with sans-serif font;
dg-categories with calligraphic font; finite dimensional $k$-algebras
with upper case Roman letters near the beginning of the alphabet; finite products of field extensions of
$k$ with $l$;
vector bundles with upper case Roman letters; and Brauer classes as
lower case Greek letters.
}

\newpage

\setcounter{tocdepth}{1}
\tableofcontents

\part{Background on derived categories and geometrically rational surfaces}

\section{Semiorthogonal decompositions and categorical representability}\label{sect:semideco}

We start by recalling the categorical notions that play the main role
in this paper.  Let $k$ be an arbitrary field. The theory of
exceptional objects and semiorthogonal decompositions in the case
where $k$ is algebraically closed and of characteristic zero was
studied in the Rudakov seminar at the end of the 80s, and developed by
Rudakov, Gorodentsev, Bondal, Kapranov, and Orlov among others, see
\cite{gorodentsev-rudakov}, \cite{bondal:representations},
\cite{bondal_kapranov:reconstructions},
\cite{bondal_orlov:semiorthogonal}, and \cite{helices-book}. As noted
in \cite{auel-berna-bolo}, most fundamental properties persist over
any base field $k$.

\subsection{Semiorthogonal decompositions and their mutations}
Let $\cat{T}$ be a $k$-linear triangulated category. 
A full triangulated subcategory
$\cat{A}$ of $\cat{T}$ is called \it admissible \rm if the embedding
functor admits a left and a right adjoint.

\begin{definition}[\cite{bondal_kapranov:reconstructions}]
\label{def-semiortho}
A {\em semiorthogonal decomposition} of $\cat{T}$ is a sequence of
admissible subcategories $\cat{A}_1, \ldots, \cat{A}_n$ of $\cat{T}$
such that

\begin{itemize}

\item $\Hom_{\cat{T}}(A_i,A_j) = 0$ for all $i>j$ and any $A_i$ in $\cat{A}_i$
and $A_j$ in $\cat{A}_j$; 

\item for all objects $A_i$ in $\cat{A}_i$ and $A_j$ in $\cat{A}_j$,
and for every object $T$ of $\cat{T}$, there is a chain
of morphisms $0=T_n \to T_{n-1} \to \ldots \to T_1 \to T_0 = T$ such that the cone
of $T_k \to T_{k-1}$ is an object of
$\cat{A}_k$ for all $k=1,\ldots,n$. 

\end{itemize}

Such a decomposition will be written
$$\cat{T} = \langle \cat{A}_1, \ldots, \cat{A}_n \rangle.$$
\end{definition}

If $\cat{A} \subset \cat{T}$ is admissible, we have
two semi-orthogonal decompositions
$$
\cat{T}=\langle \cat{A}^{\perp}, \cat{A} \rangle = \langle \cat{A},
^{\perp}\cat{A} \rangle,
$$
where $\cat{A}^{\perp}$ and $^\perp\cat{A}$ are, respectively, the
left and right orthogonal of $\cat{A}$ in $\cat{T}$ (see \cite[\S
3]{bondal_kapranov:reconstructions}).

Given a semiorthogonal decomposition $\cat{T} = \langle \cat{A},
\cat{B} \rangle$, Bondal \cite[\S3]{bondal:representations} defines
left and right mutations $L_{\cat{A}}(\cat{B})$ and
$R_{\cat{B}}(\cat{A})$ of this pair. In particular, there are
equivalences $L_{\cat{A}}(\cat{B}) \equi \cat{B}$ and
$R_{\cat{B}}(\cat{A}) \equi \cat{A}$, and semiorthogonal
decompositions
$$
\begin{array}{ccc}
\cat{T}=\langle L_{\cat{A}}(\cat{B}), \cat{A} \rangle, & \,\,\,\,\, &  
\cat{T} = \langle \cat{B}, R_{\cat{B}}(\cat{A}) \rangle.
\end{array}
$$
We refrain from giving an explicit definition for the mutation
functors in general, which can be found in \cite[\S
3]{bondal:representations}.  In \S\ref{subs:exc-objects} we will give
an explicit formula in the case where $\cat{A}$ and $\cat{B}$ are
generated by exceptional objects.

\subsection{Exceptional objects, blocks, and mutations}
\label{subs:exc-objects}

Very special examples of admissible subcategories, semiorthogonal
decompositions, and their mutations are provided by the theory of
exceptional objects and blocks. These constructions appear naturally
on Fano varieties, and especially on geometrically rational surfaces,
and they play a central role in our study. We provide here a
detailed treatment, generalizing to any field notions usually studied
over algebraically closed fields.

Let $\cat{T}$ be a $k$-linear triangulated category. The triangulated
category $\sod{\{V_i\}_{i\in I}}$ \linedef{generated} by a class of
objects $\{V_i\}_{i\in I}$ of $\cat{T}$ is the smallest thick (that
is, closed under direct summands) full triangulated subcategory of
$\cat{T}$ containing the class.  We will write $\Ext_\cat{T}^r(V,W) = \Hom_\cat{T}(V,W[r])$.

\begin{definition}
\label{def-except}
Let $A$ be a division (not necessarily central) $k$-algebra (e.g., $A$
could be a field extension of $k$).  An object $V$ of $\cat{T}$ is
called \linedef{$A$-exceptional} if
$$
\Hom_{\cat{T}}(V,V) = A \quad \text{and} \quad
\Ext^r_{\cat{T}}(V,V)=0 \quad \text{for} \quad r \neq 0.
$$
An exceptional object in the classical sense
\cite[Def.~3.2]{gorodentsev-moving} of the term is a $k$-exceptional
object.  By \linedef{exceptional} object, we mean $A$-exceptional for
some division $k$-algebra $A$.

A totally ordered set $\{V_1,\ldots,V_n\}$ of exceptional objects is
called an \linedef{exceptional collection} if
$\Ext^r_{\cat{T}}(V_j,V_i)=0$ for all integers $r$ whenever $j>i$.
An exceptional collection is \linedef{full} if it generates $\cat{T}$,
equivalently, if for an object $W$ of $\cat{T}$, the vanishing
$\Hom_{\cat{T}}(W,V_i)=0$ for all $i$ implies $W=0$.
An exceptional collection is \linedef{strong} if
$\Ext^r_{\cat{T}}(V_i,V_j)=0$ whenever $r\neq 0$.
\end{definition}

\begin{remark}
\label{rem:warning_exceptional}
The extension of scalars of an exceptional object (in this generalized
sense) need not be exceptional, see Remark
\ref{rem:warning_exceptional_explained} for more details.  However, if
$k$ is algebraically closed, then all exceptional objects are
automatically $k$-exceptional, hence remain exceptional under any
extension of scalars.
\end{remark}

Exceptional collections provide examples of semiorthogonal
decompositions when $\cat{T}$ is the bounded derived category of a
smooth projective scheme.

\begin{proposition}[{\cite[Thm.~3.2]{bondal:representations}}]
Let $\{ V_1, \ldots, V_n \}$ be an exceptional collection on the
bounded derived category $\Db(X)$ of a smooth projective $k$-scheme
$X$.  Then there is a semiorthogonal decomposition
$$
\Db(X) = \langle V_1, \ldots, V_n, \cat{A} \rangle,
$$
where $\cat{A}$ is the full subcategory of objects $W$ such that
$\Hom_{\cat{T}}(W,V_i)=0$.  In particular, the sequence if full if and
only if $\cat{A}=0$.
\end{proposition}

Given an exceptional pair $\{ V_1, V_2 \}$ with $V_i$ being
$A_i$-exceptional, consider the admissible subcategories $\langle V_i
\rangle$, forming a semiorthogonal pair. We can hence perform right
and left mutations, which provide equivalent admissible
subcategories. 

Recall that mutations provide equivalent admissible subcategories and
flip the semiorthogonality condition.
It easily follows that the object $R_{V_2}(V_1)$ is $A_1$-exceptional, the
object $L_{V_1}(V_2)$ is $A_2$-exceptional, and the pairs $\{
L_{V_1}(V_2), V_1 \}$ and $\{V_2, R_{V_2}(V_1)\}$ are exceptional.
We call $R_{V_2}(V_1)$ the \linedef{right mutation} of $V_1$ through $V_2$ and
$L_{V_1}(V_2)$ the \linedef{left mutation} of $V_2$ through $V_1$.

In the case of $k$-exceptional objects, mutations can be explicitly
computed.

\begin{definition}[{\cite[\S3.4]{gorodentsev-moving}}]
Given a $k$-exceptional pair $\{ V_1, V_2 \}$ in $\cat{T}$, the
\emph{left mutation} of $V_2$ with respect to $V_1$ is the object
$L_{V_1}(V_2)$ defined by the distinguished triangle:
\begin{equation}
\label{eq:left-mutation}
\Hom_\cat{T}(V_1,V_2) \otimes V_1 \mapto{\,ev\,} V_2
\longrightarrow L_{V_1}(V_2),
\end{equation}
where $ev$ is the canonical evaluation morphism. The \linedef{right
mutation} of $V_1$ with respect to $V_2$ is the object $R_{V_2}(V_1)$
defined by the distinguished triangle:
$$
R_{V_2}(V_1) \longrightarrow V_1 \mapto{coev}
\Hom_\cat{T}(V_1,V_2) \otimes V_2,
$$
where $coev$ is the canonical coevaluation morphism.
\end{definition}

Given an exceptional collection $\{ V_1, \ldots, V_n \}$, one
can consider any exceptional pair $\{ V_i, V_{i+1} \}$ and perform
either right or left mutation to get a new exceptional collection.

Exceptional collections provide an algebraic description of admissible
subcategories of $\cat{T}$. Indeed, if $V$ is an $A$-exceptional
object in $\cat{T}$, the triangulated subcategory $\langle V \rangle
\subset \cat{T}$ is equivalent to $\Db(k,A)$.  The equivalence
$\Db(k,A) \to \langle V \rangle$ is obtained by sending the complex
$A$ concentrated in degree $0$ to $V$.
Moreover, as shown by Bondal \cite{bondal:representations} and Bondal--Kapranov
\cite{bondal_kapranov:enhanced}, full exceptional collections give an
algebraic description of a triangulated category. 

\begin{proposition}[{\cite[Thm.~6.2]{bondal:representations}}]
\label{prop:exc-coll-and-dg-algebras}
Suppose that $\cat{T}$ is the bounded derived category of either a
smooth projective $k$-scheme or a $k$-linear abelian category with
enough injective objects.  Let $\{ V_1, \ldots, V_n \}$ be a full
strong exceptional collection on $\cat{T}$, and consider the object $V
= \bigoplus_{i=1}^n V_i$ and the $k$-algebra $A = \End_{\cat{T}}(V)$.
Then $\RHom_{\cat{T}}(V,-) : \cat{T} \to \Db(k,A)$ is a
$k$-linear equivalence.
\end{proposition}

\begin{remark}
The assumption on the category $\cat{T}$ and on the strongness of the
exceptional sequence may seem rather restrictive, and both find a
natural solution when triangulated categories are enriched with a
dg-category structure. The first assumption can be indeed replaced by
considering a dg-enhancement of $\cat{T}$ (see
\cite[Thm.~1]{bondal_kapranov:enhanced}).  When the exceptional
collection is not strong, dg-algebras are required. See
\S\ref{subs:dg-enhancement} for details.
\end{remark}

\begin{example}
The full strong $k$-exceptional collection $\{\ko, \ko(1), \dotsc,
\ko(n)\}$ on $\Db(\PP^n_k)$ is due to Beilinson \cite{beilinson},
\cite{beilinson2} and Bern\v{s}te{\u\i}n--Gelfand--Gelfand \cite{BGG}.
In this case $A = \End\bigl(\bigoplus_{i=0}^n \ko(i)\bigr)$ is
isomorphic to the path algebra of the Beilinson quiver with $n+1$
vertices, see \cite[Ex.~6.4]{bondal:representations}.  We remark that
the semisimplification of $A$ is the algebra $k^{n+1}$.
\end{example}

Proposition \ref{prop:exc-coll-and-dg-algebras} provides an
alternative approach to exceptional collections and semiorthogonal
decompositions by considering \linedef{tilting objects}. In this
paper, we will not use this approach, but we find the language
convenient at times, especially in relation to issues of Galois
descent.

\begin{definition}\label{def:tilting-bundle}
Given a strong exceptional collection $\{V_1, \ldots, V_m \}$ of a
$k$-linear triangulated category $\cat{T}$, the object $V =
\bigoplus_{i=1}^m V_i$ is called a \linedef{tilting object} for the
subcategory $\langle V_1, \ldots, V_m \rangle$.  Recall that
Proposition~\ref{prop:exc-coll-and-dg-algebras} implies that $\langle
V_1, \ldots, V_m \rangle \simeq \Db(k,\End_{\cat{T}}(V))$.  If the
$V_i$ are vector bundles on a smooth projective variety $X$, we call
$V$ a \linedef{tilting bundle}.
\end{definition}

The existence of a tilting bundle also yields, analogously to
Proposition \ref{prop:exc-coll-and-dg-algebras}, a presentation in
algebraic $K$-theory.

\begin{proposition}
Let $\{V_1,\dotsc,V_n\}$ be a full strong exceptional collection of vector
bundles on a smooth projective $k$-variety $X$ and let $V =
\bigoplus_{i=1}^n V_i$ be the associated tilting bundle.  Suppose that
$V_i$ is $A_i$-exceptional. Then $\Hom_X(V,-) : K_i(X) \to K_i(A_1
\times \dotsm \times A_n)$ is an isomorphism for all $i \geq 0$.
\end{proposition}
\begin{proof}
Letting $A = \End_{\Db(X)}(V)$.  Since the exceptional objects are
vector bundles, $\Hom_X(V,-)$ defines a morphism from the category of
vector bundles on $X$ to the category of $A$-modules.  Proposition
\ref{prop:exc-coll-and-dg-algebras} implies that $\RHom_{\Db(X)}(V,-)
: \Db(X) \to \Db(k,A)$ is an equivalence.  
Hence we can apply \cite[Thm.~1.9.8,~\S3]{thomason_trobaugh}, implying that
$\Hom_X(V,-) : K_i(X) \to K_i(A)$ is an isomorphism for all $i \geq
0$.  But $A$ has a nilpotent ideal $I$ so that $A/I \isom A_1 \times
\dotsm \times A_n$, so the result follows from the fact that $K_i(A) \isom K_i(A/I)$.
\end{proof}

\begin{remark}
One can prove the same result for any
additive invariant of $\Db(X)$ via noncommutative motives and
Tabuada's universal additive functor (see \cite{tabuada-universal}).
\end{remark}

A special case of an exceptional pair is a \linedef{completely
orthogonal} pair $\{V_1, V_2 \}$, 
i.e., an exceptional pair such that $\{ V_2, V_1 \}$ is also
exceptional.  In this case, $R_{V_2}(V_1)=V_1$ and $L_{V_1}(V_2)=V_2$.

\begin{definition}[\cite{karpov-nogin}, 1.5]
An \linedef{exceptional block} in a $k$-linear triangulated category
$\cat{T}$ is an exceptional collection $\{ V_1, \ldots, V_n \}$ such
that $\Ext^r_{\cat{T}}(V_i,V_j)=0$ for every $r$ whenever $i \neq j$.
Equivalently, every pair of objects in the collection is completely
orthogonal.  By abuse of notation, we denote by $\cat{E}$ the
exceptional block as well as the subcategory that it generates.
\end{definition}

If $\cat{E}$ is an exceptional block, then
$\End_\cat{T}(\bigoplus_{i=1}^n V_i)$ is isomorphic to the $k$-algebra $A_1 \times
\dotsm \times A_n$, where $V_i$ is $A_i$-exceptional.  Proposition
\ref{prop:exc-coll-and-dg-algebras} then yields a $k$-equivalence $\cat{E} \equi
\Db(A_1 \times \dotsm \times A_n)$.

Moreover, given an exceptional block, any internal mutation acts by
simply permuting the exceptional objects.  Given an exceptional
collection $\{V_1, \ldots, V_n, W_1, \ldots, W_m \}$ consisting of two
blocks $\cat{E}$ and $\cat{F}$, the left mutation
$L_{\cat{E}}(\cat{F})$ and the right mutation $R_{\cat{F}}(\cat{E})$
are obtained by mutating all the objects of one
block to the other side of all the objects of the other block, or,
equivalently, as mutations of semiorthogonal admissible subcategories.

\subsection{dg-enhancements}
\label{subs:dg-enhancement}

Recall that a $k$-linear \linedef{dg-category} $\dgcat{A}$ is a
category enriched over dg-complexes of $k$-vector spaces, that is, for
any pair of objects $x,y$ in $\ka$, the morphism set
$\Hom_{\dgcat{A}}(x,y)$ has a functorial structure of a differential
graded complex of vector spaces. The category $H^0(\dgcat{A})$ has the
same objects as $\dgcat{A}$, and $\Hom_{H^0(\dgcat{A})}(x,y) =
H^0(\Hom_{\dgcat{A}}(x,y))$, in particular $H^0(\dgcat{A})$ is triangulated. We will only consider
\linedef{pretriangulated} dg-categories (see \cite[\S4.5]{keller:ICM}). A
semiorthogonal decomposition of $\dgcat{A}= \langle \dgcat{B}_1, \ldots, \dgcat{B}_n
\rangle$ is a set of pretriangulated full subcategories such that
$\langle H^0(\dgcat{B}_1), \ldots, H^0(\dgcat{B}_n) \rangle$ is a semiorthogonal
decomposition of $H^0(\dgcat{A})$.  Consult \cite{keller:ICM} for an
introduction to dg-categories.

Let $X$ and $Y$ be a smooth proper $k$-schemes, and recall that a
\linedef{Fourier--Mukai functor} of kernel $P \in \Db(X \times Y)$ is
the exact functor $\Phi: \Db(X) \to \Db(Y)$ given by the formula
$\Phi_P(-)=({\mathbf{R}}p_{Y})_* (p_X^* (-) \otimes^{\mathbf L} P)$ (see, e.g., \cite[\S5]{Huybook}).

Considering the bounded derived category $\Db(X)$ as a triangulated
$k$-linear category gives rise to some functorial problems. One of
them has a geometric nature:\ given an exact functor $\Phi: \Db(X) \to
\Db(Y)$ it is not known whether it is isomorphic to a Fourier--Mukai
functor, except in some special cases, see
\cite{canonaco_stellari:survey}.  For example, suppose that
$\cat{A}_X$ and $\cat{A}_Y$ are admissible triangulated categories of
$\Db(X)$ and $\Db(Y)$, respectively, and suppose that there is a
triangulated equivalence $\cat{A}_X \simeq \cat{A}_Y$. The composition
functor
$$
\Phi: \Db(X) \twoheadrightarrow \cat{A}_X \simeq \cat{A}_Y
\hookrightarrow \Db(Y)
$$
is conjectured to be a Fourier--Mukai functor
(\cite[Conj.~3.7]{kuznetsov:hpd}).

As proved by Lunts and Orlov \cite{Lunts-Orlov}, there is a unique
functorial enhancement $\Dbdg(X)$ for $\Db(X)$, that is a
dg-category such that $H^0(\Dbdg(X)) = \Db(X)$. Given $X$ and $Y$
as before, a functor $\Phi: \Db(X) \to \Db(Y)$ is of Fourier--Mukai
type if and only if there is a functor $\varPhi: \Dbdg(X)\to
\Dbdg(Y)$ such that $H^0(\varPhi)=\Phi$, see
\cite[Prop.~9.4]{berna-tabuada-jacobians} for example, or
\cite{lunts-schnu} for a more general statement.

On the other hand, Kuznetsov \cite{kuznetbasechange} has shown that,
for any admissible subcategory $\cat{A}_X \subset \Db(X)$, the
projection functor (i.e., the right adjoint to the embedding functor)
is a Fourier--Mukai functor. Hence, the choice of such a functor
endows $\cat{A}_X$ with a dg-structure, which is in principle not
unique, but depends on the choice of projection. Hence, when
dealing with a semiorthogonal decomposition
$$
\Db(X) = \langle \cat{A}_1, \ldots, \cat{A}_n \rangle,
$$
one always has a decomposition
$$
\Dbdg(X) = \langle \dgcat{A}_1, \ldots, \dgcat{A}_n \rangle,
$$
where the dg-category $\Dbdg(X)$ is unique, though the dg-enhancements
$\dgcat{A}_i$ of $\cat{A}_i$ may not be.

Finally, we recall that if there exists a smooth projective scheme $Z$
and a Brauer class $\alpha$ in $\Br(Z)$ such that $\cat{A}_i \simeq
\Db(Z,\alpha)$, then the embedding functor $\Phi: \Db(Z,\alpha) \to
\Db(X)$ is a Fourier--Mukai functor
\cite{canonaco_stellari:twisted_fourier_mukai}, and hence the
dg-structure $\Dbdg(Z,\alpha)$ is unique. Other specific cases are
provided by Homological Projective Duality \cite{kuznetsov:hpd}.

\subsection{Categorical representability}
\label{subsec:cat_rep}

Using semiorthogonal decompositions, one can define a notion of
\linedef{categorical representability} for a triangulated category. In
the case of smooth projective varieties, this is inspired by the
classical notions of representability of cycles, see
\cite{bolognesi_bernardara:representability}.

\begin{definition}[\cite{bolognesi_bernardara:representability}]
\label{def-rep-for-cat}
A $k$-linear triangulated category $\cat{T}$ is \linedef{representable
in dimension $m$} if it admits a semiorthogonal decomposition
$$
\cat{T} = \langle \cat{A}_1, \ldots, \cat{A}_r \rangle,
$$
and for each $i=1,\ldots,r$ there exists a smooth projective connected
$k$-variety $Y_i$ with $\dim Y_i \leq m$, such that $\cat{A}_i$ is
equivalent to an admissible subcategory of $\Db(Y_i)$.
\end{definition}

\begin{definition}[\cite{bolognesi_bernardara:representability}]
\label{def-cat-rep}
Let $X$ be a projective $k$-variety of dimension $n$. We say that $X$
is \linedef{categorically representable} in dimension $m$ (or
equivalently in codimension $n-m$) if there exists a categorical
resolution of singularities of $\Db(X)$ that is representable in
dimension $m$.
\end{definition}

The following explains why categorical representability in codimension
2 is conjecturally related to birational geometry in the spirit of
Question~\ref{mainquestion-repre}.

\begin{lemma}
\label{lem:blow-up_catrep2}
Let $X \to Y$ be the blow-up of a smooth projective $k$-variety along
a smooth center.  If $Y$ is categorically representable in codimension
2 then so is $X$.
\end{lemma}
\begin{proof}
Denoting by $Z \subset Y$ the center of the blow-up $f : X \to Y$,
which has codimension at least 2 in $Y$,
Orlov's blow-up formula \cite{orlovprojbund} gives a semiorthogonal
decomposition $\Db(X) = \langle f^*\Db(Y), \Db(Z)\rangle$. As $X$ and $Y$ have the
same dimension, if $Y$ is categorically representable in
codimension 2, then so is $X$.
\end{proof}

An additive category $\cat{T}$ is \linedef{indecomposable} if for any
product decomposition $\cat{T} \equi \cat{T}_1 \times \cat{T}_2$ into
additive categories, we have that $\cat{T} \equi \cat{T}_1$ or
$\cat{T} \equi \cat{T}_2$.  Equivalently, $\cat{T}$ has no nontrivial
completely orthogonal decomposition.  Remark that if $X$ is a
$k$-scheme then $\Db(X)$ is indecomposable if and only if $X$ is
connected (see \cite[Ex.~3.2]{bridg-equiv-and-FM}).
More is known if $X$ is the spectrum
of a field or a product of fields.

\begin{lemma}
\label{lem:indecomposable}
Let $K$ be a $k$-algebra.
\begin{enumerate}
\item\label{lem:indecomposable_1} If $K$ is a field and $\cat{A}$ is a
nonzero admissible $k$-linear triangulated subcategory of $\Db(k,K)$,
then $\cat{A} = \Db(k,K)$.

\item\label{lem:indecomposable_2} If $K \isom K_1 \times \dotsm \times
K_n$ is a product of field extensions of $k$ and $\cat{A}$ is a
nonzero admissible indecomposable $k$-linear triangulated subcategory
of $\Db(k,K)$, then $\cat{A} \equi \Db(k,K_i)$ for some $i=1, \dotsc, n$.

\item\label{lem:indecomposable_3} If $K \isom K_1 \times \dotsm \times
K_n$ is a product of field extensions of $k$ and $\cat{A}$ is a
nonzero admissible $k$-linear triangulated subcategory
of $\Db(k,K)$, then $\cat{A} \equi \prod_{j \in I} \Db(k,K_{j})$ for
some subset $I \subset \{1, \dotsc, n\}$.
\end{enumerate}
\end{lemma}
\begin{proof}
For \ref{lem:indecomposable_1}, it suffices to note that any nonzero
object of $\Db(k,K)$ is a generator.  For \ref{lem:indecomposable_2},
with respect to the product decomposition $\Db(k,K) \equi \Db(k,K_1)
\times \dotsm \times \Db(k,K_n)$, consider the projections $x_i$ to
$\Db(k,K_i)$ of an object $x$ of $\cat{A}$.  Any nonzero $x_i$ will
generate the respective subcategory $\Db(k,K_i)$.  

Given a nonzero
object $x$ in $\cat{A}$, there exist $1 \leq i \leq n$ such that $x_i
\neq 0$.  If there exists another object $y$ in $\cat{A}$ with $y_j
\neq 0$ for $j \neq i$, then the objects $x_i$ and $y_j$ will generate
completely orthogonal nontrivial subcategories of $\cat{A}$.  Since
$\cat{A}$ is indecomposable, this is impossible.  Hence for every
object $x$ in $A$, we have that $x_j=0$ for every $j \neq i$.  In
particular, $\cat{A} \subset \Db(k,K_i)$, hence they are equal by
\ref{lem:indecomposable_1}.  For \ref{lem:indecomposable_3}, we apply
\ref{lem:indecomposable_2} to the indecomposable components of
$\cat{A}$.
\end{proof}

Next, we need the following affine case of a conjecture of
C\v{a}ld\v{a}raru~\cite[Conj.~4.1]{caldararu:elliptic_threefolds}.
The simple proof below is taken from the unpublished manuscript
\cite{auel_ojanguren_parimala_suresh:semilinear_morita} as part of a
proof of C\v{a}ld\v{a}raru's conjecture in the relative case.  A proof
of the conjecture (using different techniques) was recently obtained
by Antieau~\cite{antieau-caldararu}, with the case of arbitrary
(possibly non-torsion) classes over general spaces handled by
\cite{calabrese_groechenig}, after progress by
\cite{canonaco_stellari:twisted_fourier_mukai} and
\cite{perego:coherent_twisted_sheaves}.

Recall that a \linedef{central simple $k$-algebra} $A$ is a simple
$k$-algebra whose center is $k$. More generally, if $X$ is a scheme,
an \linedef{Azumaya algebra} $A$ over $X$ is a locally free
$\ko_X$-algebra of finite rank such that $A \tensor_R A\op$ is
isomorphic to the matrix algebra $\EEnd(A)$.  Azumaya algebras $A$ and
$B$ are \linedef{Brauer equivalent} if there exist locally free
$\ko_X$-modules $P$ and $Q$ such that $A \tensor \EEnd(P) \isom B
\tensor \EEnd(Q)$, and they are \linedef{Morita equivalent} over $X$
if their stacks of coherent modules over $X$ are equivalent.  Then
Brauer equivalence coincides with Morita equivalence over $X$, and the
group of equivalence classes under tensor product is the
\linedef{Brauer group} $\Br(X)$.  For an Azumaya algebra $A$ over $X$,
we write $A^{n}$ for $A^{\otimes n}$ for $n \geq 0$ and
$(A\op)^{\otimes -n}$ for $n \leq 0$.

\begin{theorem}
\label{thm:Caldararu_conj}
Let $R$ be a noetherian commutative ring, $U$ and $V$ be $R$-algebras,
and $A$ and $B$ be Azumaya algebras over $U$ and $V$, respectively.
Then $A$ and $B$ are Morita $R$-equivalent if and only if there exists
an $R$-algebra isomorphism $\sig : U \to V$ such that $A$ and
$\sig\pullback B$ are Brauer equivalent over $U$.
\end{theorem}
\begin{proof}
First suppose that $A$ and $\sig\pullback B$ are Brauer equivalent,
where $\sig : U \to V$ is an $R$-algebra isomorphism.  Then by Morita
theory for Azumaya algebras (cf.\
\cite[\S19.5]{kashiwara_schapira:categories_sheaves}),
there is a Morita $U$-equivalence $\Coh(U,A) \to \Coh(U,\sig\pullback
B)$, which by restriction of scalars, is a Morita $R$-equivalence.
Also, $(\sig\inv)\pullback : \Coh(U,\sig\pullback B) \to \Coh(V,B)$ is
an $R$-equivalence.  Composing these yields the required Morita
$R$-equivalence.

Now suppose that $F : \Coh(U,A) \to \Coh(V,B)$ is a Morita
$R$-equivalence.  Since $F$ is essentially surjective, we can choose
$P \in \Coh(U,A)$ such that $F(P) \isom B\op$ as $B$-modules.  Since
$F$ is fully faithful, any choice of isomorphism $\theta : F(P)
\isomto B\op$ defines a left $B$-module structure on $P$.  In this
way, $P$ has a $B$-$A$-module structure with commuting $R$-structure.
In particular, for any $A$-module $P'$, $\HHom_{A}(P,P')$ has a
natural right $B$-module structure via precomposition.

The equivalence $F$ being $R$-linear yields an induced $R$-algebra
isomorphism
$$
\psi : \EEnd_{A}(P) = \HHom_{A}(P,P) \isom \HHom_{B}(F(P),F(P)) \isom \EEnd_{B}(B\op) = B.
$$ 
In particular, $\psi$ restricts to an $R$-algebra isomorphism $\sig :
U \to V$ of the centers, hence becomes a $U$-module isomorphism
$\EEnd_{A}(P) \to \sig\pullback B = (\sig\pullback B)\op$, so that $A$
is Brauer equivalent to $\sig\pullback B$.
\end{proof}

\begin{corollary}
\label{cor:Caldararu_conj}
Let $A$ and $B$ be simple $k$-algebras with respective centers $K$ and
$L$.  Then $\Db(K/k,A)$ and $\Db(L/k,B)$ are ($k$-linearly) equivalent
if and only if there exists some $k$-automorphism $\sig : K \to L$
such that $A$ and $\sig\pullback B$ are Brauer equivalent over $K$.
\end{corollary}
\begin{proof}
By \cite[Cor.~2.7]{yekutieli:dualizing_morita_derived}, if $A$ or
$B$ is either commutative or local artinian (not necessarily
commutative), then they are derived $k$-equivalent if and only if they
are Morita $k$-equivalence.  Then we apply
Theorem~\ref{thm:Caldararu_conj}.
\end{proof}

We remark that Corollary \ref{cor:Caldararu_conj} is also a
consequence of the results of \cite{antieau:affine_caldararu}. 

\begin{lemma}
\label{lem:0-dim=etale-algebra}
A $k$-linear triangulated category $\cat{T}$ is representable in
dimension $0$ if and only if there exists a semiorthogonal
decomposition
$$
\cat{T} = \langle \cat{A}_1, \ldots, \cat{A}_r \rangle,
$$
such that for each $i$, there is a $k$-linear equivalence $\cat{A}_i
\simeq \Db(K_i/k)$ for an \'etale $k$-algebra $K_i$.
\end{lemma}
\begin{proof}
The smooth $k$-varieties of dimension 0 are precisely
the spectra of \'etale $k$-algebras.  Hence the semiorthogonal
decomposition condition is certainly sufficient for the
representability of $\cat{T}$ in dimension 0.  On the other hand,
if $\cat{T}$ is representable in dimension 0, we have
that each $\cat{A}_i$ an admissible subcategory of the derived category of
an \'etale $k$-algebra.  By
Lemma~\ref{lem:indecomposable}\ref{lem:indecomposable_3}, we have that
$\cat{A}_i$ is thus itself such a category.
\end{proof}

\section{Descent for semiorthogonal decompositions}
\label{sec:descent_sodec}

Given a smooth projective variety $X$ over a field $k$, and fixing a
separable closure $\ksep$ of $k$, we are interested in comparing
$k$-linear semiorthogonal decompositions of $\Db(X)$ and
$\ksep$-linear semiorthogonal decompositions of $\Db(X_{\ksep})$.  The
general question of how derived categories and semiorthogonal
decompositions behave under base field extension has started to be
addressed by several authors \cite[\S2]{an-au-ga-za},
\cite{antieau:twists}, \cite{antieau-gepner},
\cite{elagin:descent-semi}, \cite{kuznetbasechange},
\cite{sosna:scalar}, \cite{toen:derived-azumaya}.  Galois descent does
not generally hold for objects in a triangulated category, due to the
fact that cones are only defined up to quasi-isomorphism.  However,
for a linearly reductive algebraic group $G$ acting on a variety $X$,
a general theory of descent of semiorthogonal decompositions has been
developed by Elagin
\cite{elagin:descent-semi,elagin:descent-stacks,elagin:semi-equi}.  In
the case where $K/k$ is a finite $G$-Galois extension and we consider
$G$ acting on $X_K$ via $k$-automorphisms, this theory works as long
as the characteristic of $k$ does not divide the order of $G$.

We are then faced with two main questions. First, given a
$\ksep$-linear triangulated category $\overline{\cat{T}}$, what are
all the $k$-linear triangulated categories $\cat{T}$ such that
$\cat{T}_{\ksep}$ is equivalent to $\cat{T}$? This first question is
very challenging (see \cite{antieau:twists} for some examples), and we
usually restrict ourselves to consider $\cat{T}_{\ksep}$ in a
restricted class of categories (e.g., those generated by a strong
exceptional collection).  Notice that we are actually dealing with
triangulated categories which admit a canonical dg-enhancement, and
this should be the natural structure to consider.  Second, given a set
of admissible subcategories in $\Db(X)$, determine how this can be
related to the semiorthogonal decomposition of $\Db(X_{\ksep})$.

In our geometric applications, a central r\^ole is played
by descent of exceptional blocks and vector bundles.

\subsection{Scalar extension of triangulated categories}
\label{subsec:Scalar_extension_of_triangulated_categories}

Let $X$ be a $k$-scheme and $V$ an $\OO_X$-module.  For a ring
extension $K/k$ we write $X_K = X \times_{\Spec\, k} \Spec\, K$ and
$f_K : X_K \to X$ for the projection and $V_K = f_K\pullback V$ for
the \linedef{extension} of $V$ to $K$.  If $K$ is a finite $k$-algebra
and $F$ is an $\OO_{X_K}$-module, we denote by $\tr_{K/k}F =
(f_K)\pushforward F$ the \linedef{trace} of $F$ from $K$ down to $k$.
We similarly define the trace of an object of $\Db(X_K)$.  If $F$ is
locally free of rank $r$ on $X_K$ then $\tr_{K/k}F$ is locally free of
rank $r [K:k]$ on $X$.  We denote by $G_k = \Gal(\ksep/k)$ the
absolute Galois group of $k$.

If $\cat{T}$ is a $k$-linear triangulated category with a
dg-enhancement and $K/k$ is a field extension, then denote by
$\cat{T}_K$ the $K$-linear extension of scalars category defined in
\cite{sosna:scalar}.  This category is a thick dg-enhanced $K$-linear
triangulated category.  As expected, there is a $K$-linear equivalence
$\Db(X)_K \simeq \Db(X_K)$.  If $K/k$ is a Galois extension, then the
Galois group acts by $k$-linear equivalences on $\cat{T}_K$.  If $K/k$
is a finite $G$-Galois extension then $G$ acts on $X_K$ as a
$k$-scheme and there is a $k$-linear equivalence $\Db(X) \equi
\Db_G(X_K)$ with the bounded derived category of $G$-equivariant
coherent sheaves on $X_K$ considered as a $k$-scheme.

Given an orthogonal pair of admissible
subcategories $\{\cat{A}, \cat{B}\}$ in $\Db(X)$, we can perform right
and left mutations. These operations commute with base change, i.e.,
${L_{\cat{A}}(\cat{B})}_K = L_{\cat{A}_K}(\cat{B}_K)$, for any finite
extension $K$ of $k$, and similarly for the right mutation.

\begin{remark}
\label{rem:warning_exceptional_explained}
Assume that $[K:k]$ is finite and that $X$ is a smooth projective
$k$-variety. For an object $V$ in $\Db(X)$, the scalar extension $V_K$
in $\Db(X)_K=\Db(X_K)$ satisfies $\End_{\cat{T_K}}(V_K)
= \End_{\cat{T}}(V)\tensor_k K$, see \cite[\S0.5.4.9,~1.9.3.3]{EGA1}.
If $A$ is a division $k$-algebra, then $A \tensor_k K$ may not be.
Hence if $V$ is an exceptional object in $\Db(X)$, then $V_K$ may fail
to be an exceptional object in $\Db(X_K)$.  However, if $V$ is
$k$-exceptional, then $V_K$ is always $K$-exceptional.
\end{remark}

\begin{lemma}
\label{lem:descent_admissible}
Let $\cat{T}$ be an admissible $k$-linear subcategory of $\Db(X)$ for
a smooth projective $k$-variety $X$.  Let $K$ be any extension of $k$.
Then $\cat{T}_K$ is admissible in $\Db(X_K)$.
\end{lemma}
\begin{proof}
By Kuznetsov \cite{kuznetbasechange}, the inclusion functor of the
admissible subcategory $\cat{T} \to \Db(X)$ has an adjoint
$\rho:\Db(X) \to \cat{T}$ that is of Fourier--Mukai type with kernel
$P$ an object in $\Db(X \times X)$.  Taking $P \otimes_k K$ in
$\Db(X_K \times X_K)$ as the kernel of a Fourier--Mukai functor, we
obtain an adjoint for $\cat{T}_K \to \Db(X_K)$.
\end{proof}

Based on base change results for Fourier--Mukai
functor due to Orlov (see \cite{orlovequivabel}), the following statement
was proved in \cite[Lemma~2.9]{auel-berna-bolo} (see also
\cite[Prop.~2.1]{an-au-ga-za}).

\begin{lemma}\label{lemma-base-change-of-semiorth}
Let $X$ be a smooth projective variety over $k$, and $K$ a finite
extension of $k$.  Suppose that $\cat{T}_1, \ldots, \cat{T}_n$ are
admissible subcategories of $\Db(X)$ such that $\Db(X_K) = \langle
\cat{T}_{1 K}, \ldots, \cat{T}_{n K} \rangle$. Then $\Db(X) = \langle
\cat{T}_1, \ldots, \cat{T}_n \rangle$.
\end{lemma}

\subsection{Classical descent theory for vector bundles}
\label{subsec:classical_descent}

For a vector bundle $V$ on a scheme $X$, denote by $A(V)
= \End(V)/\rad(\End(V))$ and $Z(V)$ the center of $A(V)$.  Assuming
that $X$ is a proper $k$-scheme, then vector bundles on $X$ enjoy a
Krull--Schmidt decomposition, hence $A(V)$ is a semisimple
$k$-algebra.  If $V = \bigoplus_{i=1}^{m} V_i^{\oplus d_i}$ is the
Krull--Schmidt decomposition of $V$, then $A(V)$ is the product of the
algebras $M_{d_i}(A(V_i))$. In particular, $V$ is indecomposable on
$X$ if and only if $A(V)$ is a division algebra over $k$.

If $K/k$ is a field extension and $E$ is a vector bundle, then we have
$\End(V_K) = \End(V)_K$ (cf.\
Remark~\ref{rem:warning_exceptional_explained}), and hence $A(V_K) =
A(V)_K$ if $K/k$ is separable.

\begin{lemma}[{\cite[Lemma~1.1]{arason_elman_jacob:indecomposable}}]
Let $V$ be an indecomposable vector bundle on $X$.  Let $K/k$ be a
normal extension containing $Z(V)$ and splitting $A(V)$.  Write $m =
[Z(V):k]_{\mathrm{sep}}$ and $d = \deg_{Z(V)}(A(V))$.  Then 
$V_{K}$ has a Krull--Schmidt decomposition of the form
$
V_{K} = \bigoplus_{i=1}^{m} V_i^{\oplus d}
$ 
where $V_i$ are indecomposable over $X_K$ and $A(V_i) = K$.
\end{lemma}

Let $W$ be an 
indecomposable
vector bundle on $X_{\ksep}$.  A vector
bundle $V$ on $X$ is \linedef{pure (of type $W$)} if $V_{\ksep}$ is a
direct sum of vector bundles isomorphic to $W$.

\begin{proposition}[{\cite[Prop.~3.4]{arason_elman_jacob:indecomposable}}]
\label{prop:exists_pure}
Let $X$ be a proper variety over $k$.  If $W$ is a $G_k$-invariant
indecomposable vector bundle on $X_{\ksep}$ then there exists a pure
indecomposable vector bundle $V$ on $X$, unique up to isomorphism, of
type $W$.
\end{proposition}

\begin{definition}
Given a $G_k$-invariant indecomposable vector bundle $W$ on
$X_{\ksep}$, let $V$ be the vector bundle on $X$ given in
Proposition~\ref{prop:exists_pure}.  Define $\alpha(W) \in \Br(k)$ to
be the Brauer class of $A(V)$.  In fact, $\alpha(W)$ is the Brauer
class of $A(V')$ for any pure vector bundle $V'$ on $X$ of type $W$.
\end{definition}

\begin{remark}
\label{rem:Leray}
Recall the Brauer obstruction for Galois invariant line bundles on a
smooth proper geometrically integral variety $X$ over $k$, see
\cite[\S2]{stix:period-index}. The sequence of low degree terms of the
Leray spectral sequence for $\Gm$ associated to the structural
morphism $X \to \Spec\, k$ is
\begin{equation}
\label{eq:Leray}
0 \to \Pic(X) \to \Pic(X_{\ksep})^{G_k} \mapto{d} \Br(k) \to \Br(X).
\end{equation} 
The differential $d$ is called the Brauer obstruction for a Galois
invariant line bundle; a class in $\Pic(X_{\ksep})^{G_k}$ descends to
a line bundle on $X$ if and only if it has trivial Brauer obstruction.
The cokernel of $\Pic(X) \to \Pic(X_{\ksep})^{G_k}$ is torsion since
the Brauer group is.  We also recall that if $A$ is a central simple
algebra with class $d(L)$, for a Galois invariant line bundle $L$, and
$Y$ is the Severi--Brauer variety associated to $A$, then up to
tensoring by elements of $\Pic(X)$, we can assume that $L$ corresponds
to a morphism $X \to Y$, see \cite[Prop.~8]{stix:period-index}.
\end{remark}

\begin{corollary}[{\cite[Prop.~3.15]{arason_elman_jacob:indecomposable}}]
\label{cor:the-chern-class}
Let $r$ be the index of the Brauer class $\alpha(W) \in \Br(k)$. Then
$V_{\ksep} \isom W^{\oplus r}$, and hence $c_1(V_{\ksep}) = r c_1(W)$
is in the image of $\Pic(X) \to \Pic(X_{\ksep})^{G_k}$.  In particular,
$W$ descends to $X$ if and only if $\alpha(W)$ is trivial
\end{corollary}

\begin{theorem}[{\cite[Thm.~1.8]{arason_elman_jacob:indecomposable}}]
\label{thm:trace_from_maximal_subfield}
Let $X$ be a proper variety over $k$ and $V$ an indecomposable vector
bundle on $X$.  If $K$ is a maximal subfield of $A(V)$ then there is
an absolutely indecomposable vector bundle $W$ on $X_K$ such that $V
\isom \tr_{K/k}W$.
\end{theorem}

\begin{definition}
\label{def:reduced_part}
Let $V$ be a vector bundle on a proper smooth scheme $X$ over $k$.  We
write $V=V_1^{n_1} \oplus \dotsm \oplus V_r^{n_r}$ the Krull--Schmidt
decomposition, with $V_1, \dotsc, V_r$ indecomposable.  The
\linedef{reduced part} of $V$ is defined to be $V^{\min} = V_1 \oplus
\dotsm \oplus V_r$.  We remark that $V$ and $V^{\min}$ generate the same
thick subcategory of $\Db(X)$.
\end{definition}

\subsection{Descent of exceptional blocks}
\label{subsec:descent_blocks}

If $\cat{T}_{\ksep}$ admits a full exceptional block, we wish to
classify all the $k$-linear categories $\cat{T}$ base-changing to
$\cat{T}_{\ksep}$.  In particular, if $\cat{T}_{\ksep}$ is generated
by an exceptional block of vector bundles (i.e., it admits a tilting
bundle) in $\Db(X_{\ksep})$ for some smooth proper scheme $X$, the
descent of $\cat{T}_{\ksep}$ to $\cat{T}$ inside $\Db(X)$ is described
by descent of the tilting bundle.

When $\cat{T}_{\ksep}$ is generated by a $\ksep$-exceptional object,
the descent question has been studied by To\"en
\cite{toen:derived-azumaya}. Recall (see Corollary~\ref{cor:Caldararu_conj}) that
Brauer equivalence between Azumaya algebras over a ring $R$ is the
same as Morita $R$-equivalence, and that Morita $k$-equivalence
between local (possibly noncommutative) $k$-algebras is the same as
derived $k$-equivalence by
\cite[Cor.~2.7]{yekutieli:dualizing_morita_derived}.

\begin{theorem}[To\"en {\cite[Cor.~2.15]{toen:derived-azumaya}}]
\label{thm:toen}
If $\cat{T}$ is a $k$-linear triangulated category such that
$\cat{T}_{\ksep} = \Db(\ksep)$, then there exists a central simple
algebra $A$ over $k$ such that $\cat{T} = \Db(k,A)$. In particular,
$\cat{T}$ is generated by a $k$-exceptional object if and only if $A$
is trivial in the Brauer group.
\end{theorem}
\begin{proof}
Since $\Db(\ksep)$ has compact generator, then so does $\cat{T}$ by
\cite[Prop.~4.12]{toen:derived-azumaya}. Then the dg-category $\cat{T}$
is equivalent to $\Db(k,A)$ for some dg-algebra $A$ over $k$.
As $A_{\ksep}$ is Morita equivalent to $\ksep$ and $\ksep/k$ is
faithfully flat, then by \cite[Prop.~2.3]{toen:derived-azumaya}, $A$
is a derived Azumaya algebra over $k$.  By
\cite[Prop.~2.12]{toen:derived-azumaya}, $A$ is Morita equivalent to
an Azumaya algebra over $k$.
\end{proof}

Recall that $K$ is a finite \'etale $k$-algebra if and only if $K
\tensor_k \ksep \isom \ksep \times \dotsm \times \ksep$.
Equivalently, $K \isom l_1 \times \dotsm \times l_m$ where each
$l_i/k$ is a finite separable field extension.  The $k$-dimension of
$K$ is called the \linedef{degree} of $K$ over $k$.  An Azumaya
algebra $A$ over $K \isom l_1 \times \dotsm \times l_m$ is simply a
product $A \isom A_1 \times \dotsm \times A_m$ where each $A_i$ is a
central simple $l_i$-algebra.  A \linedef{separable algebra} of
dimension $n r^2$ over $k$ is an Azumaya algebra of degree $r$ over an
\'etale algebra $K$ over $k$ of dimension $n$.

The set of $k$-isomorphism classes of \'etale $k$-algebras of degree
$n$ is in bijection with the Galois cohomology set $H^1(k,S_n)$.  The
set of $k$-isomorphism classes of degree $n$ Azumaya algebras over an
\'etale $k$-algebra of degree $r$ is in bijection with the Galois
cohomology set $H^1(k,G)$, where $G$ is the group scheme defined by
the functor $G(R) = \Aut_R(M_r(R^n))$.  Note that $G$ is isomorphic to
the wreath product $\PGL_r \wr S_n = \PGL_r^n \rtimes S_n$ with $S_n$
acting via permutation of the factors.  Then there is a natural map
$H^1(k,G) \to H^1(k,S_n)$ given by taking the isomorphism class of the
center (this map is clearly surjective), whose fiber over the class of
an \'etale algebra $K$ of dimension $n$ over $k$ is given by the set
of isomorphism classes of Azumaya algebras over $K$ of degree $r$.

Finally, by the classical theory of central simple algebras over a
field, each central simple $k$-algebra $A$ of degree $n$ contains a
maximal \'etale $k$-subalgebra $K$ of degree $n$ splitting $A$, i.e.,
such that $A \tensor_k K \isom M_n(K)$.

\begin{proposition}\label{prop:descend-a-block}
Let $\cat{T}$ be a $k$-linear triangulated category such that
$\cat{T}_{\ksep}$ is $\ksep$-equivalent to $\Db(\ksep,(\ksep)^n)$.
Then there exists an \'etale algebra $K$ of degree $n$ over $k$, an
Azumaya algebra $A$ over $K$, and a $k$-linear equivalence $\cat{T}
\equi \Db(K/k,A)$.  In this case, $\cat{T}$ is an indecomposable
category if and only if $K$ is a field extension of $k$.
\end{proposition}

\begin{proof}
To\"en \cite[Cor.~3.12]{toen:derived-azumaya} shows that the derived
group stack of autoequivalences $\stack{Aut}_{dAz/k}(k)$ of the
trivial derived Azumaya $k$-algebra $k$ is equivalent to $\ZZ \times
K(\Gm,1)$, as any dg-autoequivalence of an \'etale $k$-algebra $R$ is
given by tensoring with some invertible perfect dg-$R$-module, which
by \cite[Thm.~8.15]{toen:derived-morita}, are all of the form $L[n]$
for an invertible $R$-module $L$ and $n \in \ZZ$.

As a generalization, we claim that the derived group stack of
autoequivalences $\stack{Aut}_{dg/k}(k^n)$ of the \'etale $k$-algebra $k^n$
(thought of as a dg-algebra over $k$), is equivalent to the wreath
product $(\ZZ \times K(\Gm,1)) \wr S_n$, thought of as $n \times n$
generalized permutation matrices filled with shifts of invertible
modules.  Indeed, for an \'etale $k$-algebra $R$, any dg-endofunctor
of $R^n$ can be viewed as an $n \times n$ matrix $M$ of perfect
dg-$R$-modules by \cite[Thm.~8.15]{toen:derived-morita}.  Considering
the $n\times n$ matrix $\overline{M}$ of ranks of the entries of $M$,
we see that $M$ is invertible implies that $\overline{M}$ is
invertible and its inverse also consists of nonnegative integers,
which is well-known to imply that $\overline{M}$ is a permutation
matrix (i.e., $\GL_n(\NN) = S_n$).  Hence $M$ must have exactly a
single nonzero module in each row and column, and this entry must have
rank 1, giving the desired form.

We are now interested in the stack $\stack{F} = \Forms_{dg/k}(k^n)$
associated to the prestack of dg-algebras \'etale locally Morita
equivalent to $k^n$, which is simply $K(\stack{Aut}_{dg/k}(k^n),1)$,
as in \cite[Cor.~3.12]{toen:derived-azumaya}.  The exact sequence of
group stacks
$$
1 \to (\ZZ \times K(\Gm,1))^n \to \stack{Aut}_{dg/k}(k^n) \to S_n \to 1
$$
gives rise to a fibration of stacks
$$
(K(\ZZ,1) \times K(\Gm,2))^n \to \stack{F} \to K(S_n,1).
$$
Hence to any dg-algebra $A$ \'etale locally Morita equivalent to
$k^n$, we have a class $\phi(A)$ in $\stack{F}(k)$ (the sections over
$k$ of the stack $\stack{F}$) and the association $A \mapsto \phi(A)$
is injective.  Hence $\stack{F}(k)$ fits into an exact sequence
$$
\Br(k)^n \to \stack{F}(k) \to H^1(k,S_n) \to 1,
$$
(using that $H^1(k,\ZZ)=0$) where the map $z : \stack{F}(k) \to
H^1(k,S_n)$ sends $\phi(A)$ to the isomorphism class of the center
$Z(A)$.  By a standard cohomological twisting argument, the fiber of
$z$ over an \'etale algebra $K$ of degree $n$ over $k$, is the image
of $\Br(K)$.  This gives the desired description.  The final claim
follows since for any field extension $K$ over $k$, the $k$-linear
category $\Db(K/k,A)$ is indecomposable; conversely, given any
decomposition $K=K_1 \times K_2$, there is an induced decomposition $A
= A_1 \times A_2$, with $A_i$ Azumaya over $K_i$, and then $\Db(K/k,A)
\equi \Db(K_1/k,A_1) \times \Db(K_2/k,A_2)$.
\end{proof}

Let us state explicitly the following corollary which will be
extensively used in our applications.

\begin{corollary}
\label{prop:single-exc-bundle-G-inv}
Let $X$ be a smooth projective $k$-variety, $K/k$ a Galois extension
with group $G$. Suppose that there is an admissible subcategory
$\cat{T}$ of $\Db(X)$ such that $\cat{T}_K=\langle V_1, \ldots, V_n
\rangle$ is an exceptional block of vector bundles on $X_K$. Then for
any object $A$ in $\cat{T}$ the Chern class $c_1(A)$ is a $\ZZ$-linear
combination of the $c_1(V_i)$.  In particular, there exists a
nonzero $G$-invariant $\ZZ$-linear combination of the $c_1(V_i)$.
\end{corollary}
\begin{proof}
Proposition \ref{prop:descend-a-block} implies that $\cat{T}$ has
cohomological dimension 0, so that every object is the direct sum of
its cohomologies, and this decomposition is clearly $G$-invariant. By
the thickness of $\cat{T}$, we can then suppose that $A$ is a sheaf.
Then $A_K$ is a direct sum of the vector bundles $V_i$, and is
nontrivial if $A \neq 0$.
\end{proof}

\medskip

Let us end this section by illustrating Lemma \ref{lemma-base-change-of-semiorth} by
examples known in the literature of descent of vector bundles.

\begin{example}
\label{ex:SB_decomposition}
Let $X$ be the Severi--Brauer variety associated to a central simple
algebra $A$ of degree $n+1$ over $k$, see Example \ref{exam:SB} for
definitions. 
Let $K/k$ be a field extension such that $X_K \simeq
\PP^n_K$. Thanks to Beilinson \cite{beilinson}, we have the following
semiorthogonal decomposition
\begin{equation}\label{semibeilinson}
\Db(\PP^n_K) = \langle \ko_{\PP^n_K}, \ko_{\PP^n_K}(1), \ldots, \ko_{\PP^n_K}(n) \rangle
\end{equation}
On the other hand, $\Db(X)$ admits the following semiorthogonal
decomposition\footnote{In \cite{bernardara:brauer_severi}, the full
and faithful embedding of the subcategory of $A\inv$-modules into
$\Db(X)$ is given by the choice of an ``\'etale local form'' $\kf$ of
$\ko_X(1)$; an $A$-module $M$ is sent to $M \otimes \kf$, which can be
endowed with the structure of an $\ko_X$-module. Notice that in \it
loc.cit. \rm $A\inv$ is the algebra obstructing the descent of
$\ko(1)$, so that our notations are opposite to the ones used there.}
\cite[Cor.~4.7]{bernardara:brauer_severi}
\begin{equation}\label{deco-for-bs}
\Db(X) = \langle\Db(k), \Db(k,A), \ldots,\Db(k,A^n)\rangle.
\end{equation}
We can see the semiorthogonal decomposition \eqref{semibeilinson} as
the base change of \eqref{deco-for-bs} using descent of vector bundles
as follows. The exceptional line bundle $\ko_{\PP^n_K}$ clearly
descends to $X$, and hence generates $\Db(k)$ inside $\Db(X)$.
Consider now the exceptional line bundle $\ko_{\PP^n_K}(1)$, and
consider it as a stand-alone block. It is well known (cf.\
\cite[\S8.4]{quillen:higher_K-theory}), that
$V=\ko_{\PP^n_K}(1)^{\oplus n+1}$ descends to $X$ and that
$\End(V)=A$.  Similar arguments apply to $\ko_{\PP^n_K}(i)$.  Hence
the decomposition \eqref{deco-for-bs} of $\Db(X)$ can be obtained by
descending the exceptional collection \eqref{semibeilinson} of
$\Db(\PP^n_K)$.
\end{example}

Other known (similar) examples include: the decomposition of a
relative Severi--Brauer variety \cite{bernardara:brauer_severi}, whose
base change is the decomposition of a projective bundle given by Orlov
\cite{orlovprojbund}; the decomposition of a generalized
Severi--Brauer variety given by Blunk \cite{blunk-gen-bs}, whose base
change is the decomposition of a Grassmannian variety given by
Kapranov \cite{kapranovquadric}; and the decomposition of a quadric
hypersurface given in \cite{auel-berna-bolo} (generalizing the work of
Kuznetsov \cite{kuznetquadrics}), whose base change is the
decomposition of a quadric given by Kapranov \cite{kapranovquadric}.

\begin{remark}\label{rmk:descend-a-vb}
One might wonder, motivated by the previous examples, whether the
descent of an exceptional block $\cat{E} \subset \Db(X_{\ksep})$,
generated by vector bundles $V_i$ can be realized by the descent of
the tilting bundle $V = \bigoplus_{i} V_i$.  If $V$ were Galois
invariant, the results of \S\ref{subsec:classical_descent} would
produce a vector bundle $W$ on $X$ pure of type $V$. Moreover, this
would explicitly let us consider the endomorphism algebra
$\mathrm{End}(W)$ to obtain an algebraic description of the descended
category. We wonder whether a converse statement holds: if the block
$\cat{E}$ descends, then is the tilting bundle $V$ Galois-invariant?
\end{remark}

\section{Geometrically rational surfaces}
\label{surfaces}

In this section we collect together some known results on
geometrically rational surfaces, including their classification,
Chow groups, and derived categories.

\subsection{Classification and first properties}

Let $k$ be a field, $\ksep$ a separably closure, and $\kalg$ an
algebraic closure.  A smooth projective geometrically integral surface
$S$ over $k$ such that $\overline{S} = S \times_k \kalg$ is
$\kalg$-rational is called a \linedef{geometrically rational
surface}. Recall that $S$ is a \linedef{del Pezzo} surface if
$\omega_S\dual$ is ample. The \linedef{degree} of a geometrically
rational surface is the self-intersection number $d=\omega_S \cdot \omega_S$.

We say that a field extension $l$ of $k$ is a \linedef{splitting
field} for $S$ if $S \times_k l$ is birational to $\PP^2_l$ via a
sequence of monoidal transformations centered at closed $l$-points.
An important fact is that geometrically rational surfaces are
separably split.  The following result allows one to consider the
separable closure $\ksep$ instead of the algebraic closure.

\begin{proposition}[{\cite[Thm.~1]{coombes:rational_surface_separably_split},
\cite[Thm.~1.6]{varilly:arithmetic_del_pezzo}}]
\label{prop:sep_rat}
If $S$ is a geometrically rational surface over a field $k$, then $S$
is split over $\ksep$. 
\end{proposition}

A surface $S$ is \linedef{minimal} over $k$, or \linedef{$k$-minimal},
if any birational morphism $f: S \to S'$, defined over $k$, is an
isomorphism.  Over a separably closed field, the only minimal rational
surfaces are the projective plane and projective bundles over the
projective line.  Over a general field, this is no longer true.
Minimal geometrically rational surfaces have been completely
classified, and we have the following list (see \cite{manin-book}, and
the recent \cite{hassett-ratsurf}):

\begin{enumerate}
\item\label{minimal:P2} $S=\PP^2_k$ is a projective plane, so $\Pic(S)
= \ZZ$, generated by the hyperplane $\ko(1)$;

\item\label{minimal:quadric} $S \subset \PP^3_k$ is a smooth quadric and $\Pic(S) = \ZZ$, generated by the hyperplane section $\ko(1)$;

\item\label{minimal:dP} $S$ is a del Pezzo surface with $\Pic(S) = \ZZ$, generated by the canonical class $\omega_S$;

\item\label{minimal:conic} $S$ is a conic bundle $f : S \to C$ over a geometrically rational curve, with $\Pic(S) \simeq \ZZ \oplus \ZZ$.
\end{enumerate}

\begin{example}
\label{exam:SB}
A \linedef{Severi--Brauer variety} is a variety $X$ over $k$ such that
$\overline{X} \isom \PP^{n-1}_{\kalg}$ for some $n \geq 2$.  The set
of isomorphism classes of Severi--Brauer varieties $X$ of dimension
$n-1$ over $k$ is in bijection with the set of $k$-isomorphism classes
of central simple algebras $A$ of degree $n$ over $k$, and we will
write $X=\SB(A)$ accordingly.  By a theorem of Ch\^{a}telet, a
Severi--Brauer variety $X$ is $k$-rational if and only if $X$ is
$k$-isomorphic to projective space if and only if $X(k) \neq
\varnothing$, cf.\ \cite[Thm.~5.1.3]{gille_szamuely}.  In this case,
we say that $X$ \linedef{splits} and remark that $X$ always splits
after a finite separable field extension.  The Galois action on
$\Pic(X_{\ksep}) = \Pic(\PP^{n-1}_{\ksep}) = \ZZ$ is trivial since it
preserves dimensions of spaces of global sections.  It is a theorem of
Lichtenbaum that the Brauer obstruction (see Remark~\ref{rem:Leray})
for the Galois invariant class of $\OO_{\PP^2_{\ksep}}(1)$ to descend
to $X$ is precisely the Brauer class of $A$, cf.\
\cite[Thm.~5.4.10]{gille_szamuely}.  Thus for a nonsplit
Severi--Brauer variety $X$, the low terms of the Leray spectral
sequence (see \eqref{eq:Leray}) show that $\omega_X$ generates
$\Pic(X)$.

A \linedef{Severi--Brauer surface} $S$ is a Severi--Brauer variety of
dimension 2, hence is a minimal del Pezzo surface.  As intersection
numbers do not change under scalar extension, $S$ has degree 9.  By
the above analysis of the Picard group, a nonsplit Severi--Brauer
surface belongs to the case \ref{minimal:dP}, while the split
Severi--Brauer surface $\PP^2_k$ belongs to case \ref{minimal:P2}.
\end{example}

\begin{example}
\label{exam:involution}
An \linedef{involution variety} is a variety $X$ over $k$ such that
$\overline{X}$ is $\kalg$-isomorphic to a smooth quadric in
$\PP^{n}_{\kalg}$ for some $n\geq 2$.  The set of isomorphism classes
of involution varieties over $k$ is in bijection with the set of
$k$-isomorphism classes of central simple algebras $(A,\sigma)$ of
degree $n+1$ over $k$ together with a \linedef{quadratic pair}
$\sigma$, a generalization to arbitrary characteristic of the notion
of an orthogonal involution, see \cite[\S5.B]{book_of_involutions} for
definitions.  In this case, $X$ is a degree 2 hypersurface in the
Severi--Brauer variety $\SB(A)$.  Attached to an involution variety
$X$ is the even Clifford algebra $C_0(A,\sigma)$, defined by Jacobson
\cite{jacobson:clifford_algebra_involution} via Galois descent in
characteristic not 2 and by \cite[\S8]{book_of_involutions} in
general.  If $A$ is split, then the quadratic pair on $A$ is adjoint
to a quadratic form, uniquely defined up to scaling and $X \subset
\SB(A) \isom \PP^n_k$ is the associated quadric hypersurface.  We say
that $X$ is an \linedef{anisotropic quadric} if $A$ is split yet $X(k)
= \varnothing$.

From now on, we assume that $X$ has even dimension. In this case,
$C_0(A,\sigma)$ is an Azumaya algebra over its center $l$, which is an
\'etale quadratic extension of $k$, called the \linedef{discriminant
extension} of $X$.  The Galois action on $\Pic(X_{\ksep}) \isom
\ZZ^2$, since it preserves dimensions of spaces of global sections
factors through a permutation of the divisors defining the two rulings
of the quadric over $\ksep$.  This action coincides with the Galois
action on the embeddings $l \to \ksep$ of the discriminant extension.
By a comparison of the low degree terms of the Leray spectral
sequences for $X \to \Spec\, k$ and $\SB(A) \to \Spec\, k$, we see
that the Brauer obstruction to the Galois invariant class of
$\OO_{X_{\ksep}}(1)$ (which is the sum of the two classes of rulings)
is precisely the Brauer class of $A$.  If the discriminant extension
is trivial, then both ruling classes are Galois invariant, and each
one could have its own Brauer obstruction.  Thus for an involution
variety $X$, the low terms of the Leray spectral sequence (see
\eqref{eq:Leray}) show that $\omega_X$ generates $\Pic(X)$ if and only
if the discriminant extension is nontrivial.

An \linedef{involution surface} $S$ is an involution variety of
dimension 2, i.e., $S_\ksep \isom
\PP^1_{\ksep}\times\PP^1_{\ksep}$. In particular, $S$ is a minimal
geometrically rational del Pezzo surface.  As intersection
numbers do not change under scalar extension, $S$ has degree 8.  The
set of $k$-isomorphism classes of involution surfaces is also in
bijection with the set of isomorphism classes of pairs $(l,B)$, where
$l$ is an \'etale quadratic extension of $k$ and $B$ is a quaternion
algebra over $l$, see \cite[\S15.B]{book_of_involutions}.  Given an
involution variety $S$ corresponding to a central simple algebra
$(A,\sigma)$ of degree 4 with quadratic pair, the even Clifford
algebra $C_0(A,\sigma)$ is a quaternion algebra over the discriminant
extension.  Conversely, given a pair $(l,B)$, the associated
involution variety is the Weil restriction $S=R_{l/k} \SB(B)$.  As far
as placing involution surfaces into the classification of minimal
geometrically rational surfaces, there are several cases: If the
discriminant extension is trivial, then $S$ belongs to case
\ref{minimal:conic}.  In this case $S \isom C_1 \times C_2$ is a
product of Severi--Brauer curves.  Writing $C_i = \SB(B_i)$ for
quaternion algebras $B_i$ over $k$, then $A \isom B_1 \tensor B_2$. In
this case, the Brauer class of $A$ is trivial if and only if $C_1
\isom C_2$.  If the discriminant extension is nontrivial, then $S$
belongs to case \ref{minimal:quadric} or \ref{minimal:dP} depending on
whether the Brauer class of $A$ is trivial or not, respectively.
\end{example}

\begin{example}
\label{exam:dP78}
Let $S$ be a geometrically rational del Pezzo surface of degree 8 not
isomorphic to an involution variety.  The unique exceptional curve on
$S_{\ksep}$ is a Galois invariant subvariety, hence can be contracted
to arrive at a del Pezzo surface of degree 9 with a rational point,
which is thus $\PP^2_k$. Thus $S \to \PP^2_k$ is the blow up of
$\PP^2_k$ at a single $k$-rational point.  In particular, $S$ is a
rational del Pezzo surface and is never minimal.

Let $X$ be a geometrically rational del Pezzo surface of degree 7.  As
$S_{\ksep}$ is the blow-up of $\PP^2_{\ksep}$ at two points, the
exceptional divisor consists of a Galois invariant pair of
$(-1)$-curves, hence can be contracted to arrive at a del Pezzo
surface of degree 9 with a point of degree 2, which is thus $\PP^2_k$.
Thus $S \to \PP^2_k$ is the blow up of $\PP^2_k$ at a closed point
of degree 2.  In particular, $S$ is a rational del Pezzo surface and
is never minimal.
\end{example}

\begin{remark}
If $S$ is a minimal del Pezzo surface of degree $\leq 6$, then
$\Pic(S) \to \Pic(S_{\ksep})^{G_k}$ is an isomorphism, cf.\ \cite[Lemma~2.5,~Prop.~5.3]{colliot-karpenko-merkurev}.
\end{remark}

We continue our general discussion of geometrically rational surfaces.
Denote by $\rho(S)$ the Picard rank of $S$.

\begin{proposition}
\label{prop:K_0-minimal}
If $S$ is a geometrically rational surface over a field $k$, then
$K_0(S)_\QQ$ is a $\QQ$-vector space of dimension $2+\rho(S)$.  In
particular, if $S$ is minimal then $K_0(S)_{\QQ}$ has dimension 3 or
4, and in the latter case $S$ has a conic bundle structure.
\end{proposition}
\begin{proof}
The Chern character
$$
ch : K_0(S) \tensor_\ZZ \QQ \to \bigoplus_i \CH^i(S) \tensor_\ZZ \QQ
$$
is an isomorphism.  We have $\CH^0(S) \isom \ZZ$ and $\CH^1(S) \isom
\Pic(S)$.  Finally, to describe $\CH^2(S) = \CH_0(S)$, we have an
exact sequence
$$
0 \to A_0(S) \to \CH_0(S) \mapto{\deg} \ZZ \to \ZZ/i(S)\ZZ \to 0
$$
where $A_0(S)$ is defined to be the kernel of the degree map and
$i(S)$ is the \linedef{index}, i.e., the greatest common divisor of
degrees of closed points on $S$.  Since $S$ is rational after a finite
separable extension (by Proposition~\ref{prop:sep_rat}), a
restriction-corestriction argument shows that $A_0(S)$ is torsion,
cf.\ \cite[Prop.~6.4]{colliot-thelene_coray}.  This implies that the
degree map becomes an isomorphism $\CH_0(S) \tensor_\ZZ \QQ
\stackrel{\simeq}{\to} \QQ$ after tensoring with $\QQ$.  This
completes the calculation of the dimension.  The final statement is a
result of the classification of minimal geometrically rational
surfaces.
\end{proof}

\begin{corollary}\label{coro:minimal-have-3-blocks}
Let $S$ be a geometrically rational surface over $k$.  If there exist
field extensions $l_1, \ldots, l_n$ of $k$ and Azumaya algebras $A_i$
over $l_i$, such that there is a semiorthogonal decomposition
$$
\Db(S) = \langle \Db(l_1/k,A_1), \ldots, \Db(l_n/k,A_n) \rangle,
$$
then $n=2+\rho(S)$.  In particular, $S$ is categorically representable
in dimension 0 if and only if there exist field extensions $l_1,
\ldots, l_n$, with $n=\rho(S)+2$, and a semiorthogonal decomposition
\begin{equation}\label{eq-semiort-for-represent}
\Db(S) = \langle \Db(l_1/k),\ldots, \Db(l_n/k) \rangle.
\end{equation}
\end{corollary}
\begin{proof}
The first statement is a corollary of
Proposition~\ref{prop:K_0-minimal}, since the semiorthogonal
decomposition gives a splitting
$$
K_0(S) = \bigoplus_{i=1}^n K_0(l_i,A_i)
$$
and $K_0(l,A) \isom \ZZ$ for any field $l$ and any Azumaya algebra $A$
over $l$.  To prove the second statement notice that $S$ is
categorically representable in dimension zero if and only if a
semiorthogonal decomposition like \eqref{eq-semiort-for-represent}
exists by Lemma~\ref{lem:0-dim=etale-algebra}, and the number of
component is given by the first statement.
\end{proof}

Let us recall a straightforward consequence of Orlov's result on
blow-ups, reducing the question of being categorically representable in
dimension 0 to minimal surfaces.

\begin{lemma}\label{lem:exclude-non-minimal}
Let $S$ be a smooth projective non-minimal surface over $k$. Then
there is a smooth projective minimal surface $S'$, and a fully
faithful functor $\Phi:\Db(S') \to \Db(S)$ such that the orthogonal
complement of $\Phi(\Db(S'))$ is representable in dimension 0.
\end{lemma}
\begin{proof}
Since $S$ is not minimal, there exists a $k$-birational morphism $\pi:
S \to S'$ to a minimal surface.  Then $\pi$ is the blow-up of a closed
zero-dimensional subvariety $Z \subset S'$.  The proof follows from
Orlov blow-up formula \cite{orlovprojbund}.
\end{proof}

Manin has proved that, given a (non necessarily minimal) del Pezzo
surface of degree $d \geq 2$, the existence of a $k$-rational point
(not lying on any exceptional curve if $d \leq 4$) implies the
existence of a unirational parametrization, i.e., a map $\PP^2_k \to
S$ of finite degree. 

\begin{theorem}[{Manin~\cite[Thm.~29.4]{manin-book}}]
\label{thm:manin-unirat-param}
Let $S$ be a del Pezzo surface of degree $d \geq 2$ over $k$ with
$S(k) \neq \varnothing$.
If $d \leq 4$ suppose
moreover that the point does not lie on any exceptional curve. Then
there exists a rational map $\phi: \PP^2_k \to S$ whose degree
$\delta_d$ is given by the following table \medskip
\begin{center}
\begin{tabular}{c|cccc}
$d$ & $\geq 5$ & $4$ & $3$ & $2$ \\
\hline
$\delta_d$ & $1$ & $2$ & $6$ & $24$ 
\end{tabular}
 \end{center}
\medskip
In particular, if $d \geq 5$, the surface $S$ has a $k$-rational point if and only if it is $k$-rational.
\end{theorem}

Minimal del Pezzo surfaces of degree $\leq 7$ can be characterized by
the Galois action on exceptional curves.  We say that a del Pezzo
surface $S$ over $k$ is \linedef{totally split} if $S$ is $k$-rational
and all exceptional curves are defined over $k$.  Any field extension
of $k$ over which a del Pezzo surface becomes totally split will be
called a \linedef{total splitting field} for $S$.  We can always
choose a finite Galois total splitting field for a del Pezzo surface.
We remark that $S$ can be split, but not necessarily totally split,
over a given field.

We end this section by recalling the following classification of birational maps between non-rational minimal
del Pezzo surfaces, which can be proved by classifying all the
possible links in the Sarkisov program (see \cite{isko-sarkisov-complete}).

\begin{proposition}[{Iskovskikh~\cite[Thms.~1.6,~4.5,~4.6]{isko-sarkisov-complete}}]
\label{prop:isko-rigidity-of-nonrat-dps}
Let $S$ be a non rational del Pezzo surface of Picard rank 1, $S'$ a
minimal surface, and $\phi: S \dashrightarrow S'$ a
$k$-birational map.
\begin{enumerate}
\item \label{isko.i} If $\deg(S)=1$ or if $S$ has no closed point $x$
of degree $< \deg(S)$, then $\phi$ is an isomorphism (i.e., $S$ is \rm
birationally rigid\it).

\item \label{isko.ii} If $\deg(S)=2$ and $S(k) \neq \varnothing$, or
if $\deg(S)=3$ and $S(k) \neq \varnothing$ or if $\deg(S)=4$ and $S$
has a point of degree 2 and no point of degree 1 or 3, or if $S$ has
degree $8$ and a point of degree 4 (but no point of lower degree),
then $\phi$ can be composed with a birational map $S \dashrightarrow
S$ to give an isomorphism (i.e. $S$ is \rm birationally semirigid\it).

 \item \label{isko.iii} 
If $\deg(S)=6$ or $\deg(S)=9$, then
 $\deg(S)=\deg(S')$ (i.e. $S$ is \rm deg-rigid\it).

\item \label{isko.iv} 
If $\deg(S)=4$ and $S$ has a $k$-rational point or if $\deg(S)=8$ and $S$ has a degree 2 point
 (but no $k$-rational point), then the blow-up of $S$ along such a point is a conic bundle of
 degree $3$ or $6$, respectively. In particular, $S$ is not deg-rigid.
 \end{enumerate}
All non rational del Pezzo surfaces of Picard rank one
are covered by one of these cases.
\end{proposition}
\begin{proof}
A proof of all of these results can be found in
\cite{isko-sarkisov-complete} (some of them were previously
known). Item \ref{isko.i} is \cite[Thm.~1.6]{isko-sarkisov-complete},
while items \ref{isko.ii}, \ref{isko.iii}, and \ref{isko.iv} are
summarized in \cite[Comment~5]{isko-sarkisov-complete}. Notice that if
$S$ has degree $8$ and a point of degree $4$ but no point of lower
degree, then \cite[Thm.~4.4]{isko-sarkisov-complete} only says that
$S$ is deg-rigid. 
In this case, $S$ is an involution variety, for
which the birational semirigidity can be proved directly via the
theory of quadratic forms, see Proposition~\ref{prop:birational_involution_partners}.

Finally, the list is exhaustive since if $\deg(S) \leq 6$ and $S$ has
a closed point of degree $< \deg(S)$ and coprime with
$\deg(S)$, then $S$ is rational (for an argument, see
\cite[p.~624]{isko-sarkisov-complete}). Similarly if $S$ has degree
$5$ or $7$ then it is rational.
\end{proof}

\part{Del Pezzo surfaces}

The plan of this part is as follows.  In Section
\ref{sect:karpovnogin} we recall the necessary facts concerning
3-block decompositions. In Section \ref{sect:decos-of-minimal-dps} we
consider the case of degree $d \leq 4$ and prove Theorem \ref{thm:main} and \ref{thm:birat-class-of-dpsmall}.
The rest of the proofs 
will be divided into the cases of Severi--Brauer surfaces (Section
\ref{subsec:dP9}), involution surfaces (Section \ref{subsec:dP8}), del
Pezzo surfaces of degree 7 (Section \ref{subsec:dP7}), 6 (Section \ref{subsec:dP6}), and 5 (Section
\ref{subsec:dP5}).

\section{Exceptional objects and blocks on del Pezzo surfaces}\label{sect:karpovnogin}

In this section, we assume that $k$ is separably closed.  The derived
category of a totally split del Pezzo surface (over an algebraically
closed field) has been extensively studied by the Moscow school in the
1990's (see, e.g., \cite{gorodentsev-moving},
\cite{gorodentsev-rudakov}, \cite{rudakov-quadric},
\cite{kuleshov-orlov,karpov-nogin}), with a particular attention to
the structure of exceptional collections.

While these authors restrict to working over algebraically closed
fields of characteristic zero, their proofs are based on properties of
vector bundles and on the description of a del Pezzo surface as a
blow-up of $\PP^2$, hence they actually hold for any totally
split del Pezzo surface, in particular, over any separably closed
field.

Let $S$ be a del Pezzo surface over $k$. Then $S$ is totally split and
$S$ is either a quadric surface and has degree 8, or $S$ is a blow-up
of $\PP^2$ in $9-d$ points, for $1 \leq d \leq 9$, and has degree $d$.
A \linedef{3-block decomposition} of $\Db(S)$ is a semiorthogonal
decomposition
$$
\Db(S) = \langle \cat{E}, \cat{F}, \cat{G} \rangle
$$
consisting of exceptional blocks.  Gorodentsev--Rudakov, Rudakov, and
Karpov--Nogin (see \cite{gorodentsev-rudakov}, \cite{rudakov-quadric},
\cite{karpov-nogin}) have proved the existence of 3-block
decompositions by exceptional vector bundles with some unicity
property (up to mutations). Indeed, the ranks and degrees of vector
bundles are constant in a block, and there is an algorithm to compute
how slopes change under mutations, as explained in
\cite{karpov-nogin}.  A 3-block decomposition is \linedef{minimal} if
any mutation increases the rank of one of the blocks.  The existence
of minimal decomposition is related to solutions of Markov-type
equations.

\begin{table}
\small{
\centering
\begin{tabular}{||c||c|c|c||c|c|c||c|c|c||}
\hline\hline

\multicolumn{1}{||c||}{\multirow{2}{*}{$\deg(S)$}} & \multicolumn{3}{|c||}{$\cat{E}$} & \multicolumn{3}{|c||}{$\cat{F}$} & \multicolumn{3}{|c||}{$\cat{G}$} \\
\cline{2-10}
& $n$ & $r$ & exc. collection &  $n$ & $r$ & exc. collection &  $n$ & $r$ & exc. collection \\
\hline \hline
9 & 1 & 1 & $\ko$ & 1 & 1 & $\ko(H)$ & 1 & 1 & $\ko(2H)$  \\
\hline
\multirow{2}{*}{8 (inv)} & \multirow{2}{*}{1} & \multirow{2}{*}{1} & \multirow{2}{*}{$\ko$} & \multirow{2}{*}{2} & \multirow{2}{*}{1} 
& $\ko(1,0)$ & \multirow{2}{*}{1} & \multirow{2}{*}{1} & \multirow{2}{*}{$\ko(1,1)$} \\
& & & & & & $\ko(0,1)$ & & &\\
\hline
8 (dP) & \multicolumn{9}{|c||}{\multirow{2}{*}{No 3-block decomposition}}\\
\cline{1-1}
7  & \multicolumn{9}{c||}{}\\
\hline
\multirow{3}{*}{6} & \multirow{3}{*}{1} & \multirow{3}{*}{1} & \multirow{3}{*}{$\ko$} & \multirow{3}{*}{2} & \multirow{3}{*}{1} &
$\ko(H)$ & \multirow{3}{*}{3} & \multirow{3}{*}{1} & $\ko(2H-L_1-L_2)$  \\
& & & & & & & & & $\ko(2H-L_1-L_3)$ \\
& & & & & & $\ko(2H-L_{\leq 3})$ & & & $\ko(2H-L_2-L_3)$\\
\hline
\multirow{5}{*}{5} & \multirow{5}{*}{1} & \multirow{5}{*}{1} & \multirow{5}{*}{$\ko$} & \multirow{5}{*}{1} & \multirow{5}{*}{2} & \multirow{5}{*}{$F$} 
& \multirow{5}{*}{5} & \multirow{5}{*}{1} & $\ko(H)$ \\
& & & & & & & & & $\ko(L_1-\omega-H)$ \\
& & & & & & & & & $\ko(L_2-\omega-H)$ \\
& & & & & & & & & $\ko(L_3-\omega-H)$ \\
& & & & & & & & & $\ko(L_4-\omega-H)$ \\
\hline
\multirow{4}{*}{4} & \multirow{4}{*}{2} & \multirow{4}{*}{1} & \multirow{2}{*}{$\ko(L_4)$} & \multirow{4}{*}{2} & \multirow{4}{*}{1} & $\ko(H)$ & 
\multirow{4}{*}{4} & \multirow{4}{*}{1} & $-\omega$ \\
& & & & & & & & & $\ko(2H-L_1-L_2)$ \\
& & & \multirow{2}{*}{$\ko(L_5)$} & & & \multirow{2}{*}{$\ko(2H-L_{\leq 3})$} & & & $\ko(2H-L_1-L_3)$ \\
& & & & & & & & & $\ko(2H-L_2-L_3)$ \\
\hline
\multirow{3}{*}{3 $(i)$} & \multirow{3}{*}{3} & \multirow{3}{*}{1} & $\ko(L_4)$ & \multirow{3}{*}{3} & \multirow{3}{*}{1} 
& $\ko(H-L_1)$ & \multirow{3}{*}{3} & \multirow{3}{*}{1} & $-\omega$  \\
& & & $\ko(L_5)$ & & & $\ko(H-L_2)$ & & & $\ko(H)$ \\
& & & $\ko(L_6)$ & & & $\ko(H-L_3)$ & & & $\ko(2H-L_{\leq 3})$ \\
\hline
\multirow{2}{*}{3 $(ii)$} & \multirow{2}{*}{1} & \multirow{2}{*}{2} & \multirow{2}{*}{$T_6$} & \multirow{2}{*}{2} & \multirow{2}{*}{1} & $\ko(H)$  &
\multirow{2}{*}{6} & \multirow{2}{*}{1} & $\ko(L_1-\omega)\,\,\,\,\,\ko(L_2-\omega)\,\,\,\,\,\ko(L_3-\omega)$  \\
& & & & & & $-\omega$ & & & $\ko(L_4-\omega)\,\,\,\,\,\ko(L_5-\omega)\,\,\,\,\,\ko(L_6-\omega)$ \\
\hline
\multirow{2}{*}{2 $(i)$} & \multirow{2}{*}{1} & \multirow{2}{*}{2} & \multirow{2}{*}{$E_7$} & \multirow{2}{*}{1} & \multirow{2}{*}{2} 
& \multirow{2}{*}{$T_7$} & \multirow{2}{*}{8} & \multirow{2}{*}{1} & $-\omega$ \\
& & & & & & & & & $\ko(H-L_1) \,\, \ldots \,\, \ko(H-L_7)$ \\
\hline
\multirow{4}{*}{2 $(ii)$} & \multirow{4}{*}{2} & \multirow{4}{*}{2} & \multirow{2}{*}{$E_7$} & \multirow{4}{*}{4} & \multirow{4}{*}{1} 
& $\ko(L_4)$ & \multirow{4}{*}{4} & \multirow{4}{*}{1} & $-\omega$ \\
& & & & & & $\ko(L_5)$ & & & $\ko(H-L_1)$ \\
& & & $E_7'$ & & & $\ko(L_6)$ & & & $\ko(H-L_2)$ \\
& & & & & & $\ko(L_7)$ & & & $\ko(H - L_3)$ \\
\hline
\multirow{3}{*}{2 $(iii)$} & \multirow{3}{*}{1} & \multirow{3}{*}{3} & \multirow{3}{*}{$E_7''$} & \multirow{3}{*}{3} & \multirow{3}{*}{1} 
& $\ko(H-L_1)$ & \multirow{3}{*}{6} & \multirow{3}{*}{1} & $\ko(H)$ \\
& & & & & & $\ko(H-L_2)$ & & & $\ko(2H-L_{\leq 3})$ \\
& & & & & & $\ko(H-L_3)$ & & & $\ko(L_4-\omega) \,\, \ldots \,\, \ko(L_7-\omega)$ \\
\hline
\multirow{2}{*}{1 $(i)$} & \multirow{2}{*}{1} & \multirow{2}{*}{3} & \multirow{2}{*}{$E_8$} & \multirow{2}{*}{1} & \multirow{2}{*}{3} & \multirow{2}{*}{$F_8$}
& \multirow{2}{*}{9} & \multirow{2}{*}{1} & $-\omega$  \\
& & & & & & & & & $\ko(-L_1) \,\, \ldots \,\, \ko(-L_8)$ \\
\hline
\multirow{2}{*}{1 $(ii)$} & \multirow{2}{*}{1} & \multirow{2}{*}{4} & \multirow{2}{*}{$E_8'$} & \multirow{2}{*}{2} & \multirow{2}{*}{2}
& $T_8$ & \multirow{2}{*}{8} & \multirow{2}{*}{1} & $\ko(L_4-\omega)\,\,\ldots\,\,\ko(L_8-\omega)$ \\
& & & & & & $T_8'$ & & & $\ko(H-L_1) \,\,\,\,\, \ko(H-L_2) \,\,\,\,\, \ko(H-L_3)$ \\
\hline
\multirow{4}{*}{1 $(iii)$} & \multirow{4}{*}{2} & \multirow{4}{*}{4} & \multirow{2}{*}{$E_7''$} & \multirow{4}{*}{3} & \multirow{4}{*}{2} 
& $T_8$ & \multirow{4}{*}{6} & \multirow{4}{*}{1} & \multirow{2}{*}{$\ko(L_4-\omega) \,\,\,\,\, \ko(L_5-\omega) \,\,\,\,\, \ko(L_6-\omega)$} \\
& & & & & & \multirow{2}{*}{$T_8'$} & & & \\
& & & \multirow{2}{*}{$E_8''$} & & & & & & \multirow{2}{*}{$\ko(H-L_1)\,\,\,\,\, \ko(H-L_2)\,\,\,\,\, \ko(H-L_3)$}\\
& & & & & & $T_8''$ & & & \\
\hline
\multirow{3}{*}{1 $(iv)$} & \multirow{3}{*}{1} & \multirow{3}{*}{5} & \multirow{3}{*}{$E_8'''$} & \multirow{3}{*}{5} & \multirow{3}{*}{2} 
& \multirow{3}{*}{$F_{4,8}\,\, \ldots \,\, F_{8,8}$} & \multirow{3}{*}{5} & \multirow{3}{*}{1} & $\ko(H)$\\
& & & & & & & & & $\ko(2H-L_{\leq 3})$ \\
& & & & & & & & & $\ko(H-L_1-\omega)\,\,\ldots\,\,\ko(H-L_3-\omega)$ \\
\hline\hline
\end{tabular}}
\caption{Representative sets of exceptional objects generating the
various 3-block
decompositions of $\Db(S_\ksep)$, taken up to mutation, tensoring through
by a line bundle, and the Weyl group action.  For each
block $n=$ number of bundles in the block and $r=$ rank of these bundles. Here: $S \to
\PP^2_{\ksep}$ is the blow-up at $9-\deg(S)$ points, $L_i$ the
exceptional divisors, and $H$ the pull back of the
hyperplane class, $L_{\geq 4}:=\sum_{i \geq 4} L_i$, and $L_{\leq 3}:=\sum_{i \leq 3}L_i$.
For the higher rank vector bundles, we follow the notation in Karpov--Nogin~\cite[\S
4]{karpov-nogin}.} 
\label{table:blocks}
\end{table}

\begin{proposition}[\cite{gorodentsev-rudakov},
\cite{rudakov-quadric}, \cite{karpov-nogin}]
\label{prop:karpov-nogin}
Let $S$ be either a quadric surface or a del Pezzo surface of degree
$d \neq 7,8$ over a separably closed field $k$.  Set $s =
\max\{1,5-d\}$.  Then $\Db(S)$ has $s$ (up to tensoring by
line bundles and the action of the Weyl group) minimal 3-block
decompositions such that any other 3-block decomposition of $\Db(S)$
can obtained from one of these by a finite number of mutations.

In all cases, the blocks are generated by a completely orthogonal set
of vector bundles, so that each block has a tilting bundle.
\end{proposition}

In the cases of degree $d=7,8$ (see Example \ref{exam:dP78}), there is
always a 4-block decomposition, but not a 3-block one.  On the other
hand, del Pezzo surfaces of these degrees are never minimal (see
Example~\ref{exam:dP78}).

We summarize the possible minimal 3-block decompositions in
Table~\ref{table:blocks}. Recall that these decompositions hold over a
separably closed field, or in general for totally split del Pezzo
surfaces, so that $S$ is either $\PP^1_k \times \PP^1_k$ or the blow
up of $\PP^2$ in $9-d$ rational points, and exceptional means
$k$-exceptional. In the first case we use the standard notation
$\ko(a,b)$ for line bundles of bidegree $(a,b)$. In the latter, we
denote by $L_i$ (for $i=1,\ldots r$) the exceptional divisors of $S
\to \PP^2$, and by $H$ the pull-back of the hyperplane class in
$\PP^2$. For each block, the table lists the number $n$ of exceptional
bundles, their rank $r$ (which is constant within the block), and the
1st Chern class of the tilting bundle of the block (i.e., the sum of
the 1st Chern classes of the bundles in the block).

\section{Del Pezzo surfaces of Picard rank 1 and low degree}\label{sect:decos-of-minimal-dps}

From now on, let $k$ be an arbitrary field and $S$ be a del Pezzo
surface of degree $d$ and Picard rank 1 over $k$.  Recall that
$K_0(S)_{\QQ} \simeq\QQ^{\oplus 3}$, so that we can wonder, according
to Proposition~\ref{prop:K_0-minimal}, whether there is a
semiorthogonal decomposition given by three simple $k$-algebras. If
such a decomposition exists, it must base change to a 3-block
decomposition of $\Db(S_{\ksep})$ by
Proposition~\ref{prop:descend-a-block}. Conversely, if we suppose that
there is a semiorthogonal decomposition
$$
\Db(S) = \langle \cat{E}, \cat{F}, \cat{G} \rangle
$$
whose base change is a 3-block decomposition, then
Proposition~\ref{prop:descend-a-block} guarantees that the three
components are equivalent to derived categories of simple
$k$-algebras.

Combining the explicit description of the vector bundles forming
exceptional blocks on $S_{\ksep}$, together with the previous
observations, we gain control over semiorthogonal decompositions of
derived categories of del Pezzo surfaces of Picard rank 1.

\begin{theorem}
\label{thm:decos-of-minimal-dps}
Let $S$ be a del Pezzo surface of degree $d \leq 4$. Then there is no
semiorthogonal decomposition
\begin{equation}
\label{eq:semiortho-for-minimal}
\Db(S) = \langle \Db(l_1/k, \alpha_1), \Db(l_2/k,\alpha_2), \Db(l_3/k, \alpha_3) \rangle,
\end{equation}
with $l_i$ field extensions of $k$ and $\alpha_i$ in $\Br(l_i)$. 
\end{theorem}
\begin{proof}
If a decomposition \eqref{eq:semiortho-for-minimal} exists, then
$\rho(S)=1$ by Corollary~\ref{coro:minimal-have-3-blocks}, hence
$\Pic(S) = \ZZ[\omega]$ by the classification.  Moreover,
Proposition~\ref{prop:descend-a-block} ensures us that its base change
to $\ksep$ is a 3-block decomposition. Up to mutating the three
blocks, tensoring with line bundles, and the action of the Weyl group,
we can appeal to the Karpov--Nogin classification (cf.\
Proposition~\ref{prop:karpov-nogin}) and proceed in a case-by-case
analysis.  Such exceptional blocks are generated by vector bundles, so
Corollary~\ref{prop:single-exc-bundle-G-inv} guarantees that a
nontrivial $\ZZ$-linear combination of the first Chern classes of
these vector bundles is a multiple of $\omega$.

First, the Galois action on $\langle \omega \rangle^\perp \subset
\Pic(S_\ksep)$ factors through the Weyl group of the associated root
system, see \cite[Thm.~2]{demazure}.  Hence over $\ksep$, there is a
choice of $9-d$ pairwise disjoint exceptional lines $L_1,\dotsc,L_d$
so that the three blocks are described as in Table \ref{table:blocks},
up to tensoring all exceptional bundles by the same line bundle
$\ko(M)$ in $\Pic(S_{k\sep})$.

Our argument will follow two different paths, depending on the degree
and subcase:

\begin{enumerate}
\item[\textit{(1)}] 

Either we show that one of the blocks contains a proper
admissible subcategory generated by $\omega$, and then
Lemma~\ref{lem:indecomposable} shows that this contradicts
$\rho(S)=1$.

\item[\textit{(2)}] 

Or we show the impossibility of descending a non-trivial
generator of one of the blocks, by proving that its first Chern class
could never be a multiple of $\omega$.

\end{enumerate}

In what follows, we will consider a line bundle $M$ on $S_{k\sep}$
written as $M=nH+\sum_{i=1}^{d} a_iL_i$.  Tensoring by powers of
$\omega= -3H + \sum_{i=1}^d L_i$, we can choose to fix one of the
coefficients $a_i$ or a representative of $n$ modulo 3.  

\medskip

{\it Degree 4.}  In degree 4, $\cat{E}$ is generated by $\ko(L_4)$ and
$\ko(L_5)$ over $k\sep$.  Let $\ko(M)$ be a line bundle on
$S_{k\sep}$.  If there exists a pair of integers $a$ and $b$ such that
\begin{equation}\label{eq:L4L5}
a(L_4+M)+b(L_5+M)=r\omega,
\end{equation}
then $r=0$.  Indeed, we make this first calculation explicit: let us
assume that $M= nH + \sum_{i=1}^5 a_iL_i$ with $a_1=0$. Then it
follows that $r=0$, in which case, we must also have $n=0$ and
$a_2=a_3=0$. Of course, we can assume that both $a$ and $b$
are not both zero, otherwise no nontrivial generator of the block
$\cat{E}$ descends.  In fact, since $\Pic(S_\ksep)$ is torsion-free,
we can further take $a$ and $b$ to be coprime.

With this in mind, equation \eqref{eq:L4L5} yields
$$
a(L_4 + a_4L_4 + a_5L_5) + b(L_5 + a_4L_4 + a_5L_5)=0.
$$
It follows, looking at the coefficients of $L_4$ and $L_5$,
respectively, that
\begin{equation}\label{eq:calcul-of-a-b}
\left\{ \begin{array}{l}
a + (a+b) a_4 = 0 \\
b + (a+b) a_5 = 0
          \end{array}\right.
\end{equation}
Since $a$ and $b$ are coprime, $a+b\neq 0$ and hence the only integer
solutions to this system of equations has $a_4=a_5=0$ (which implies
$a=b=0$, a contradiction) or $a+b=\pm 1$.

Suppose that $a+b=1$. From \eqref{eq:calcul-of-a-b}, we get that
$a_4=-a$ and $a_5=a-1$, so that the possibilities to descend the block
$\cat{E}$ are obtained by tensoring the Karpov--Nogin 3-block
decomposition over $S_\ksep$ by $M=-aL_4 + (a-1)L_5$, for some integer
$a$. Now we consider the block $\cat{F}$. After tensoring with
$\ko(M)$, the block $\cat{F}$ is generated, over $k\sep$, by
$\ko(H+M)$ and $\ko(2H-L_1-L_2-L_3+M)$. If the block descends to $k$,
we have integers $\alpha$, $\beta$, and $\rho$ such that
\begin{equation}\label{eq:with-alph-bet}
\alpha(H+M)+\beta(2H-L_1-L_2-L_3+M) = \rho \omega.
\end{equation}
Looking at coefficients of $L_1$ in \eqref{eq:with-alph-bet}, we get
$\beta=-\rho$. Then considering the coefficient of $H$ in
\eqref{eq:with-alph-bet} gives $\alpha=-\rho$.  Finally, the
coefficients of $L_4$ and $L_5$ give $\alpha\rho=\rho$ and
$(1-\alpha)\rho=\rho$, respectively. From this, it follows that
$\rho=0$, so that $\alpha=\beta=0$, hence no nontrivial generator of
the block $\cat{F}$ can descend.

If we suppose that $a+b=-1$, then $M=aL_4 - (a+1) L_5$, and similar
arguments show that no nontrivial generator of $\cat{F}$ can descend.
Using Corollary~\ref{prop:single-exc-bundle-G-inv}, this means that
there is no way to descend both the blocks $\cat{E}$ and
$\cat{F}$ to $k$ at the same time.

\medskip

{\it Degree 3, case $(i)$.} The block $\cat{E}$ is generated by
$\ko(L_1)$, $\ko(L_2)$, and $\ko(L_3)$.  We consider the equation
$$
a(L_1+M) + b(L_2+M) + c(L_3+M) = r\omega.
$$
As before, up to tensoring by
multiples of $\omega$, we can assume that $M = nH + \sum_{i=1}^6
a_iL_i$ with $a_4=0$, from which we similarly conclude that $r=0$,
hence $c_1(\cat{E})=0$, and $n=a_4=a_5=a_6=0$.

The block $\cat{F}$ is generated by $\ko(H-L_4)$, $\ko(H-L_5)$,
$\ko(H-L_6)$, so that we are looking for nontrivial integers $\alpha$, $\beta$,
and $\gamma$ such that
$$
\alpha(H-L_4+M)+\beta(H-L_5+M)+\gamma(H-L_6+M)=\rho\omega
$$
where $M=a_1L_1 + a_2L_2+ a_3L_3$, as just shown. From this, we
get $\alpha+\beta+\gamma = -3\rho$ (considering the coefficients of
$H$) and $\alpha+\beta+\gamma=-\rho$ (considering the coefficients of $L_4$). It
follows that $\rho=0$, and hence $c_1(\cat{F})=0$. Notice that the latter
does not mean that the coefficients $\alpha$, $\beta$, and $\gamma$ are trivial, so
we can't appeal to Corollary~\ref{prop:single-exc-bundle-G-inv} here.
However, looking at the coefficients of $L_1$, $L_2$ and $L_3$, we
also get $a_1=a_2=a_3=0$, from which we have that $M$ is a multiple of
$\omega$.

Over $\ksep$, the block $\cat{G}$ contains $\omega$ as part of an
exceptional sequence (according to Table~\ref{table:blocks}), hence
after tensoring by $M$, the block $\cat{G}$ contains a multiple
$\omega^{\otimes m}$ for some integer $m$.
The category generated by this
$\omega^{\otimes m}$ is a proper admissible subcategory of $\cat{G}$.
If $\cat{G}$ descends to an admissible subcategory $\Db(K,A)$ of
$\Db(S)$, then this subcategory admits $\langle \omega^{\otimes m}
\rangle$ as a proper admissible subcategory. Then
Lemma~\ref{lem:indecomposable} shows that $K$ cannot be a field.

\medskip

{\it Degree 3, case $(ii)$.} The block $\cat{E}$ is generated by a rank 2
vector bundle $T_6$ with $c_1(T_6)=H$. We have $c_1(T_6 \otimes M) = H
+ 2M$, and for the block to descend, we need $H+2M = r\omega$.  Taking
$M$ (up to multiples of $\omega$) with $a_1=0$, we conclude that
$r=0$.  As a result, $a_i=0$ for all $i$, so that $M=nH$ is a multiple
of $H$, and then $H+2M=(2n+1)H$. We conclude that the only multiple of
$T_6$ that can descend is the trivial one. This contradicts Corollary~\ref{prop:single-exc-bundle-G-inv}.

\medskip

{\it Degree 2, case $(i)$.} This case is very similar to the case of {\it
Degree 3, case $(ii)$}.  Indeed, the block $\cat{F}$ is generated by a
single rank 2 vector bundle with first Chern class $H$.

\medskip

{\it Degree 2, case $(ii)$.} 
In this case, $\cat{F}$ is generated by
$\ko(L_4), \ldots, \ko(L_7)$.  Arguing as in the case of {\it Degree
4}, we deduce that if there exists a linear combination of $L_4+M,
\dotsc, L_7+M$ (a necessary condition for $\cat{F}$ to descend) then
$r=0$ and $M=\sum_{i=4}^7 a_iL_i$. 
The block $\cat{E}$ is generated
by two rank 2 vector bundles of first Chern classes $H + \omega$ and
$-H+L_{\geq4}$, respectively. As a necessary condition for $\cat{E}$ to descend, we are looking for a pair of
integers $a$ and $b$ such that
$$
a(H+\omega+2M) + b(-H+L_{\geq4}+2M)=r\omega.
$$
Since $M=a_4L_4+\ldots+a_7L_7$, considering the coefficient of $L_1$,
we arrive at $a=r$, from which it follows that $r(H+2M) +
b(-H+L_{\geq4}+2M)=0$, by subtracting $r\omega$ on both sides.  Now,
looking at the coefficient of $H$, we arrive at $b=r$, hence we
conclude that $r(4M+L_{\geq 4})=0$. Once again considering the
coefficient of $H$, we see that $r\neq 0$ is impossible.  It follows
that $a=0$ and hence $b=0$, as well. This contradicts Corollary~\ref{prop:single-exc-bundle-G-inv}.

\medskip

{\it Degree 2, case $(iii)$.}  In this case, the block $\cat{E}$ is
generated, over $\ksep$, by a rank 3 vector bundle with first Chern
class $L_{\geq 4}$.  A necessary condition for $\cat{E}$ to descend is
that there exists an integer $a$ such that $a(L_{\geq
4}+3M)=r\omega$.  Taking $M$ with $a_1=0$, we easily get $r=0$ and
$M=a_4L_4+ \ldots + a_7L_7$.  As in the case of {\it Degree 2, case $(ii)$}
we find a contradiction to Corollary~\ref{prop:single-exc-bundle-G-inv}.

\medskip

{\it Degree 1, case $(i)$.} This case is similar to that of {\it Degree 3,
case $(ii)$}. Indeed, in both cases the block $\cat{E}$ is generated by a
single vector bundle.  In this case, its first Chern class is
$-H+2\omega$.  Arguing as in the previous case, we arrive at a
contradiction to Corollary~\ref{prop:single-exc-bundle-G-inv}.

\medskip

{\it Degree 1, case $(ii)$.} In this case the block $\cat{E}$ is generated
by a single vector bundle of rank $4$ and first Chern class
$2H-L_1-L_2-L_3$.  Modifying $M$ by multiples of $\omega$, we can
choose $a_4=0$, and we can then conclude that
$a(2H-L_1-L_2-L_3+4M)=r\omega$ implies that $r=0$ and
$M=nH+\sum_{i=4}^7a_iL_i$. Looking at the coefficients of $H$, we see
that if $a\neq 0$, there is no $n$ such that the latter equation
holds.  Hence $a=0$ and we find a contradiction to Corollary~\ref{prop:single-exc-bundle-G-inv}.

\medskip

{\it Degree 1, case $(iii)$.} In this case, the block $\cat{E}$ is generated
by two rank 3 vector bundles with first Chern classes
$L_4+L_5+L_6+L_7$ and $L_4+L_5+L_6+L_8$, respectively. Hence
a necessary condition for $\cat{E}$ to descend is that there exist
integers $a$ and $b$ such that
$$
a(L_4+L_5+L_6+L_7+3M) + b(L_4+L_5+L_6+L_8+3M)=r\omega.
$$
As before, we can assume that $a$ and $b$ are both nonzero and
coprime.  Choosing $M=nH + \sum_{i=1}^8 a_iL_i$ with $a_1=0$, we
arrive at $r=0$ (by considering the coefficient of $L_1$),
and hence also $n=a_1=a_2=a_3=0$. Now, looking to the coefficients of $L_7$
and $L_8$, we arrive at a system of equations
$$\left\{ \begin{array}{l}
a + 3(a+b) a_7 = 0 \\
b + 3(a+b) a_8 = 0
          \end{array}\right.$$
However, this system has no integer solutions when $a$ and $b$ are
coprime, as seen by reducing modulo 3.
This contradicts Corollary \ref{prop:single-exc-bundle-G-inv}.

\medskip

{\it Degree 1, case $(iv)$.} In this case, $\cat{E}$ is generated by a
single rank 5 vector bundle with Chern class $-2\omega+ L_{\geq
4}$. Similarly as in the case {\it Degree 3, case $(ii)$}, the descent of
$\cat{E}$ yields a contradiction to
Corollary~\ref{prop:single-exc-bundle-G-inv}.

Thus in each case of degree $\leq 4$, the assumptions that all three
blocks can descend to simple categories leads to a contradiction and
the proof is complete.
\end{proof}

\begin{corollary}
Let $S$ be a del Pezzo surface of degree $d \leq 4$ and Picard rank one.
Then $S$ is not categorically representable in dimension 0.
\end{corollary}
\begin{proof}
By Corollary~\ref{coro:minimal-have-3-blocks}, categorical
representability of $S$ would be given by a semiorthogonal
decomposition as \eqref{eq:semiortho-for-minimal} with $\alpha_i=0$.
\end{proof}

On the other hand, given any geometrically rational surface $S$, the
line bundle $\ko_S$ (or the line bundle $\omega_S$) defines an
exceptional object in $\Db(S)$, hence we always have a nontrivial
semiorthogonal decomposition.

\begin{proposition}\label{cor:useful-formulation}
Let $S$ be a geometrically rational surface over $k$. Then there is a
semiorthogonal decomposition
$$
\Db(S)= \langle \Db(k), \cat{A}_S \rangle.
$$
If $\rho(S)=1$, then $K_0(\cat{A}_S)_{\QQ}=\QQ^{\oplus
2}$. Furthermore, if $S$ is a del Pezzo surface of degree
$\leq 4$, then there is no semiorthogonal decomposition
$$
\cat{A}_S= \langle \Db(l_1/k,\alpha_1), \Db(l_2/k, \alpha_2)\rangle
$$
with $l_i$ fields and $\alpha_i$ in $\Br(l_i)$.
\end{proposition}
\begin{proof}
The admissible subcategory $\Db(k)$ is generated by the exceptional
object $\ko_S$, hence the semiorthogonal decomposition exists. The
calculation of $K_0(\cat{A}_S)$ is straightforward, and the last
statement is a consequence of Theorem~\ref{thm:decos-of-minimal-dps}.
\end{proof}

\begin{proposition}
Let $S$ be a del Pezzo surface of degree $d \leq 4$ and Picard rank
1, and suppose that, if $d=4$, then $\ind(S) > 1$.  Let $S'$ be
birational to $S$. Then there is a semiorthogonal decomposition
$$
\cat{A}_{S'} = \langle \cat{T}, \cat{A}_S \rangle,
$$
where $\cat{T}$ is representable in dimension 0, and $\cat{T}=0$ if
and only if $S \simeq S'$.
\end{proposition}
\begin{proof}
Under the assumptions, $S$ is birationally rigid. Hence, if $S'$ is
minimal, $S' \simeq S$, and if $S'$ is not minimal, there is a blow-up
$S' \to S$.  Then we conclude by the blow-up formula.
\end{proof}

\begin{remark}
If $S$ is a del Pezzo surface of degree $\leq 4$,
notice from Table \ref{table:blocks} that there is no 3 block
decomposition with a block generated by a single exceptional line
bundle.  It follows that the category $\cat{A}_S$ does not base-change
to a category generated by two exceptional blocks.
\end{remark}

We end this section with a conjecture on the structure of the category
$\cat{A}_S$ for a del Pezzo surface of degree $\leq 4$.  The highly
non-trivial noncommutative structure of $\cat{A}_S$ should be a
reflection of the more complicated arithmetic behavior of $S$.

\begin{conjecture}\label{conj:indecomp-fot-low-deg}
If $S$ is a del Pezzo surface of degree $\leq 4$ and $\rho(S)=1$, then
the category $\cat{A}_S$ has no nontrivial semiorthogonal decomposition.
\end{conjecture}

We now describe the few known facts about $\cat{A}_S$ in
degrees $3$ and $4$.  If $S$ has degree $4$, then the anticanonical
embedding $S \to \PP^4_k$ realizes $S$ as the intersection of two
quadric hypersurfaces. The following result was shown in
\cite{auel-berna-bolo}, as a consequence of Homological Projective
Duality (see also \cite{kuznetquadrics}).

\begin{theorem}\label{thm:A-for-dp4}
Let $S$ be a del Pezzo surface of degree $4$ and $X \to \PP^1$ be
the associated pencil of quadrics of $\PP^4$ containing $S$, with
associated even Clifford algebra $\kc_0$ over $\PP^1$. Then there is a
semiorthogonal decomposition
$$
\Db(S) = \langle \Db(k), \Db(\PP^1,\kc_0)\rangle.
$$
\end{theorem}

If $S(k)\neq \varnothing$ then the quadric threefold fibration $X \to
\PP^1$ can be reduced by hyperbolic splitting (cf.\
\cite[\S1.3]{auel-berna-bolo}) to a conic bundle $Y \to \PP^1$ with
Clifford algebra $\kc_0'$, and there is an equivalence between
$\Db(\PP^1,\kc_0)$ and $\Db(\PP^1,\kc_0')$ by
\cite[Thm.~3]{auel-berna-bolo}.  We will see another (more explicit)
way to describe the orthogonal complement $\cat{A}_S$ via a conic
bundle in Appendix~\ref{subs:links-of-type-1}.

If $S$ has degree $3$, then the anticanonical embedding $S \to
\PP^3_k$ realizes $S$ as a cubic hypersurface, so that we can use
Kuznetsov's calculation (see \cite{kuznetsov:v14}).

\begin{theorem}\label{thm:A-for-dp3}
If $S$ is a del Pezzo surface of degree $3$, there is a semiorthogonal decomposition
$$\Db(S) = \langle \Db(k), \cat{A}_S \rangle,$$
with $\cat{A}_S$ a Calabi--Yau category of dimension $\frac{4}{3}$.
\end{theorem}

The category $\cat{A}_S$ can be described via Matrix Factorization as in
\cite[Thm. 40]{orlov-matrix} and via Homological Projective Duality as in \cite{BDFIK-higher-veronese}.

\section{Severi--Brauer surfaces (del Pezzo surfaces of degree 9)}
\label{subsec:dP9}

If $S$ is a del Pezzo surface of degree 9 over $k$, then $S$ is the
Severi--Brauer surface associated to a degree 3 central simple algebra
$A$ over $k$. Denote by $\alpha \in \Br(k)$ the Brauer class of $A$.

\begin{proposition}
\label{propo-sbsurf}
Let $S=\SB(A)$ be a Severi--Brauer surface. Then the following are equivalent:
\begin{enumerate}
 \item\label{propo-sbsurf.1} $S$ has a $k$-point,
 \item\label{propo-sbsurf.2} $S$ is $k$-rational,
 \item\label{propo-sbsurf.3} $S$ is categorically representable in dimension 0,
 \item\label{propo-sbsurf.4} $\cat{A}_S$ is categorically representable in dimension 0.
\end{enumerate}
\end{proposition}
\begin{proof}
It's a result of Ch\^atelet (cf.\ Example~\ref{exam:SB}) that
\ref{propo-sbsurf.1} is equivalent to \ref{propo-sbsurf.2} is
equivalent to $S \isom \PP^2$ is equivalent to the triviality of $A$
in the Brauer group.  In turn, this implies \ref{propo-sbsurf.3} and
\ref{propo-sbsurf.4} by considering the full exceptional collection
$\{ \ko, \ko(1), \ko(2) \}$ of $\Db(\PP^2)$ described by Beilinson
\cite{beilinson} and Bern{\v{s}}te{\u\i}n--Gelfand--Gelfand
\cite{BGG}.  It then suffices to prove that \ref{propo-sbsurf.3}
implies that $A$ has trivial Brauer class.

In general, a semiorthogonal decomposition
\begin{equation}
\label{equa:deco-for-bs-surface_full}
\Db(S) = \langle \Db(k), \Db(k,A), \Db(k,A^{-1}) \rangle
\end{equation} 
was constructed in \cite{bernardara:brauer_severi}, which base changes
to the semiorthogonal decomposition $\Db(\PP^2_\ksep)=\langle \ko,
\ko(1), \ko(2) \rangle$, see Example~\ref{ex:SB_decomposition}.

Assuming \ref{propo-sbsurf.3}, then by
Corollary~\ref{coro:minimal-have-3-blocks}, there are fields $l_i$ and
a semiorthogonal decomposition
\begin{equation}\label{eq:BS-catrep}
\Db(S) = \langle \Db(l_1/k), \Db(l_2/k), \Db(l_3/k) \rangle,
\end{equation}
which by Proposition~\ref{prop:karpov-nogin}, base changes to a
3-block exceptional collection, unique up to mutation and tensoring by
line bundles on $S$. Hence, up to mutations, the decomposition
\eqref{eq:BS-catrep} base changes to the decomposition
$\Db(\PP^2_\ksep) = \sod{\ko(i),\ko(i+1),\ko(i+2)}$.
Twisting by powers of the canonical bundle and performing one more
mutation, we can assume that $i=0$. Hence the decomposition
\eqref{eq:BS-catrep} is equivalent to the decomposition
\eqref{equa:deco-for-bs-surface_full}.  In particular, we get that
$\Db(k,A) \simeq \Db(l_i/k)$ for some $i=1,2$, which by
Corollary~\ref{cor:Caldararu_conj}, implies that $l_i = k$ and that
$A$ is split.
\end{proof}

\begin{table}
\centering \small
\begin{tabular}{||c||c||c|c|c|c|c||c|c|c|c|c||}
\hline\hline
$S$ & $\ind(S)$ & $A_1$ & $Z$ & $\ind$ & $c_2$ & $(V_1)_{k\sep}$ & $A_2$ & $Z$ & $\ind$ & $c_2$ & $(V_2)_{k\sep}$ \\
\hline
\hline
$\SB(A)$ & 3 & $A$ & $k$ & 3 & 3 & $\ko(1)^{\oplus 3}$ & $A^{-1}$ & $k$ & 3 & 12 & $\ko(2)^{\oplus 3}$ \\
\hline
$\PP^2_k$ & 1  & $k$ & $k$ & 1 & 0 & $\ko(1)$ & $k$ & $k$ & 1 & 0 & $\ko(2)$    \\
\hline\hline
\end{tabular}
\caption{The invariants of a Severi--Brauer surface $S$ (a del Pezzo
surface of degree 9). Here, the algebras $A_i=\End(V_i)$ are listed up
to Morita equivalence; $Z$ and $\ind$ refer to the center and
index of $A_i$; and $c_2$ refers to the
2nd Chern class of $V_i$.} 
\label{table:dp9}
\end{table}

Now we verify that the Griffiths--Kuznetsov component $\cat{GK}_S$ is
well defined.

\begin{proposition}
Let $S=\SB(A)$ be a nonrational Severi--Brauer surface.  Then the
Griffiths--Kuznetsov component is a well defined birational invariant
in the following sense: if $S_1 \dashrightarrow S$ is a birational map
then there is a semiorthogonal decomposition
$$\Db(S_1)=\sod{\cat{T},\Db(k,\alpha),\Db(k,\alpha\inv)}$$ where
$\cat{T}$ is representable in dimension 0.
\end{proposition}
\begin{proof}
If $S_1$ is minimal then $S_1 = \SB(B)$ is a Severi--Brauer surface by
Proposition~\ref{prop:isko-rigidity-of-nonrat-dps}. Amitsur's theorem
\cite{amitsur} implies that $B=A$ or $B=A\inv$.  Indeed, in the
Appendix (see Proposition~\ref{prop:amitsur-revisited}), we can show
how the decomposition of the birational map $S_1 \dashrightarrow S$
gives either $\Db(k,A) \simeq \Db(k,B)$ or $\Db(k,A) \simeq
\Db(k,B^{-1})$ only using the description of tilting bundles and their
behavior under birational maps.  Then the statement follows, possibly
up to a mutation, from the semiorthogonal decomposition
\eqref{equa:deco-for-bs-surface_full}.  If $S_1$ is not minimal, there
is a minimal model $S_1 \to S_0$, and we conclude using
Lemma~\ref{lem:blow-up_catrep2} and the first part of the proof.
\end{proof}

\begin{remark}
Recall the vector bundle $V$ of rank 3 on $S$ constructed by
Quillen~\cite[\S8.4]{quillen:higher_K-theory} (see
Example~\ref{ex:SB_decomposition}) such that $V_\ksep = \ko(1)^{\oplus
3}$.  More explicitly, if $K/k$ is a separable degree 3 extension
splitting $A$, then $V_i = \tr_{K/k}\ko_{\PP^2_K}(i)$ by
Theorem~\ref{thm:trace_from_maximal_subfield}.  We set $V_1=V^{\min}$
and $V_2 = (V\dual\tensor\omega_S\dual)^{\min}$.  Remark that
$V_{2,\ksep}$ is either $\ko(2)$ or $\ko(2)^{\oplus 3}$.  These vector
bundles are tilting bundles for the blocks $\cat{F}$ and $\cat{G}$,
respectively, and $A_1=\End(V_1)$ is Morita equivalent to $A$,
$A_2=\End(V_2)$ is Morita equivalent to $A\inv$, and and $V_i$ are
indecomposable.  We list the ranks and 2nd Chern classes of the vector
bundles $V_i$ on Table~\ref{table:dp9}.

The calculation of the second Chern classes of the vector bundles
$V_1$ and $V_2$ is easily obtained by their description over
$k\sep$. In particular, we notice that
$\ind(S)=\mathrm{gcd}(c_2(V_1),c_2(V_2))$ when $S$ is nonsplit and 
$\mathrm{gcd}(c_2(V_1^{\oplus 2}),c_2(V_2^{\oplus 2}))=\ind(S)$. 
\end{remark}

\begin{remark}
The second Chern classes of the generators of $\Db(k,A)$ and
$\Db(k,A^{-1})$ are not stable under mutations. The values $3$ and
$12$ are obtained for the specific choices of bundles listed in
Table~\ref{table:dp9}.  This same remark applies to all other degrees.
\end{remark}

\section{Involution varieties (del Pezzo surfaces of degree 8)}
\label{subsec:dP8}

If the degree of $S$ is $8$, either $S$ is an involution surface (cf.\
Example~\ref{exam:involution}), or $S$ is the blow-up of
$\PP^2_{\ksep}$ at a $k$-rational point (cf.\ Example~\ref{exam:dP78}). In the
latter case, $S$ is rational, not minimal, 
and has a semiorthogonal
decomposition
$$
\Db(S) = \langle \ko , \ko(H) , \ko(2H) , \ko(L_1)\rangle = \langle
\Db(k), \Db(k), \Db(k), \Db(k) \rangle
$$
where $H$ is the pull back of the hyperplane class from $\PP^2_k$ to
$S$ and $L_1$ is the exceptional divisor of the blow-up.  There is no
3-block decomposition in this case. 

So we focus our attention on involution surfaces.  In this case, $S$
is associated to a degree 4 central simple $k$-algebra $(A,\sigma)$
with quadratic pair.  The even Clifford algebra $C_0(A,\sigma)$ (see
Example~\ref{exam:involution}) is a
quaternion algebra over its center $l$, which is the \'etale quadratic
discriminant extension of $S$. Denote by $\alpha$ the Brauer class of
$A$ and $\gamma$ the Brauer class of $C_0(A,\sigma)$ over the
discriminant extension $l$.  When $l=k^2$, we write $\gamma =
(\gamma_+,\gamma_-) \in \Br(k^2) = \Br(k)\times \Br(k)$.

The fundamental relations for the Clifford
algebra of an algebra with involution
\cite[Thm.~9.14]{book_of_involutions} imply that $\alpha =
\cor_{l/k}\gamma \in \Br(k)$.  We record another important fact from
the algebraic theory of quadratic forms. 

\begin{lemma}
\label{lem:isotropy_dim4}
Let $S$ be an involution surface over a field $k$.  Then $S(k) \neq
\varnothing$ if and only if $\gamma \in \Br(l)$ is trivial.
\end{lemma}
\begin{proof}
If $S(k) \neq \varnothing$ then $\SB(A)(k) \neq \varnothing$, hence
$\alpha$ is split.  Also, if $\gamma \in \Br(l)$ is trivial, then
$\alpha \in \Br(k)$ is trivial by the fundamental relations.  Thus we
can reduce to the case when $(A,\sigma)$ is adjoint to a quadratic
form $q$ of dimension 4 over $k$, which in this case $q$ is isotropic
if and only if $C_0(q)$ is split over its center, i.e., $\gamma \in
\Br(l)$ is trivial, by \cite[Thm.~6.3]{knus_parimala_sridharan:rank_4}
(also see \cite[2~Thm.~14.1,~Lemma~14.2]{scharlau} in characteristic
$\neq 2$ and \cite[II~Prop.~5.3]{baeza:semilocal_rings} in
characteristic $2$).
\end{proof}

\begin{proposition}
\label{prop:involution_decomposition}
Let $S$ be an involution surface over a field $k$.  Then the following
are equivalent:
\begin{enumerate}
 \item\label{prop:involution_decomposition.1} $S$ has a $k$-point,
 \item\label{prop:involution_decomposition.2} $S$ is $k$-rational,
 \item\label{prop:involution_decomposition.3} $S$ is categorically representable in dimension 0,
 \item\label{prop:involution_decomposition.4} $\cat{A}_S$ is categorically representable in dimension 0.
\end{enumerate}
\end{proposition}

\begin{table}
\centering\small
\begin{tabular}{||c||c||c||c||c|c|c|c|c||c|c|c|c|c|c||}
\hline\hline
& $S$ & $\ind(S)$ & $\rho(S)$ & $A_1$ & $Z$ & $\ind$ & $c_2$ & $\rk$& $A_2$ & $Z$ & $\per$ & $\ind$ & $c_2$ & $\rk$  \\
\hline
\hline
8.1 & $S \subset \SB(A)$ & 4 & 1 & $C_0$ & $l$ & 2 & 4 & 4& $A$ & $k$ & 2 & 4 & 12 & 4  \\
\hline
8.2 & $\SB(B) \times \SB(B')$ & 4 & 2 & $B \times B'$ & $k^2$ & 2 & 4 & 4& $B \otimes B'$ & $k$ & 2 & 4 & 12 & 4  \\
\hline
8.3 & $S \subset \SB(A)$ & 2 & 1 & $C_0$ & $l$ & 2 & 4 & 4& $A$ & $k$ & 2 & 2 & 2 & 2  \\
\hline
8.4 & $S \subset \PP^3_k$ & 2 & 1 & $C_0$ & $l$ & 2 & 4 & 4& $k$ & $k$ & 1 & 1 & 0 & 1  \\
\hline
8.5 & $\SB(B) \times \SB(B')$ & 2 & 2 & $B \times B'$ & $k^2$ & 2 & 4 & 4& $B \otimes B'$ & $k$ & 2 & 2 & 2 & 2  \\
\hline
8.6 & $\SB(B) \times \SB(B)$ & 2 & 2 & $B \times B$ & $k^2$ & 2 & 4 & 4& $k$ & $k$ & 1 & 1 & 0 & 1  \\
\hline
8.7 & $\SB(B) \times \PP^1$ & 2 & 2 & $B \times k$ & $k^2$ & 2 & 4 & 4& $B$ & $k$ & 2 & 2 & 2 & 2 \\
\hline
8.8 & $S \subset \PP^3_k$ & 1 & 1 & $l$ & $l$ & 1 & 1 & 2& $k$ & $k$ & 1 & 1 & 0 & 1  \\
\hline
8.9 & $\PP^1 \times \PP^1$ & 1 & 2 & $k^2$ & $k^2$ & 1 & 1 & 2& $k$ & $k$ & 1 & 1 & 0 & 1  \\
\hline\hline
\end{tabular}
\caption{The invariants of an involution surface $S$ (a minimal del
Pezzo of degree 8).  Here: the algebras $\End(V_1)$ and
$\End(V_2)$ are Morita equivalent to  $A_1=C_0$ and $A_2=A$; also $Z$,
$\per$, and $\ind$ refer to the center, period, and index of $A_i$;
and $c_2$ and $\rk$ refer to the 2nd Chern class and rank of $V_i$.
Note that $(V_1)_{k\sep}$ is a direct sum of spinor bundles while
$(V_2)_{k\sep}$ is a direct sum of hyperplane classes. Also $l$ stands for
the separable quadratic discriminant extension of $k$. Recall that $S$
is rational if and only if $\ind(S)=1$. In all these cases, $S$ is
minimal. We give a brief geometric description of all cases in
Remark~\ref{rem:dP8_geom}.
 }
\label{table:dp8}
\end{table}

\begin{proof}
Our aim is to produce a semiorthogonal decomposition that base changes
to the 3-block decomposition from Table \ref{table:blocks}. In
characteristic $\neq 2$, this is a result of
Blunk~\cite[\S7]{blunk-gen-bs}, who works with tilting bundles and
does not explicitly mention semiorthogonal decompositions.  We will
give a sketch of an alternate proof that works in any characteristic. First,
under the closed embedding $S \subset \SB(A)=X$, we get $V_2$ as the
pull back of the indecomposable vector bundle on $X$ of pure type
$\ko_{\PP^3}(1)$. Then $V_2$ has rank dividing 4 and
$\End(V_2)$ is Morita equivalent to $A$. Second, the fully
faithful embedding of $\Db(k,C_0)$ can be seen as the twisted version
of Kuznetsov's result \cite{kuznetquadrics} (see
\cite[Thm.~2.2.1]{auel-berna-bolo} for the case of a quadric over a
general field). More explicitly, $\ko_{S_\ksep}(1,0) \oplus
\ko_{S_\ksep}(0,1)$ is a Galois invariant vector bundle, with the
Galois group of $l/k$ acting by switching the factors (when the
discriminant is nontrivial), and there is a
unique indecomposable vector bundle $V_1$ of this pure type by
\S\ref{subsec:classical_descent}.  Hence over $\ksep$, by comparing
with the usual decomposition (cf.\ \cite[Lemma 4.14]{kuznetquadrics},
which is non other than the 3-block decomposition in Table
\ref{table:blocks}), we find a semiorthogonal decomposition
\begin{equation}
\label{equa:deco-for-quad-surf}
\Db(S) = \langle \Db(k),\Db(k,C_0),\Db(k,A)\rangle.
\end{equation}
The fact that the endomorphism algebra of $\ko_{S_{\ksep}}(1,0)\oplus
\ko_{S_{\ksep}}(0,1)$ is Morita-equivalent to the even Clifford
algebra $C_0$ goes back to Kapranov~\cite[\S4.14]{kapranovquadric}.

Now we proceed with the proof of the equivalences. It's a classical
result (cf.\ Example~\ref{exam:involution}) that
\ref{prop:involution_decomposition.1} is equivalent to
\ref{prop:involution_decomposition.2}.  By
Lemma~\ref{lem:isotropy_dim4}, condition
\ref{prop:involution_decomposition.1} is equivalent to the triviality
of $\gamma \in \Br(l)$ (and also $\alpha \in\Br(k)$).  In particular,
\ref{prop:involution_decomposition.1} implies
\ref{prop:involution_decomposition.3} and
\ref{prop:involution_decomposition.4}, since then the semiorthogonal
decomposition just constructed is of the form $\Db(S) = \langle
\Db(k),\Db(l),\Db(k)\rangle$, with the first block generated by
$\ko_S$.  Hence both $S$ and $\cat{A}_S$ is categorically
representable in dimension 0 by Lemma~\ref{lem:0-dim=etale-algebra}.

Finally, we are left to proving that
\ref{prop:involution_decomposition.3} or
\ref{prop:involution_decomposition.4} implies the triviality of
$\gamma \in \Br(l)$.

First, assume that $S$ has Picard rank 1.  If $S$ is categorically
representable in dimension 0, then by
Corollary~\ref{coro:minimal-have-3-blocks}, there is a semiorthogonal
decomposition
\begin{equation}
\label{eq:quad-catrep}
\Db(S) = \langle \Db(l_1/k), \Db(l_2/k), \Db(l_3/k) \rangle,
\end{equation}
which by Proposition~\ref{prop:karpov-nogin}, base-changes to a
3-block exceptional collection.  Hence, up to mutation, tensoring by
line bundles on $S$, and the Weyl group action, we can assume that
$\Db(l_1/k)$ base changes to $\sod{\ko}$ and $\Db(l_3/k)$ base changes
to $\sod{\ko(1,1)}$.  Hence the decomposition \eqref{eq:quad-catrep} base
changes to the decomposition \eqref{equa:deco-for-quad-surf}. In
particular, we get that $\Db(k,C_0) \simeq \Db(l_2/k)$ and $\Db(k,A)
\simeq \Db(l_3/k)$, which by Corollary~\ref{cor:Caldararu_conj},
implies that $l=l_2$ and $C_0$ is split over $l$ and that $A$ is split
over $l_3=k$.  

Second, assume that $S$ has Picard rank $2$. In this case, we have $S
= C \times C'$ for Severi--Brauer curves $C=\SB(B)$ and
$C'=\SB(B')$, and then $C_0=B \times B'$ and $A = B \tensor B'$, cf.\
Example~\ref{exam:involution}.  Hence we have a semiorthogonal
decomposition
\begin{equation}\label{eq:C-times-C}
\cat{A}_S=\langle \Db(k,B), \Db(k,B'), \Db(k,A) \rangle.
\end{equation}
If $\cat{A}_S$ is representable in dimension $0$, then by
Corollary~\ref{coro:minimal-have-3-blocks}, there is a semiorthogonal
decomposition
\begin{equation}
\label{eq:quad-catrep_A}
\cat{A}_S = \langle \Db(l_1/k), \Db(l_2/k), \Db(l_3/k) \rangle,
\end{equation}
Since the Picard rank of $S$ is stable under base change, it follows
that $l_i=k$ for $i=1,2,3$. Thus $\cat{A}_S$ is generated by three
$k$-exceptional objects, hence $\Db(S)$ is generated by four
$k$-exceptional objects. Over $k\sep$, we can appeal to
Rudakov~\cite{rudakov-quadric} to mutate the semiorthogonal
decomposition \eqref{eq:quad-catrep_A} into the decomposition
\eqref{eq:C-times-C}. Using Theorem~\ref{thm:Caldararu_conj}, we
deduce that the Brauer classes of $A$, $B$, $B'$, and hence also
$C_0$, are trivial.
\end{proof}

We now want to show that the Griffiths--Kuznetsov component is well
defined, except possibly when $S$ has index 2 and Picard rank 1.  In
the Appendix \ref{subs:links-of-type-1}, we will see how this case
should be thought of as a conic bundle of degree 6 from the
categorical point of view.

\begin{proposition}
Let $S$ be an involution surface.  Then the Griffiths--Kuznetsov
component $\cat{GK}_S$ is well defined as a birational invariant in
the following cases.  Letting $S_1 \dashrightarrow S$ be a birational
map, we have:
\begin{itemize}
\item $\ind(S)=4$ if and only if $\alpha \in \Br(k)$ has index 4;
there is a semiorthogonal decomposition $\cat{A}_{S_1} = \langle
\cat{T}, \Db(l,\gamma), \Db(k,\alpha)\rangle$ 

\item $\ind(S)=2$ and $\rho(S)=2$ then $\gamma$ is never trivial; if $\alpha$ is
trivial then there is a semiorthogonal decomposition $\cat{A}_{S_1} =
\langle \cat{T}, \Db(l,\gamma)\rangle$ and if $\alpha$ is nontrivial
then there is a semiorthogonal decomposition $\cat{A}_{S_1} = \langle
\cat{T}, \Db(l,\gamma), \Db(k,\alpha)\rangle$
\end{itemize}
where $\cat{T}$ always denotes a category representable in dimension
0.
\end{proposition}
\begin{proof}
By Lemma~\ref{lem:isotropy_dim4}, $\gamma$ is trivial if and only if
$S(k) \neq \varnothing$.  Furthermore, we can always find a quadratic
extension $l'/l$ that splits $C_0$ over $l$.  Hence $S(l') \neq
\varnothing$.  Since $l'/k$ has degree 4, it follows that $S$ has a
closed point of degree $4$ so that $\ind(S)$ divides $4$.

Since $S \subset \SB(A)$, we have that $\ind(S)$ must be a multiple of
$\ind(\SB(A))=\ind(A)$.  In particular, $\ind(A)=4$ if and only if
$\ind(S)=4$.  Also if $\ind(A)=2$, then a generalization of Albert's
result on common splitting fields for quaternion algebras, cf.
\cite[Cor.~16.28]{book_of_involutions}, shows that there is a
quadratic extension of $k$ splitting $C_0$, hence also $A$ by the
fundamental relations, and thus $\ind(S)=2$.

Now suppose that $\ind(S)=4$.  Then both $\gamma$ and $\alpha$ are
nontrivial.  If $\rho(S)=1$ then $S$ is birationally rigid by
Proposition~\ref{prop:isko-rigidity-of-nonrat-dps}.  If $\rho(S)=2$,
then $S \isom \SB(B)\times\SB(B')$ such that $A=B\tensor B'$ has index
4 and $C_0=B\times B'$.  In \S\ref{app:conic_bundles}, we show that
$S_1$ admits a birational morphism $S_1 \to S_0$ where $S_0$ is a
conic bundle over either $\SB(B)$ or $\SB(B')$ and has the required
semiorthogonal decomposition.

Now suppose that $\ind(S)=2$ and $\rho(S)=2$, then $S_1$ admits a
birational morphism $S_1 \to S_0$ where $S_0$ is a conic bundle of
degree 8, and in the Appendix \S\ref{app:conic_bundles}, we show that
it has the required semiorthogonal decomposition.
\end{proof}

\begin{remark}
\label{rem:dP8_geom}
We now describe the geometry of the all possible cases listed in
Table~\ref{table:dp8}.  Any involution surface is minimal.  If
$\rho(S)=1$ (equivalently, $l$ is a field) then $\Pic(S)$ is generated
either by the anticanonical bundle or its square root.  In the second
case, $\ind(S)|2$ and $S \subset \PP^3$ is a quadric surface, which
can either be isotropic (case 8.8) or anisotropic (case 8.4).  If the
Picard group is generated by the anticanonical bundle, then $S \subset
\SB(A)$ is a degree 2 divisor of a Severi--Brauer threefold and
$\ind(S)=\ind(A)$ (cases 8.1 and 8.3).  In the case when $\rho(S)=2$
then $S$ is isomorphic to a product of Severi--Brauer curves $\SB(B)$
and $\SB(B')$.  In this case, $C_0 = B \times B'$ and $A=B\tensor B'$
and the possible cases are: $B$ not equivalent to $B'$ and both
nontrivial (cases 8.2 and 8.5 according to the index of $A$); $B=B'$
nontrivial (case 8.6); and $B$ nontrivial and $B'$ trivial (case 8.7);
and both $B$ and $B'$ trivial (case 8.9).  We remark that cases 8.6
and 8.7 are $k$-birational to each other for any given $B$.
\end{remark}

\begin{remark}
There is a natural vector bundle $W$ of rank 4 on $S$ such that
$W_{\ksep}=\ko(1,0)^{\oplus 2}\oplus\ko(0,1)^{\oplus 2}$.  There is
also a natural rank 4 vector bundle $U$ on $\SB(A)$ such that
$U_{\ksep} = \ko(1)^{\oplus 4}$.  We let $V_1 = W^{\min}$ and $V_2 =
U|_S^{\min}$ (recall Definition~\ref{def:reduced_part}).  These vector
bundles are tilting bundles for the blocks $\cat{F}$ and $\cat{G}$,
respectively, and have the following properties: $A_1=\End(V_1)$ is
Morita equivalent to $C_0$ and $A_2=\End(V_2)$ is Morita equivalent to
$A$; $V_1$ is indecomposable if and only if $l$ is a field; and $V_2$
is always indecomposable.  We list the ranks and 2nd Chern classes of
the vector bundles $V_i$ on Table~\ref{table:dp8}.

The calculation of the second Chern classes of the vector bundles
$V_1$ and $V_2$ is easily obtained by their description over
$k\sep$. In particular, we notice that
$\ind(S)=\mathrm{gcd}(c_2(V_1),c_2(V_2))$, except when $S$ is an
anisotropic quadric surface, in which case
$\mathrm{gcd}(c_2(V_1),c_2(V_2^{\oplus 2}))=\ind(S)$. Here, we use the
convention that $\mathrm{gcd}(a,0)=a$.
\end{remark}

\section{del Pezzo surfaces of degree $7$}
\label{subsec:dP7}

A del Pezzo surface $S$ of degree $7$ is the blow up of $\PP^2_k$
along a point of degree 2, see Example~\ref{exam:dP78}.  If the
residue field of the center of blow up is $l$, 
then there is a
semiorthogonal decomposition
$$
\Db(S) = \langle \ko, \ko(H), \ko(2H), V \rangle = \langle \Db(k),
\Db(k), \Db(k), \Db(l)\rangle
$$
where $H$ is the pull back of the hyperplane class from $\PP^2_k$ to
$S$, and $V$ is a
rank 2 vector bundle on $S$ such that $\End(V) = l$ and $V \tensor
\ksep = \ko(L_1)\oplus\ko(L_2)$, where $L_1$ and $L_2$ are the exceptional
divisors on $k_\sep$.  In particular, $S$ is $k$-rational
and is categorically representable in dimension 0 by
Lemma~\ref{lem:blow-up_catrep2}.  However, there is no 3-block
decomposition of $S$.

\section{del Pezzo surfaces of degree $6$} 
\label{subsec:dP6}

Let $S$ be a del Pezzo surface of degree 6. There is an associated
quaternion Azumaya algebra $Q$ over a cubic \'etale $k$-algebra $L$
and an associated degree 3 Azumaya algebra $B$ over a quadratic
\'etale $k$-algebra $K$.  Blunk~\cite{blunk-dp6} gives an
interpretation of these algebras in terms of a toric presentation of
$S$, building on the geometric construction of Colliot-Th\'el\`ene,
Karpenko, and Merkurjev~\cite{colliot-karpenko-merkurev}.  Let $\kappa
\in \Br(L)$ and $\beta \in \Br(K)$ denote the Brauer classes of $Q$
and $B$, respectively.  Blunk, Sierra, and Smith
\cite{blunk-sierra-smith} provide a semiorthogonal decomposition
\begin{equation}
\label{equa:deco-for-dp6}
\Db(S) = \langle \Db(k), \Db(k,Q), \Db(k,B) \rangle = \langle \Db(k), \Db(L/k,\kappa), \Db(K/k,\beta) \rangle,
\end{equation}
Our first task is to prove that the semiorthogonal decomposition
\eqref{equa:deco-for-dp6} descends the minimal three block
decomposition from \cite{karpov-nogin}. This will give an alternative
description of the tilting bundles generating the blocks.

\begin{proposition}\label{propo:mutation-for-dp6}
The base change of the semiorthogonal decomposition
\eqref{equa:deco-for-dp6} coincides with the minimal 3-block
decomposition of $\Db(S_{\ksep})$.
\end{proposition}

\begin{remark}
We can appeal to Proposition~\ref{prop:karpov-nogin} for the fact that
the semiorthogonal decomposition \eqref{equa:deco-for-dp6} coincides,
up to mutation, tensoring by a line bundle, and the Weyl group action,
with the minimal 3-block decomposition. Hence in
Proposition~\ref{propo:mutation-for-dp6}, we prove slightly more:\
using Blunk's work \cite{blunk-dp6}, we explicitly describe the
generators of the semiorthogonal components of
\eqref{equa:deco-for-dp6} that base change to the 3-block collection
from Table~\ref{table:blocks}.  Aside from being necessary for the
sequel, we believe that the direct proof clarifies the connection
between the beautiful geometry and arithmetic of del Pezzo surfaces of
degree 6 and its derived category.
\end{remark}

Before giving the proof, we recall the construction in
Blunk \cite{blunk-dp6}, and Blunk, Sierra, and Smith
\cite{blunk-sierra-smith}, of certain vector bundles on $S$. The del
Pezzo surface $S_{\ksep}$ of degree 6 over $\ksep$ is the blow-up of
$\PP^2_{\ksep}$ in three noncolinear points $p_1,p_2,p_3$. There are
six exceptional lines, coming in two pairs of three lines, say $L_1,
L_2, L_3$ and $M_1, M_2, M_3$. The intersection products are $M_i.M_j
= L_i.L_j = -\delta_{ij}$, and $M_i.L_j = \delta_{ij}$. So there is a
map $\pi: S_{\ksep} \to \PP^2_{\ksep}$, whose exceptional divisors are
the $L_i$ (with the convention that $L_i$ is over $p_i$).  The other
three exceptional lines $M_i$ are the strict transform of the lines in
$\PP^2_{\ksep}$ joining pairs of the three points (with the convention
that $M_i$ corresponds to the line not going through $p_i$).  There is
another birational morphism $\eta: S_{\ksep} \to \PP^2_{\ksep}$
contracting the $M_i$ to three points $q_1,q_2,q_3$, and sending $L_i$
to lines joining two of those three points. We end up with the
following diagram
\begin{equation}\label{diagram-for-dp6}
\xymatrix{
& S_{\ksep} \ar[rd]^{\pi} \ar[dl]_{\eta} & \\
\PP^2_{\ksep} \ar@{-->}[rr]^{\phi} & & \PP^2_{\ksep},
}
\end{equation}
where $\phi$ is the well-known Cremona involution, a birational
self-map of the projective plane of degree 2.

This description allows us to present the Picard group of $S_{\ksep}$
in a way convenient to compare the base change of the semiorthogonal
decompositions of Blunk--Sierra--Smith and Karpov--Nogin 3-block
decomposition.  Indeed, the Picard group of $S_{\ksep}$ has rank 4 and
is generated by the exceptional lines $L_i$ and $M_i$, with the
relations $L_i + M_j = L_j + M_i$. If we denote by $H = \pi^*
\ko_{\PP^2_{\ksep}}(1)$, we have that $H= L_1 + L_2 + M_3$. The
anticanonical divisor $K_{S_{\ksep}}$ is then $K_{S_{\ksep}} = L_1 +
L_2 + L_3 + M_1 + M_2 + M_3$.  On the other hand, if we denote by $H'
= \eta^* \ko_{\PP^2_{\ksep}}(1)$, we have $H' = M_1 + M_2 + L_3 =
-K_{S_{\ksep}} - H$.

To describe the semiorthogonal decomposition
\eqref{equa:deco-for-dp6}, Blunk, Sierra, and Smith construct vector
bundles over $S_{\ksep}$ that descend to $S$ in the following way
\cite{blunk-sierra-smith}.  The first one is just $\ko_{S_{\ksep}}$.

To define the second vector bundle, consider the following rank 2
vector bundles
\begin{equation}\label{presentation-J1-J2-J3}
\begin{array}{c}
J_1 = \ko(L_3 + M_2) \oplus \ko(L_2+M_3), \\
J_2 = \ko(L_1 + M_3) \oplus \ko(L_3+M_1), \\
J_3 = \ko(L_1 + M_2) \oplus \ko(L_2+M_1),
\end{array}
\end{equation}
on $S_{\ksep}$.  The presentation \eqref{presentation-J1-J2-J3} shows
that $\ol{J} = J_1 \oplus J_2 \oplus J_3$ is Galois invariant.
Blunk, Sierra, and Smith assert that $\ol{J}$ descends to a vector
bundle $J$ of rank $6$ on $S$ and they consider $Q = \End(J)$.  On
$S_{\ksep}$, we remark that $J_i = \ko(H-L_i)^{\oplus 2}$. Thus, by
base change, we get that $Q \tensor \ksep = \End(\ko(H-L_1)^{\oplus
2} \oplus \ko(H-L_2)^{\oplus 2} \oplus \ko(H-L_3)^{\oplus 2})$,
which is Morita equivalent to $\End\bigl(\bigoplus_{i=1}^3
\ko(H-L_i)\bigr)$.

To define the third vector bundle, consider the two rank 3 vector
bundles
\begin{equation}\label{presentation-I1-I2}
\begin{array}{c}
I_1 = \ko(L_1 + M_2 + M_3) \oplus \ko(M_1+L_2+M_3) \oplus \ko(M_1+M_2+L_3), \\
\\
I_2 = \ko(L_1 + L_2 + M_3) \oplus \ko(L_1+M_2+L_3) \oplus \ko(M_1+L_2+L_3),
\end{array}
\end{equation}
on $S_{\ksep}$.  The presentation \eqref{presentation-I1-I2} shows
that the sum $\overline{I}:= I_1 \oplus I_2$ is Galois invariant.
Blunk, Sierra, and Smith assert that $\ol{I}$ descends to a vector
bundle $I$ of rank $6$ on $S$ and they consider $B = \End(I)$.  On
$S_{\ksep}$, we remark that $I_1 = \ko(H')^{\oplus 3}=
\ko(-K_{S_{\ksep}}-H)^{\oplus 3}$, and $I_2 = \ko(H)^{\oplus 3}$. In
particular, by base change, we get that $B \tensor \ksep
= \End(\ko(H)^{\oplus 3} \oplus \ko(-K_{S_{k\sep}}-H)^{\oplus 3})$,
which is Morita equivalent to $\End(\ko(H) \oplus
\ko(-K_{S_{k\sep}}-H))$.

\begin{proof}[Proof of Proposition~\ref{propo:mutation-for-dp6}]
Let us now recall the construction of the 3 block decomposition over
$S_{\ksep}$ described by Karpov and Nogin \cite{karpov-nogin}. We
provide a slightly different presentation.  Consider the birational
morphism $\pi: S_{\ksep} \to \PP^2_{\ksep}$, which is the blow up of
three points with exceptional divisors $L_1$, $L_2$ and $L_3$. A
semiorthogonal decomposition of $\Db(S_{\ksep})$ is given by Orlov's
formula (see \cite{orlovprojbund}) as follows:
$$
\Db(S_{\ksep})= \langle \pi^* \Db(\PP^2_{\ksep}), \ko_{L_1},
\ko_{L_2}, \ko_{L_3} \rangle.
$$
Choosing the full exceptional collection $\{ \ko(-1),\ko,\ko(1) \}$ on
$\PP^2_{\ksep}$, we get the semiorthogonal decomposition
$$
\Db(S_{\ksep})= \langle \ko(-H),\ko,\ko(H),\ko_{L_1}, \ko_{L_2},
\ko_{L_3} \rangle.
$$
Mutating $\ko(-H)$ to the left with respect to the whole orthogonal
complement we get
$$
\Db(S_{\ksep})= \langle \ko,\ko(H),\ko_{L_1}, \ko_{L_2}, \ko_{L_3},
\ko(-K_{S_{\ksep}}-H) \rangle,
$$
using \cite[Prop.~3.6]{bondal_kapranov:reconstructions}. Now we mutate
the three exceptional objects $\ko_{L_i}$ to the left with respect to
$\ko(H)$. An easy calculation with the evaluation sequences for $L_i$
gives
\begin{equation}
\label{equa-3-block-of-KN}
\Db(S_{\ksep}) = \langle \ko, \ko(H-L_1), \ko(H-L_2), \ko(H-L_3), \ko(H), \ko(-K_{S_{\ksep}}-H) \rangle.
\end{equation}
The decomposition \eqref{equa-3-block-of-KN} is a mutation of the
3-block decomposition \cite[(3)]{karpov-nogin}. The latter is indeed
obtained mutating the
three exceptional objects $\ko_{L_i}$ to the right with respect to
$\ko(-K_{S_{\ksep}}-H)=\ko(2H-L_1-L_2-L_3)$, as Karpov and Nogin do.
The presentation \eqref{equa-3-block-of-KN} allows the following
description of the three blocks:
$$
\cat{E}=\langle \ko \rangle, \quad
\cat{G}= \langle \ko(H-L_1), \ko(H-L_2), \ko(H-L_3) \rangle, \quad
 \cat{F}= \langle \ko(H), \ko(-K_{S_{\ksep}}-H) \rangle.
$$
So, the block $\cat{E}$ corresponds to the first component of
\eqref{equa:deco-for-dp6}. Recall that $\End\bigl(\bigoplus_{i=1}^3
\ko(H-L_i)\bigr)$ is Morita equivalent to $Q \tensor \ksep$,
and that $\End(\ko(H) \oplus \ko(-K_{S_{\ksep}}-H))$ is Morita equivalent to
$B \tensor \ksep$. The claim follows now by Proposition \ref{prop:exc-coll-and-dg-algebras}.
\end{proof}

\begin{table}
\centering\tiny
\begin{tabular}{||c||c||c||c||c|c|c|c|c||c|c|c|c|c||}
\hline\hline
& $S$ & $\ind(S)$ & $\rho(S)$ & $A_1$ & $Z$ & $\ind$ & $c_2$ & $\rk$ & $A_2$ & $Z$ & $\ind$ & $c_2$ & $\rk$ \\
\hline
\hline
6.1& $S \subset R_{K/k}\SB(B)$ & 6 & 1 & $Q$ & $L$ & 2 & 12 & 6 & $B$ & $K$ & 3 & 24 & 6 \\
\hline
6.2& $S \subset R_{K/k}\SB(B)$ & 3 & 1 & $L$ & $L$ & 1 & 3 & 3 & $B$ & $K$ & 3 & 24 & 6 \\
\hline
6.3& $S \subset \SB(A) \times \SB(A^{-1})$ & 3 & 2 & $L$ & $L$ & 1 & 3 & 3 & $A \times A^{-1}$ & $k^2$ & 3 & 24 & 6 \\
\hline
6.4& $S  \subset R_{K/k}\PP^2$ & 2 & 1 & $Q$ & $L$ & 2 & 12 & 6 & $K$ & $K$ & 1 & 2 & 2 \\
\hline
6.5& $S  \subset R_{K/k}\PP^2$ & 2 & 2 & $Q'' \times Q'$ & $k\times L'$ & 2 & 12 & 6 & $K$ & $K$ & 1 & 2 & 2 \\
\hline
6.6& $S  \subset R_{K/k}\PP^2$ & 2 & 2 & $k \times Q'$ & $k\times L'$ & 2 & 8 & 5 & $K$ & $K$ & 1 & 2 & 2 \\
\hline
6.7& $S  \subset R_{K/k}\PP^2$ & 2 & 3 & $Q'\times Q''\times Q'''$ & $k^3$ & 2 & 12 & 6 & $K$ & $K$ & 1 & 2 & 2 \\
\hline
6.8& $S  \subset R_{K/k}\PP^2$ & 2 & 3 & $k\times Q' \times Q'$ & $k^3$ & 2 & 8 & 5 & $K$ & $K$ & 1 & 2 & 2 \\
\hline
6.9& $S  \subset R_{K/k}\PP^2$ & 1 & 1 & $L$ & $L$ & 1 & 3 & 3 & $K$ & $K$ & 1 & 2 & 2 \\
\hline
6.10& $S  \subset R_{K/k}\PP^2$ & 1 & 2 & $k \times L'$ & $k\times L'$ & 1 & 3 & 3 & $K$ & $K$ & 1 & 2 & 2 \\
\hline
6.11& $S  \subset \PP^2 \times \PP^2$ & 1 & 2 & $L$ & $L$ & 1 & 3 & 3 & $k^2$ &
$k^2$ & 1 & 2 & 2 \\
\hline
6.12& $S  \subset R_{K/k}\PP^2$ & 1 & 3 & $k^3$ & $k^3$ & 1 & 3 & 3 & $K$ & $K$ & 1 & 2 & 2 \\
\hline
6.13& $S  \subset \PP^2 \times \PP^2$ & 1 & 3 & $k \times L'$ & $k\times L'$ & 1 & 3 &
3 & $k^2$ & $k^2$ & 1 & 2 & 2 \\
\hline
6.14& $S \subset \PP^2 \times \PP^2$ & 1 & 4 & $k^3$ & $k^3$ & 1 & 3 & 3 & $k^2$ & $k^2$ & 1 & 2 & 2 \\
\hline\hline
\end{tabular}
\caption{The invariants of a del Pezzo surface $S$ of degree 6. Here,
the algebras $\End(V_1) = A_1=Q$ and $\End(V_2) = A_2=B$ up
to Morita equivalence; $l_1=Z(A_1)=L$ and
$l_2=Z(A_2)=K$ are separable cubic and quadratic extensions of $k$;
the columns $Z$ and $\ind$ refer to the center and index of $A_i$; and
the columns $c_2$
and $\rk$ refer to the 2nd Chern class and rank of $V_i$. Note that
$(V_1)_{k\sep}$ is a direct sum of $\oplus(\ko(H-L_i))$, while
$(V_2)_{k\sep}$ is a direct sum of $\ko(H)\oplus \ko(H')$. 
Recall that $S$ is rational if and only if $\ind(S)=1$, see 
\cite[\S2]{coray:del_Pezzo}. See Remark~\ref{rem:dP6_geom} for a
geometric description of each case.}
\label{table:dp6}
\end{table}

\begin{remark}
\label{rem:dP6_splitting_conditions}
A consequence of \cite[Thm.~4.1]{blunk-dp6} is that $B$ comes with a
natural $K/k$-unitary involution, equivalently, the corestriction of
$B$ from $K$ to $k$ is split.  This involution on $B$ was already
constructed by Colliot-Th\'el\`ene, Karpenko, and Merkurjev
\cite{colliot-karpenko-merkurev}.  Furthermore, the corestriction of
$Q$ from $L$ to $k$ is split.  Also, $B$ is split by $L$ and $Q$ is
split by $K$.  Otherwise, any choices of $K$ and $L$ are possible and
any choices of algebras $B/K$ and $Q/L$ are possible, subject to the above
restrictions, see \cite[Thm.~2.2]{blunk-dp6}.
\end{remark}

\begin{remark}\label{rmk:morita-eq-for-dp6s}
Given a del Pezzo surface $S$ of degree 6, Blunk constructs the triple
$(Q,B,KL)$ and shows that a toric presentation of $S$ is
uniquely determined by the equivalence class under pairwise
$L$-isomorphisms of $Q$ and $K$-isomorphisms of $B$, see
\cite[Thm. 2.4]{blunk-dp6}. On the other hand, a consequence of
Blunk's work is that the isomorphism class of $S$ is uniquely
determined by the equivalence class under pairwise $k$-isomorphisms of
$Q$ and $B$, see \cite[Prop. 3.2]{blunk-dp6}.  By
Theorem~\ref{thm:Caldararu_conj}, the semiorthogonal decomposition
\eqref{equa:deco-for-dp6} determines $Q$ and $B$ up to pairwise
$k$-linear Morita equivalence, hence $k$-isomorphism, since the
algebras involved are semi-simple of finite rank.  We conclude that
the semiorthogonal decomposition \eqref{equa:deco-for-dp6} identifies
the isomorphism class of $S$.
\end{remark}

Now we prove that rationality is equivalent to categorical
representability in dimension 0.

\begin{proposition}
\label{prop-one-dir-for-dp6}
Let $S$ be a del Pezzo surface $S$ of degree 6. The following are equivalent:
\begin{enumerate}
 \item \label{prop-one-dir-for-dp6.1} $S$ has a $k$-rational point
 \item \label{prop-one-dir-for-dp6.2} $S$ is $k$-rational
 \item \label{prop-one-dir-for-dp6.3} $S$ is categorically representable in dimension 0
 \item \label{prop-one-dir-for-dp6.4} $\cat{A}_S$ is representable in dimension 0
\end{enumerate}
\end{proposition}
\begin{proof}
By \cite[\S2]{coray:del_Pezzo}, $S$ is rational if and only if $S(k)
\neq \varnothing$ if and only if $\ind(S)=1$.  By
\cite[Cor.~3.5]{blunk-dp6}, $S(k) \neq \varnothing$ is equivalent to
the triviality of the Brauer classes $\kappa$ and $\beta$ and so
\eqref{equa:deco-for-dp6} becomes a semiorthogonal decomposition
$$
\Db(S) = \langle \Db(k), \Db(L/k), \Db(K/k) \rangle.
$$
Hence \ref{prop-one-dir-for-dp6.1} is equivalent
to \ref{prop-one-dir-for-dp6.2}, which
implies \ref{prop-one-dir-for-dp6.3}.

On the other hand, suppose that $S$ has Picard rank 1. If $S$ is
categorical representable in dimension 0, then there is a
semiorthogonal decomposition
\begin{equation}\label{eq:dp6-catrep}
\Db(S) = \langle \Db(l_1/k), \Db(l_2/k), \Db(l_3/k) \rangle,
\end{equation}
base-changing to a 3-block exceptional collection over $\ksep$. Hence,
up to mutation, the decomposition \eqref{eq:dp6-catrep} equals the
decomposition \eqref{equa:deco-for-dp6} and hence $\cat{A}_S$ is
representable in dimension 0 and both the Brauer classes of $Q$ and
$B$ are trivial, hence $S(k) \neq \varnothing$ by
\cite[Cor.~3.5]{blunk-dp6}. Thus \ref{prop-one-dir-for-dp6.3} implies
\ref{prop-one-dir-for-dp6.4}, which implies
\ref{prop-one-dir-for-dp6.1}. If the Picard rank of $S$ is $>1$, then
$S$ is not minimal and we can consider its minimal model $S'$, which
is a del Pezzo surface of degree $\geq 7$.  We conclude in this case
by appealing to the results from the previous sections.
\end{proof}

\begin{remark}
Colliot-Th\'el\`ene, Karpenko, and Merkurjev
\cite[Rem.~4.5]{colliot-karpenko-merkurev} provide a geometric
argument to show that the splitting of $K$ implies the non-minimality
of $S$.

This can also be seen via the description of $B = \End(I)$,
where $I$ is the vector bundle constructed above.  Indeed, the
splitting of $K$, i.e., $B \isom A \times A\op$, means that both
$H^{\oplus 3}$ and $(H')^{\oplus 3}$ descend to vector bundles of rank
3 on $S$, hence both $\overline{H}$ and $\ol{H}'$ are Galois invariant
line bundles, with Brauer obstruction to descent being the Brauer
classes of $A$ and $A\op$, respectively.  Hence (cf.\
\cite[Prop.~8]{stix:period-index}) we get birational morphisms $S \to
\SB(A)$ and $S \to \SB(A^{-1})$ (compare with
\cite[Rem.~4.5]{colliot-karpenko-merkurev}).  As an interesting
consequence, the birational Cremona involution $\phi$ from diagram
\eqref{diagram-for-dp6} descends to a $k$-birational map $\SB(A)
\dashrightarrow \SB(A^{-1})$.
\end{remark}

\begin{lemma}
\label{lem:effective_index_dP6}
Let $S$ be a del Pezzo surface of degree 6.  The index of $X$ divides
$6$ and there is always a closed point whose degree equals the index. 
\end{lemma}
\begin{proof}
There is always a closed point of degree 6 on $S$.  Indeed, the 6
points of intersection of the hexagon of exceptional curves defined
over $\ksep$ is Galois invariant, hence defines a point of degree
6. Since the index is the greatest common divisor of the degrees of
all closed points, the index of $S$ divides $6$.  

Considering the Galois action on the hexagon of exceptional curves
defined over $\ksep$, there exists a quadratic extension $K/k$ (resp.\
cubic extension) such that $S_K$ has a triple (resp.\ pair) of skew
exceptional curves, cf.\ \cite[Lemme~1]{colliot-thelene:dP6}.  Blowing
down, we have $S_K \to S_0$ where $S_0$ is a Severi--Brauer surface
defined over $K$ and also $S_L \to S_1$ where $S_1$ is a del Pezzo of
degree 8 defined over $L$. In this later case, a lattice computation
over $\ksep$ shows that $S_1$ is actually an involution surface.  If
we assume that $S$ has index 2, then so does $S_0$, hence $S_0 \isom
\PP^2_K$.  In particular, $S(K) \neq \varnothing$ so $S$ has a closed
point of degree 2.  Similarly, if we assume that $S$ has index 3, then
so does $S_1$, hence $S_1(L) \neq 0$ by Springer's theorem.  In
particular, $S(L) \neq \varnothing$ so $S$ has a closed point of
degree 3.  Finally, if we assume that $S$ has index 1, then it must
have a point of degree relatively prime to $6$, hence $S(k) \neq
\varnothing$ by \cite{coray:del_Pezzo}.
\end{proof}

We now want to show that the Griffiths--Kuznetsov component is well
defined.  This also gives a strengthening of \cite[Lemma~4.6]{colliot-karpenko-merkurev}.

\begin{proposition}
Let $S$ be a del Pezzo surface of degree 6.  Then the
Griffiths--Kuznetsov component $\cat{GK}_S$ is well defined as a
birational invariant as follows.  Letting $S_1 \dashrightarrow S$ be a
birational map, we have:
\begin{itemize}
\item $\ind(S)=6$ if and only if both $\kappa$ and $\beta$ are
nontrivial; there is a semiorthogonal decomposition $\cat{A}_{S_1} = \langle
\cat{T}, \Db(L/k,\kappa), \Db(K/k,\beta)\rangle$  
  
\item $\ind(S)=3$ if and only if $\kappa$ is trivial and $\beta$ is
nontrivial; there is a semiorthogonal decomposition $\cat{A}_{S_1} =
\langle \cat{T}, \Db(K/k,\beta)\rangle$

\item $\ind(S)=2$ if and only if $\kappa$ is nontrivial and $\beta$
trivial; there is a semiorthogonal decomposition $\cat{A}_{S_1} =
\langle \cat{T}, \Db(L/k,\kappa)\rangle$
\end{itemize}
where $\cat{T}$ always denotes a category representable in dimension
0.
\end{proposition}
\begin{proof}
We can reduce to the case when $S$ is minimal, equivalently, has
Picard rank 1.  Indeed, if $S$ is not minimal (and not rational), then
there is a birational morphism $S \to S_0$ where $S_0$ is either a
nonsplit involution surface (when $\ind(S)=2$) or $S'$ is a nonsplit
Severi--Brauer surface (when $\ind(S)=3$).  We have already treated
the Griffiths--Kuznetsov component in these cases, see
\S\ref{subsec:dP9} and \S~\ref{subsec:dP8}.  For the interpretations
of $\kappa$ and $\beta$ in the nonminimal cases, see Remark~\ref{rem:dP6_geom}.

First, we remark that if $\beta$ is trivial then $S(K) \neq
\varnothing$, hence $\ind(S)|2$ by
\cite[Rem.~4.5]{colliot-karpenko-merkurev}.  Similarly, we will argue
that if $\kappa$ is trivial, then $S(L) \neq \varnothing$, hence
$\ind(S)|3$.  Indeed, $\beta_L$ is split by
Remark~\ref{rem:dP6_splitting_conditions}, so assuming that $\kappa$
is split implies that $S_L$ is categorically representable in dimension
0, hence is rational by Proposition~\ref{prop-one-dir-for-dp6}.  As a
consequence, if $\ind(S)=6$, then both $\beta$ and $\kappa$ are
nontrivial and also $S$ is birationally rigid by
Proposition~\ref{prop:isko-rigidity-of-nonrat-dps}.

Now assume that $S$ has index 3.  Then $S$ has a degree 3 point $x$ by
Lemma \ref{lem:effective_index_dP6}, and we can appeal to the
description of elementary links detailed in
\S\ref{subs:Sarkisov-deg6}. There is only one
elementary link $\phi:S \dashrightarrow S'$ defined by degree
three points $x$ on $S$ and $x'$ on $S'$.  Let $X$ be the del Pezzo
surface of degree $3$ obtained as a resolution of $\phi$ with the
following diagram (over $k\sep$):
$$
\xymatrix{
& X \ar[dl]_{\sigma} \ar[dr]^\tau &\\
S \ar[d]_{\sigma_0}\ar@{-->}[rr]^{\phi} & & S' \ar[d]^{\tau_0} \\
\PP^2 \ar@{-->}[rr]^{\phi_0}& & \PP^2 
}
$$
Denote by $G = \tau_0^*\ko_{\PP^2}(1)$ and $F_i$ be the exceptional
divisors of $\tau_0$ and by abuse of notations, we write $L$ for
$\sigma^*L$ and $G$ for $\tau^*G$.  Finally, $L_4, L_5, L_6$ are the
exceptional divisors of $\sigma$ and $F_4,F_5,F_6$ the exceptional
divisors of $\tau$ (over $\ksep$ we are blowing up three points). 
As calculated in Appendix \ref{app:links1}, we have that
\begin{equation}\label{eq:hom-syst-for-M36}
\begin{array}{rl}
G=& 5H - \sum_{j=1}^6 2E_j \\
F_i =& 2H - \sum_{j \neq i+3} E_j \,\,\,\text{  for } i=1,2,3\\ 
F_i =& 2H - \sum_{j \neq i-3} E_j \,\,\,\text{  for } i=4,5,6.
\end{array}
\end{equation}

We
perform a series of mutations in $\Db(X)$ over $k\sep$ in order to
compare $\End(I)$, $\End(J)$, $\End(I')$, $\End(J')$ and the residue
fields $k(x)$ and $k(x')$.  Let us choose the following 3-block
semiorthogonal decomposition
\begin{equation}\label{equa1}
\Db(S') = \langle \ko_{S'} | \ko(G), \ko(-K_{S'}-G) |
\ko(-K_{S'}-G+F_1),\ko(-K_{S'}-G+F_2),\ko(-K_{S'}-G+F_3) \rangle,
\end{equation}
which is the original 3-block decomposition of Karpov--Nogin
\cite{karpov-nogin}. We Denote $\cat{E}'$, $\cat{F}'$ and $\cat{G}'$ the
three blocks in the order given in \eqref{equa1} as in
Table~\ref{table:blocks}. In particular, $\cat{E}'$ descends to
$\Db(k)$, $\cat{F}'$ descends to $\Db(k,B')$ and $\cat{G}'$ descends
to $\Db(k,Q')$. We mutate the second and the third block of
\eqref{equa1} to the left with respect to $\ko_{S'}$ to obtain
\begin{equation}\label{equa2}
\Db(S') = \langle  \ko(K_{S'}+G), \ko(-G) |
\ko(-G+F_1),\ko(-G+F_2),\ko(-G+F_3) | \ko_{S'} \rangle,
\end{equation}
since the top left mutation amounts to tensoring with $\omega_{S'}=\ko(K_{S'})$.
Via the blow-up $\tau$, we get the following four-block decomposition of $X$:
\begin{equation}\label{equa3}
\Db(X) = \langle \ko(K_{S'}+G), \ko(-G) |
\ko(-G+F_1),\ko(-G+F_2),\ko(-G+F_3) | \ko_X | \ko_{F_4}, \ko_{F_5}, \ko_{F_6}
\rangle.
\end{equation}
The decomposition \eqref{equa3} is made of the four blocks $\cat{F}'$,
$\cat{G}'$, $\cat{E}'$, and a new block $\cat{H}'$ arising from the
blow up, descending to $\Db(k(x')/k)$. Finally, if we mutate
$\cat{H}'$ to the left with respect to $\ko_X$, the evaluation
sequence holds:
\begin{equation}\label{equa4}
\Db(X)=\langle \ko(K_{S'}+G), \ko(-G) |
\ko(-G+F_1),\ko(-G+F_2),\ko(-G+F_3) | \ko(-F_4), \ko(-F_5), \ko(-F_6) | \ko_X \rangle.
\end{equation}
Now we rewrite all line bundles in terms of $H$ and $L_i$ using
the relations \eqref{eq:hom-syst-for-M36}.
This makes \eqref{equa4} into:
\begin{equation}\label{equa5}
\begin{array}{rl}
\Db(X) \, = & \langle \ko(2K_X + 2H-L_1-L_2-L_3), \ko(2K_X+H) | \ko(K_X+L_4),
\ko(K_X+L_5), \ko(K_X +L_6) | \\
& \qquad\qquad\qquad \ko(K_X+H-L_1), \ko(K_X+H-L_2),  \ko(K_X+H-L_3) | \ko_X
\rangle 
\end{array}
\end{equation}
where the blocks are now $\cat{F}'$, $\cat{G}'$, $\cat{H}'$, and
$\cat{E}'$.  We apply the autoequivalence $\tensor \omega_X\dual$ and
mutate the first block $\cat{F}'$ to the right with respect to its
right orthogonal to obtain:
\begin{equation}\label{equa6}
\Db(X) = \langle \ko(L_4), \ko(L_5), \ko(L_6) | \ko(H-L_1),
\ko(H-L_2), \ko(H-L_3) | \ko(K_X) | \ko(-\sigma^*K_S - H), \ko(H)
\rangle,
\end{equation}
where the blocks are now $\cat{G}'$, $\cat{H}'$, $\cat{E}'$, and
$\cat{F}'$. 

Now consider the semiorthogonal decomposition
\begin{equation}\label{equa7}
\Db(S) = \langle \ko_S | \ko(H-L_1), \ko(H-L_2), \ko(H-L_3) | \ko(H), \ko(-K_{S}-H) \rangle,
\end{equation}
as in Proposition \ref{propo:mutation-for-dp6}. This decomposition has
blocks $\cat{E}$ (descending to $\Db(k)$), $\cat{G}$ (descending to
$\Db(k,Q)$) and $\cat{F}$ (descending to $\Db(k,B)$) in the order
presented in \eqref{equa7}, as in Table~\ref{table:blocks}.  This
presentation provides, via $\sigma^*$, equivalences $\cat{F}\simeq
\cat{F}'$, whence $\Db(k,B) \simeq
\Db(k,B')$. On the other hand, $\cat{G} \simeq \cat{H}'$, whence
$\Db(k,Q) \simeq \Db(k(x')/k)$.  By symmetry, we have $\Db(k,Q')
\simeq \Db(k(x)/k)$. Using Theorem~\ref{thm:Caldararu_conj}, we
conclude that $\kappa \in \Br(L)$ is trivial.  If in addition $\beta
\in \Br(K)$ is trivial, then $S(k) \neq \varnothing$ and $S$ is
rational.  Otherwise, if $\beta \in \Br(K)$ is nontrivial, the
category $\Db(K/k,\beta)$ is a birational invariant.  Indeed, in this
case, the index of $S$ is 3 so can have no point of degree 2 and all
birational maps $S \dashrightarrow S'$ decompose into elementary links
of type $M_{6,3}$.  We have proved that $\ind(S)|3$ implies that
$\kappa$ is trivial.

Now we handle the case where $S$ has index $2$. Then $S$ has a degree
2 point $x$ (cf.\ Lemma \ref{lem:effective_index_dP6}),
and we can appeal to the description of elementary links
detailed in \S \ref{subs:Sarkisov-deg6}. Let $\phi:S \dashrightarrow
S'$ be the elementary link defined by the two degree two points $x$ on
$S$ and $x'$ on $S'$.  Let $X$ be the del Pezzo surface of degree $4$
obtained as a resolution of $\phi$ with the following diagram (over
$k\sep$):
$$\xymatrix{
& X \ar[dl]_{\sigma} \ar[dr]^\tau &\\
S \ar[d]_{\sigma_0}\ar@{-->}[rr]^{\phi} & & S' \ar[d]^{\tau_0} \\
\PP^2 \ar@{-->}[rr]^{\phi_0}& & \PP^2 }$$
Denote by $G = \tau_0^*\ko_{\PP^2}(1)$ and $F_i$ be the exceptional
divisors of $\tau_0$ and by abuse of notations, we write $L$ for
$\sigma^*L$ and $G$ for $\tau^*G$.  Finally, $L_4$ and $L_5$ are the
exceptional divisors of $\sigma$ and $F_4$ and $F_5$ the exceptional
divisors of $\tau$ (over $\ksep$ we are blowing up two points).  As in
the index 3 case, we perform a series of mutations in $\Db(X)$ over
$k\sep$ in order to compare $\End(I)$, $\End(J)$, $\End(I')$,
$\End(J')$ and the residue fields $k(x)$ and $k(x')$.  We perform the
same first mutation as in the index 3 case, and consider the
decomposition \eqref{equa2} of $\Db(S'_{\ksep})$.  Via the blow-up
$\tau$, we get the following four-block decomposition of $X$:
\begin{equation}\label{equa3t}
\Db(X) = \langle \ko(K_{S'}+G), \ko(-G) |
\ko(-G+F_1),\ko(-G+F_2),\ko(-G+F_3) | \ko_X | \ko_{F_4}, \ko_{F_5}
\rangle.
\end{equation}
The decomposition \eqref{equa3t} is made of the four blocks $\cat{F}'$,
$\cat{G}'$, $\cat{E}'$, and $\cat{H}'$. The latter arises from the
blow-up and descends to $\Db(k(x')/k)$. Finally, if we mutate
$\ko_{F_4}$ and $\ko_{F_5}$ to the left with respect to $\ko_X$, the
evaluation sequence holds:
\begin{equation}\label{equa4t}
\Db(X)=\langle \ko(K_{S'}+G), \ko(-G) | \ko(-G+F_1),\ko(-G+F_2),\ko(-G+F_3) | \ko(-F_4), \ko(-F_5) | \ko_X \rangle.
\end{equation}
Now we rewrite all line bundles in terms of $H$ and the $L_i$ using
the relations in \S\ref{subs:Sarkisov-deg6} for the link $M_{6,2}$:
\begin{align*}
G {} & = 3H-L_1-L_2-L_3-L_4-2L_5 \\
F_i {} & = H-L_i-L_5, \qquad i=1,\dotsc,4\\
F_5 {} & = 2H - L_1-L_2-L_3-L_4-L_5
\end{align*}
This makes \eqref{equa4t} into:
\begin{equation}\label{equa5t}
\begin{array}{rl}
\Db(X) \; = & \langle \ko(K_X + L_4), \ko(K_X+L_5) | \ko(K_X+H-L_1),
\ko(K_X+H-L_2), \ko(K_X + H-L_3) | \\
{}& \qquad\qquad \ko(-H+L_4+L_5), \ko(K_X+H) | \ko_X
\rangle, 
\end{array}
\end{equation}
where the blocks are $\cat{F}'$, $\cat{G}'$, $\cat{H}'$, and
$\cat{E}'$. Now we mutate the left orthogonal to $\ko_X$ to the right
to obtain:
\begin{equation}\label{equa6t}
\Db(X) = \langle \ko_X | \ko(L_4), \ko(L_5) | \ko(H-L_1),
\ko(H-L_2), \ko(H-L_3) | \ko(-\sigma^*K_S-H), \ko(H)
\rangle 
\end{equation}
where the blocks are $\cat{E}'$, $\cat{F}'$, $\cat{G}'$, and $\cat{H}'$. Now consider the semiorthogonal decomposition
\begin{equation}\label{equa7t}
\Db(S) = \langle \ko_S | \ko(H-L_1), \ko(H-L_2), \ko(H-L_3) | \ko(H), \ko(-K_{S}-H) \rangle,
\end{equation}
as in Proposition \ref{propo:mutation-for-dp6}. This decomposition has
blocks $\cat{E}$ (descending to $\Db(k)$), $\cat{G}$ (descending to
$\Db(k,Q)$) and $\cat{F}$ (descending to $\Db(k,B)$) in the order
presented in \eqref{equa7t}, as in Table~\ref{table:blocks}. 

This presentation provides, via $\sigma^*$, equivalences
$\cat{E}\simeq \cat{E}'$, and $\cat{G}\simeq \cat{G}'$, whence
$\Db(k,Q) \simeq \Db(k,Q')$. On the other hand, $\cat{F} \simeq
\cat{H}'$, whence $\Db(k,B) \simeq \Db(k(x')/k)$.  By symmetry, we
have $\Db(k,B') \simeq \Db(k(x)/k)$. Using
Theorem~\ref{thm:Caldararu_conj}, we conclude that $\beta \in \Br(K)$
is trivial.  If in addition $\kappa \in \Br(L)$ is trivial, then $S(k)
\neq \varnothing$ and $S$ is rational.  Otherwise, if $\kappa \in
\Br(L)$ is nontrivial, the category $\Db(L/k,Q)$ is a birational
invariant.  Indeed, in this case, the index of $S$ is 2 so can have no
point of degree 3 and all birational maps $S \dashrightarrow S'$
decompose into elementary links of type $M_{6,2}$.  We have proved
that $\ind(S)|2$ implies that $\beta$ is trivial.
\end{proof}

\begin{remark}
\label{rem:dP6_geom}
We now describe the geometry of the all possible cases listed in
Table~\ref{table:dp6}.  In particular, for the nonminimal cases, we
describe how the classes $\kappa$ and $\beta$ are related to the
Brauer classes arising from their minimal model. 

Cases 6.1, 6.2, 6.4, and 6.9 are minimal since they have Picard rank
1.  Cases 6.9--6.14 are rational by Proposition~\ref{prop-one-dir-for-dp6}.

Case 6.3 is the blow up of the Severi--Brauer surface $\SB(A)$ in a
point $x$ of degree 3 with residue field $L$, via the natural
projection of $S \subset \SB(A)\times\SB(A\inv)$.  In fact, $S$
resolves the standard Cremona quadratic transformation $\SB(A)
\dashrightarrow \SB(A\inv)$.  Recall that $B = \End(I)
= \End(I_1\oplus I_2)$ and notice that $I_1$ and $I_2$ are the
pull-backs of the natural rank 3 vector bundles on $\SB(A)$ and
$\SB(A\inv)$, respectively.  Over $\ksep$, the block $\cat{G}$ is
obtained by right mutation of the category generated by the
exceptional divisors of $S_{\ksep} \to \PP^2_{\ksep}$.  From this, we
see that $\Db(L/k,\kappa) \simeq \Db(k(x)/k)$.  Hence $\kappa$ is
trivial and $k(x)=L$.

We now argue that cases 6.5--6.8 are blow ups of an involution variety
associated to $(A,\sigma)$ in a point $x$ of order 2 with residue
field $K$, where the center of $C_0$ is isomorphic to $L'$ (or
$k^2$). The minimal model $\pi : S \to S_0$ has index 2 and degree
$>6$, hence must be an involution surface of index 2.  Over $\ksep$,
consider the diagram:
$$
\xymatrix{
 & S \ar[d]^\eta \ar[ddl]_\pi & \\
 & S_1 \ar[dl]^(.4)\sigma \ar[dr]_(.4)\tau & \\
S_0 = \PP^1\times\PP^1 & &\PP^2
}
$$
where $\sigma$ blows up a point $q$ with exceptional divisor $F$,
$\tau$ blows up two points $p_1$ and $p_2$ with exceptional divisors
$L_1$ and $L_2$, and $\eta$ blows up a point $p_3$ with exceptional
divisor $L_3$.  Here, only $\pi$ is defined over $k$ and the
exceptional divisor is $L_3+F$.  Let us denote by $H_1$ and $H_2$ the
two ruling of $\PP^1\times\PP^1$ and by $H$ the hyperplane section of
$\PP^2$, and by abuse of notations, all their pullbacks.  Then we have
$H_i = H-L_i$ for $i=1,2$ and $F = H-L_1-L_2$. Now we consider the
3-block decomposition
$$
\Db(\PP^1\times\PP^1) = \langle \ko(-H_1-H_2) | \ko |
\ko(H_1),\ko(H_2) \rangle
$$
where the first block descends to $\Db(k,A)$ and the third block
descends to $\Db(k,C_0)$.  Via the blow up $\pi$, we obtain
$$
\Db(S) = \langle \ko(-L_3),\ko(-F) | \ko(-H_1-H_2) | \ko | \ko(H_1),
\ko(H_2) \rangle
$$
where the first block, call it $\cat{H}$, descends to $\Db(k(x)/k)$.
Mutating this block to the right with respect to $\ko(-H_1-H_2)$ and
mutating $\ko(-H_1-H_2)$ to the right with respect to its right
orthogonal, and substituting the previous relations, we obtain:
$$
\Db(S) = \langle \cat{H} | \ko | \ko(H-L_1), \ko(H-L_2) | \ko(H-L_3) \rangle
$$
so that the last two blocks descend together to $\Db(L/k,\kappa)$.  We
conclude that $\Db(L/k,\kappa) = \Db(k,C_0) \times \Db(k,A)$.  It also
follows that $\cat{H}$ descends to $\Db(K/k,\beta)$ and we conclude
that $\beta$ is trivial and $K=k(x)$.

By comparing with the index 2 cases on Table~\ref{table:dp8}, we see
that: case 6.5 is the blow up of case 8.3, $Q'$ is Morita equivalent to
$C_0(A,\sigma)$ and also $Q''$ is the corestriction of $Q'$ from
$L'/k$ and is Morita equivalent to $A$;  case 6.6 is the blow up of
case 8.4 (a quadric of Picard rank 1) and $Q'$ is Morita equivalent to
$C_0(A,\sigma)$; case 6.7 is the blow up of case 8.5 and $Q' \tensor
Q'' \tensor Q'''$ is trivial; and case 6.8 is the blow up of either cases
8.6 or 8.7 (which are anyway birational to each other), in fact, it is
the resolution of this birational map. 

Finally, when $S$ is rational, we can see that: case 6.10 is the blow
up of a rational quadric of Picard rank 1 along a point of degree 2 or
and is not the blow up of $\PP^2$; case 6.11 is the blow up of $\PP^2$
in a point of degree 3 with residue field $L$ and cannot be the blow
up of a quadric; case 6.12 is the blow up of $\PP^1\times\PP^1$ in a
point of degree 2 or the blow up of a rational quadric of Picard rank
1 in two rational points (this resolves the rational map between these
two); case 6.13 is the blow up of $\PP^1\times\PP^1$ in a point of
degree 2 with residue field $L'$ or of $\PP^2$ in a union of a rational
point and a point of degree 2 with residue field $L'$ (this resolves a
birational map between the quadric and the Hirzebruch surface of
degree 1). Case 6.14 is totally split.
\end{remark}

\begin{remark}
Let $V_1 = J^{\min}$ and $V_2 = I^{\min}$ (recall
Definition~\ref{def:reduced_part}).  These vector bundles are tilting
bundles for the blocks $\cat{F}$ and $\cat{G}$, respectively, and have
the following properties: $A_1=\End(V_1)$ is Morita equivalent to $Q$
and $A_2=\End(V_2)$ is Morita equivalent to $B$; $V_1$ is
indecomposable if and only if $L$ is a field; and $V_2$ is
indecomposable if and only if $K$ is a field.  In particular, both
$V_i$ are indecomposable if and only if $S$ is minimal.  We list the
ranks and 2nd Chern classes of the vector bundles $V_i$ on Table~\ref{table:dp6}.

The calculation of the second Chern classes of the vector bundles
$V_1$ and $V_2$ is easily obtained by their description over
$k\sep$. In particular, we notice that
$\ind(S)=\mathrm{gcd}(c_2(V_1),c_2(V_2))$ unless $\ind(S)=6$, in which
case $\mathrm{gcd}(c_2(V_1),c_2(V_2))=12$.  When $\ind(S)=6$, we have
to appeal to a particular generator ($\omega_S^{\oplus 2}$ works) of
the remaining block to obtain a bundle with second Chern class equal
to 6.
\end{remark}

\section{del Pezzo surfaces of degree $5$}
\label{subsec:dP5}

Let $S$ be a Del Pezzo surface of degree $5$. It is a classical fact
(announced by Enriques \cite{enriques} and proved first by
Swinnerton-Dyer \cite{swinnerton-dyer}), that $S(k) \neq \varnothing$
(see \cite{skorobogatov-toulouse} for a different proof).

The base
change $S_{\ksep}$ is the blow-up of $\PP^2_{\ksep}$ at four points in
general position.  Such a surface has 10 exceptional lines.  Fix $x
\in S(k)$. If $x$ lies in the intersection of two exceptional lines,
then $S$ is not minimal (see \cite[29.4.4.(v)]{manin-book}). So we can
suppose that $x$ does not lie on the intersection of two exceptional
lines and consider the geometric construction described by Manin to
show that $S$ is rational: let $X \to S$ be the blow-up of $x$, and
$D$ its exceptional divisor. Then there are 5 pairwise
nonintersecting exceptional lines $L_1, \ldots, L_5$, where $L_5=D$, on $X_{\ksep}$.

Manin shows that on $X$ there is an exceptional divisor $Z \subset X$
whose contraction gives a birational map onto a del Pezzo surface of
degree 9. Since the target also has a rational point, we have a birational
map $\pi: X \to \PP^2_k$.
We have a diagram of birational morphisms
$$
\xymatrix{
& X \ar[dl]_{\pi} \ar[dr]^\varepsilon \\
\PP^2_k & & S, }
$$
where $\pi: X \to \PP^2_k$ is the blow-up of a closed point of degree
5 in general position.  

Passing to the algebraic closure, we can describe these birational
maps in the following way: let $p_1, \ldots, p_5$ be five points in
general position on $\PP^2_{\ksep}$. Then $X_{\ksep}$ is a del Pezzo
of degree four which has 16 exceptional lines: five of them are the
exceptional divisors $E_1, \ldots, E_5$ of $\overline{\pi}$, ten of
them are the strict transforms of the lines $L_{i,j}$ passing through
the points $p_i$ and $p_j$. The last one is the strict transform $D$
of the conic through the five points $p_i$. This is, indeed, the
exceptional divisor of the blow-up $\overline{\varepsilon}: X_{\ksep}
\to S_{\ksep}$. On the other hand, the lines $E_1, \ldots, E_5$ all
meet twice $D$, and it can be checked that they are the only
exceptional lines on $X_\ksep$ with such property.  Since $D$ is
defined over $k$; it follows that $E_1,\ldots,E_5$ are
Galois-invariant and hence the cycle $Z$ is the descent of the
disjoint union of the $E_i$'s.  Conversely, given any point of degree
5 (geometrically) in general position on $\PP^2_k$, we can blow it up,
and then blow down the strict transform of the conic through the point
to obtain a del Pezzo surface of degree 5.

Moreover, the surface $S_{\ksep}$ is a del Pezzo surface of degree $5$
and can therefore be realized as the blow up of $\PP^2_{\ksep}$ in
four points in general positions. Given $S_{\ksep}$, this can be
realized by the choice of 4 pairwise nonintersecting exceptional lines $L_1, \ldots, L_4$. It is easy to see, via the previous construction,
that we have five choices, one for each of the points $p_i$. Once we
fix such a point, it is then enough to consider all the lines joining
it to the other four points, which are blown up by $\overline{\pi}$ to
exceptional lines (call them loosely $L_1, \ldots, L_4$), which are,
in turn, blown down by $\overline{\varepsilon}$ to 4 exceptional nonintersecting lines. Let us then fix $p_5$, so that $L_i$ is the
(strict transform of) the line through $p_5$ and $p_i$, and consider
the blow-down $\eta: S_{\ksep} \to \PP^2_{\ksep}$.  This latter map is
not, in general, defined over $k$, so we will avoid the ``overline''
notation to mark this difference.

We end up with the following diagram:
\begin{equation}\label{eq-big-diagram-dp5}
\xymatrix{
&  \{E_1,\ldots, E_5\} \ar@{^{(}->}[r] & X_{\ksep} \ar[ddl]_{\overline{\pi}} \ar[dr]^{\overline{\varepsilon}} & \{D,L_1, \ldots, L_4\} \ar@{_{(}->}[l]\\
& & & S_{\ksep} \ar[d]^{\eta} & \{x,L_1, \ldots, L_4\} \ar@{_{(}->}[l] \\
\{p_1, \ldots, p_5\} \ar@{^{(}->}[r] & \PP^2_{\ksep} \ar@{-->}[rr]^{\phi} & & \PP^2_{\ksep} & \{x,q_1,\ldots,q_4\}, \ar@{_{(}->}[l]
 }
\end{equation}
where $q_i$ are the points blown-up by $\eta$, and $\phi$ the birational map obtained by composition.

Let us denote by $\ko_{S_{\ksep}}(H) = \eta^*
\ko_{\PP^2_{\ksep}}(1)$. We can assume that this line bundle
is not defined over $k$, since we can suppose that $S$ is
minimal. Otherwise, $S$ is the blow-up of $\PP^2_k$.

We will explicitly use this construction to show that the three-block
exceptional collection described by Karpov and Nogin
\cite{karpov-nogin} over $\ksep$ descends to a zero-dimensional
semiorthogonal decomposition of $\Db(S)$.

\begin{table}
\centering
\begin{tabular}{||c||c||c|c|c||c|c|c||}
\hline\hline
$S$ & $\ind(S)$ & $A_1$   & $c_2$ & $\rk$ & $A_2$  & $c_2$ & $\rk$ \\
\hline
\hline
$S \subset \PP^5_k$ & 1 & $k$ & 2 & 2 & $l$ & 20 & 5 \\
\hline\hline
\end{tabular}
\caption{The invariants of a del Pezzo surface $S$ of degree 5 of Picard rank 1.  Here,
the algebras $\End(V_1) = k$ and $\End(V_2) = l$ are listed up
to Morita equivalence; $l$ is an \'etale $k$-algebra of
degree 5, and is a field if and only if $\rho(S)=1$; and $Z$, $\per$, and $\ind$ refer to the center, period, and
index of $A_i$; and $c_2$ and $\rk$ refer to the
2nd Chern class and rank of $V_i$. Note that $(V_1)_{k\sep}$ is the unique extension of
$\ko(H)$ by $\ko(K_S-H)$, while $(V_2)_{k\sep}$ is the direct sum
$\ko(H)\oplus \bigoplus_{i=1}^4 \ko(L_i-K_S-H)$.} 
\label{table:dp5}
\end{table}

\begin{proposition}\label{propo:dp5}
Any del Pezzo surface of degree $5$ is $k$-rational and is
categorically representable in dimension 0. In particular, there is a
degree $5$ \'etale $k$-algebra $l$ and a semiorthogonal decomposition
$$
\cat{A}_S = \langle \Db(k), \Db(l/k)\rangle
$$
hence $\cat{A}_S$ is also representable in dimension 0.
Moreover, $l$ is a field if and only if $\rho(S)=1$. 
\end{proposition}
\begin{proof}
Over $\ksep$, Karpov and Nogin \cite[\S4]{karpov-nogin} provide the
following 3-block decomposition:
\begin{equation}
\label{eq-3-block-for-dp5}
\Db(S_\ksep)=\langle \ko_{S_{\ksep}} \, \vert \, F \, \vert \, \ko_{S_{\ksep}}(H),
\ko_{S_{\ksep}}(L_1 - K_{S_{\ksep}} - H), \ldots, \ko_{S_{\ksep}}(L_4 - K_{S_{\ksep}} - H)\rangle,
\end{equation}
where $F$ is the rank 2 vector bundle given by the nontrivial
extension
\begin{equation}\label{eq:def-of-F}
0 \longrightarrow \ko_{S_\ksep} (-K_{S_\ksep}-H) \longrightarrow F \longrightarrow \ko_{S_\ksep}(H) \longrightarrow 0. 
\end{equation}
These blocks are denoted $\cat{E}$, $\cat{F}$, and $\cat{G}$.
Consider the rank 5 vector bundle on $S_\ksep$
\begin{equation}\label{eq-def-of-V-dp5}
V = \bigoplus_{i=1}^4 \ko_{S_{\ksep}}(L_i- K_{S_{\ksep}} - H) \oplus \ko_{S_{\ksep}}(H), 
\end{equation}
and $B_V = \End_{S_{\ksep}}(V)$. Since the five line
bundles form an exceptional block, we have that $B_V \simeq
(\ksep)^5$ and $V$ is a tilting bundle for the block $\cat{G}$.
We are going to show that $V$ descends to $k$, hence $B_V$ descends to
a degree 5 extension $l$ of $k$, and $\cat{G}$ descends to $k$.  To
this end, recall that the functors $\varepsilon^*: \Db(S) \to \Db(X)$
and $\overline{\varepsilon}^*:\Db(S_{\ksep}) \to \Db(X_{\ksep})$ are
fully faithful. We analyze the pull-back of the three-block collection
\eqref{eq-3-block-for-dp5} as an exceptional collection in
$\Db(X_{\ksep})$ to deduce the descent. In order to do that, we first
stress the structure of the Picard group of $X_{\ksep}$.

On one hand, we have the line bundle $\ko_{X_{\ksep}}(H) =
\overline{\varepsilon}^* \ko_{S_{\ksep}}(H)= \overline{\varepsilon}^*
(\eta^* \ko_{\PP^2_{\ksep}}(1))$. We have the five exceptional lines
$D$ (the exceptional locus of $\varepsilon$) and $L_1, \ldots, L_4$
(by abuse of notation, we denote $L_i = \overline{\varepsilon}^*
L_i$). We have that $K_{X_{\ksep}} = \overline{\varepsilon}^*
K_{S_{\ksep}} - D$.

On the other hand, we have the line bundle $\ko_{X_{\ksep}}(G) =
\overline{\pi}^* \ko_{\PP^2_{\ksep}}(1)$, and the five exceptional
lines $E_1, \ldots, E_5$ of the blow-up $\overline{\pi}$. We have that
$K_{X_{\ksep}}= -3G + \sum_{i=1}^5 E_i$.

The divisor $D$ is the strict transform via $\pi$ of the conic passing
through the five points $p_i$. Hence $D = 2G - \sum_{i=1}^5 E_i$, so
that $-K_{X_{\ksep}}=D+G$. For each $1 \leq i \leq 4$, the divisor
$L_i$ is the strict transform of the line through $p_5$ and $p_i$, so
that $L_i= G - E_5 - E_i$.

Finally, the birational map $\phi$ is given by the system of cubics
through the five points $p_i$ which have multiplicity 2 in
$p_5$. Indeed, the map $\phi$ is not defined over the curve given by
the conic $D$ and the four lines joining $p_i$ to $p_5$. This is a
curve $C$ of degree $6$.  The map $\phi$ can be written as three
homogeneous polynomials of the same degree $d$, which we proceed to determine. The degeneracy locus
$C$ of $\phi$ is then the zero locus of the determinant of the
Jacobian matrix. This is a $3 \times 3$ matrix with entries of degree
$d-1$, hence the polynomial defining $C$ has degree $3(d-1)$.
Thus $6 = \mathrm{deg}(C)= 3(d-1)$, from which we deduce that
$d=3$.  This implies that $\phi$ is given by a linear system of
cubics.  The hyperplane section $H$ of the target $\PP^2_{\ksep}$
corresponds then to the linear system $3G- \sum_{i=1}^n m_i x_i$,
where $x_i$ are the points of multiplicity $m_i > 0$ of the map
$\phi$.  By construction, it is clear that the points of positive
multiplicity are exactly the $p_i$ (they are transformed via $\phi$
into lines), so we get the linear system $3G - \sum_{i=1}^5 m_i
p_i$. This linear system must have degree 1, so we get $9 -
\sum_{i=1}^5 m_i^2 = 1$.  Since $\phi$ is given by the cubics passing
through the points $p_i$, with multiplicity two in $p_5$, we deduce
that $m_5=2$ and $m_i=1$ for $1 \leq i \leq 4$.

From this, we get that 
\begin{equation}\label{eq-pullback-of-L}
\overline{\varepsilon}^* H= 3G - 2E_5 - \sum_{i=1}^4 E_i = 3G - \sum_{i=1}^5 E_i - E_5 = -K_{X_{\ksep}} - E_5.
\end{equation}

On the other hand, consider, for any $1 \leq i \leq 4$, the divisor
$L_i - K_{S_{\ksep}} - H$ and pull it back via
$\overline{\varepsilon}$. Recall that $\overline{\varepsilon}^*
K_{S_{\ksep}} = K_{X_{\ksep}} + D$, and that $L_i = H - E_i - E_5$
over $X_{\ksep}$. We then get
$$
\overline{\varepsilon}^* (L_i -  K_{S_{\ksep}} - H) = G - E_5 - E_i - K_{X_{\ksep}} + D - H.$$
using equation \eqref{eq-pullback-of-L}, we substitute $L$ to get:
\begin{equation}\label{eq-pullback-of-Li}
\overline{\varepsilon}^* (L_i -  K_{S_{\ksep}} - H) = 
G + D - E_i = -K_{X_{\ksep}} - E_i.
\end{equation}
Using equations \eqref{eq-pullback-of-L} and
\eqref{eq-pullback-of-Li}, the block $\cat{G}$ pulls back via
$\overline{\varepsilon}$ to the exceptional block
\begin{equation}\label{eq:block-for-dp5}
\langle\ko_{X_{\ksep}}(-K_{X_{\ksep}}-E_1),\ldots \ko_{X_{\ksep}}(-K_{X_{\ksep}}-E_5)\rangle,
\end{equation}
in $\Db(X_\ksep)$, and where we have performed a mutation of the
completely orthogonal bundles in the block to arrive at this
ordering. It follows that
$$
\overline{\varepsilon}^* V = 
\bigoplus_{i=1}^5 \ko_{X_{\ksep}}(-K_{X_{\ksep}}-E_i) =
\omega_{X_\ksep}\dual \tensor \, \bigoplus_{i=1}^5 \ko_{X_\ksep}(-E_i),
$$
hence $\overline{\varepsilon}^* V$ descends to a vector bundle of rank
5 on $X$ and $V$ descends to a vector bundle (again denoted by $V$) of
rank 5 on $S$ since $\varepsilon^*$ is fully faithful.  We see that
$\End(V)$ is then isomorphic to the structure sheaf $l$ of the degree
5 point in $\PP^2_k$ that is blown up by $\pi$, which is an
$k$-\'etale algebra of degree 5.  We conclude that $\cat{G}$ descends
to a block over $k$ equivalent to $\Db(l/k)$.

It is now sufficient to prove that $F$ descends to $k$, which would
imply that $\cat{F} \simeq \Db(k)$.  Since $\cat{E}$ and $\cat{G}$
descend to blocks of $\Db(S)$ defined over $k$, so does $\cat{F}$,
being the orthogonal complement.  Hence by Theorem~\ref{thm:toen},
$\cat{F}$ descends to a block equivalent to $\Db(k,\alpha)$ for some
$\alpha \in \Br(k)$.  We proceed to show that $\alpha$ is trivial.
To this end, consider the semiorthogonal decomposition
\eqref{eq-3-block-for-dp5}.  Orlov's formula applied to the
blow up $\ol{\varepsilon}$ gives the following 4-block semiorthogonal
decomposition:
\begin{equation}\label{eq:4-block-for-x}
\Db(X_\ksep) = \sod{\ko_{X_\ksep}(-D) | \ko_{X_\ksep} | \overline{\varepsilon}^*F | \ko_{X_\ksep}(-K_{X_\ksep}-E_1),\ldots, \ko_{X_\ksep}(-K_{X_\ksep}-E_5) }
\end{equation}
where we used the identifications \eqref{eq-pullback-of-L} and
\eqref{eq-pullback-of-Li} in writing $\cat{G}$.  Mutating $\cat{G}$ to
the left with respect to its left orthogonal, we obtain:
\begin{equation}\label{eq:4-block-for-x2}
\Db(X_\ksep) = \sod{\ko_{X_\ksep}(-E_1),\ldots,
\ko_{X_\ksep}(-E_5) | \ko_{X_\ksep}(-D) | \ko_{X_\ksep} | \overline{\varepsilon}^*F }.
\end{equation}
As the first block, the mutation of $\cat{G}$, is generated by the
exceptional divisors of the blow up $\ol{\pi}$, by Orlov's blow up
formula, it follows that $\sod{\ko_{X_\ksep}(-D) | \ko_{X_\ksep} |
\overline{\varepsilon}^*F}$ can be identified with
$\ol{\pi}^*\Db(\PP^2_\ksep)$.  Since $\pi$, as well as the line
bundles $\ko_{X}(-D)$ and $\ko_X$, is defined over $k$, we arrive at a
3-block decomposition $\pi^*\Db(\PP^2_k) =
\sod{\ko_X(-D),\ko_X,\Db(k,\alpha)}$.  By the uniqueness of 3-block
decompositions on $\PP^2$ (see \cite{gorodentsev-rudakov} or
Proposition~\ref{prop:karpov-nogin}), and by
Corollary~\ref{cor:Caldararu_conj}, we conclude that $\alpha$ is
trivial.  Moreover, $\ol{\varepsilon}^*F$ can be mutated into an
exceptional line bundle $\ol{\pi}^*\ko_{\PP^2_\ksep}(i)$ via a
sequence of mutations inside $\ol{\pi}^*\Db(\PP^2_\ksep)$, which are a
posteriori, all defined over $k$.  It follows that $\varepsilon^* F$
can be mutated to $\pi^*\ko_{\PP^2}(i)$, hence $F$ descends to a
$k$-exceptional vector bundle of rank 2.
\end{proof}

\begin{remark}
Let $V_1 = F$ and $V_2 = V$.  These vector bundles are tilting bundles
for the blocks $\cat{F}$ and $\cat{G}$, respectively, and have the
following properties: $A_1=\End(V_1)$ is Morita equivalent to $k$ and
$A_2=\End(V_2)$ is Morita equivalent to $l$; $V_1$ is always
indecomposable; and $V_2$ is indecomposable if and only if $l$ is a
field.  In particular, both $V_i$ are indecomposable if and only if
$S$ is minimal.  We list the ranks and 2nd Chern classes of the vector
bundles $V_i$ on Table~\ref{table:dp5}.

The calculation of the second Chern classes of the vector bundles
$V_1$ and $V_2$ is easily obtained by their description over
$k\sep$. In particular, we notice that
$1=\ind(S)=\mathrm{gcd}(c_2(V_1),c_2(V_2),c_2(\omega_S^{\oplus 2}))$
showing that we must include a generator of each block.
\end{remark}

\appendix

\part{Appendix: Explicit calculations with elementary links}

\section{Elementary links for non-rational minimal del Pezzo surfaces}\label{app:links1}

In this appendix, we consider all possible links between two
non-rational minimal del Pezzo surfaces $S$ and $S'$. Let us briefly
sketch the notion of elementary link in Sarkisov's program from
\cite{isko-sarkisov-complete}.  We consider $\pi: S \to Z$ to be a
minimal geometrically rational surface with an extremal contraction.
Hence one obtains that either $Z$ is a point and $S$ a minimal surface
with Picard number 1, or $Z$ is a Severi--Brauer curve and $S \to Z$ is a
minimal conic bundle, and the Picard number of $S$ is $2$.

If $S \to Z$ and $S' \to Z'$ are such extremal contractions, an \linedef{elementary link} is a birational map $\phi: S \dashrightarrow S'$
of one of the following types:

\begin{enumerate}
 \item[Type I)] There is a commutative diagram
 $$\xymatrix{S \ar[d] & S' \ar[l]_\sigma \ar[d]\\
 Z & Z' \ar[l]_\psi}$$
 where $\sigma: S' \to S$ is a Mori divisorial elementary contraction and
 $\psi: Z' \to Z$ is a morphism. In this case, $Z = \Spec(k)$,
 $\rho(S)=1$, $S$ is a minimal del Pezzo, and $S' \to Z'$ is a
 conic bundle over a Severi--Brauer curve.

\medskip

\item[Type II)]  There is a commutative diagram
 $$\xymatrix{S \ar[d] & X \ar[l]_\sigma \ar[r]^\tau& S' \ar[d] \\
 Z & & Z' \ar[ll]_\isom}$$
 where $\sigma:X \to S$ and $\tau: X \to S'$ are Mori divisorial
 elementary contractions.  In this case, $S$ and $S'$ have the same
 Picard number.

\medskip

\item[Type III)]  There is a commutative diagram
 $$\xymatrix{S \ar[d]\ar[r]^\sigma  & S' \ar[d]\\
 Z \ar[r]^\psi& Z' }$$
where $\sigma: S \to S'$ is a Mori divisorial elementary contraction and $\psi: Z \to Z'$ is 
a morphism. These links are inverse to links of type I.

\medskip

\item[Type IV)] There is a commutative diagram
 $$\xymatrix{S \ar[d]\ar[rr]^\isom & & S' \ar[d]\\
 Z \ar[rd]^\psi& & Z' \ar[ld]_{\psi'} \\
 & \Spec(k) & }$$
 where $Z$ and $Z'$ are Severi--Brauer curves and $\psi$ and $\psi'$
 are the structural morphisms. This link amounts to a change of conic
 bundle structure on $S$.
\end{enumerate}

Iskovskikh shows that any birational map $S \dashrightarrow S'$
between minimal geometrically rational surfaces factors into a finite
sequence of elementary links \cite{isko-sarkisov-complete}.  We are
interested in the case where $S$ and $S'$ are both non $k$-rational del Pezzo surfaces
of Picard rank 1.

Thanks to Iskovskikh's classification, a link of type I (resp.\ III)
can happen in the non $k$-rational cases only if $S$ (resp.\ $S'$) has either degree $8$ and a point
of degree $2$, or has degree $4$ and a rational point
\cite[Thm.~2.6]{isko-sarkisov-complete}.  It follows that if we assume
$S$ to not be of this type, then we only have to consider links of
type II where $\rho(S) = \rho(S') = 1$. By Iskovskikh's
classification, there is a finite list of such links. In particular,
if we assume $S$ to not be $k$-rational, and $S'$ not isomorphic to
$S$, then we have that $\deg(S) = \deg(S')$ can be only $6$, $8$ or $9$,
and we are left with five possible links.

Let $\phi: S \dashrightarrow S'$ be a link of type II between
non-$k$-rational non-isomorphic surfaces, and recall that we assume
that if $S$ has degree $8$ (resp. $4$), there is no degree 2
(resp. rational) point on $S$.  Then $\deg(S)=\deg(S')$ and there is a
closed point $x$ in $S$ of degree $d$ such that $\phi$ is resolved as
$$
\xymatrix{& X \ar[dr]^\tau \ar[dl]_{\sigma} & \\
S & & S'
}
$$
where $\sigma$ is the blow-up of $x$ and $\tau$ is the blow up of a
point $x'$ of degree $d$ on $S'$.  Let $E$ be the exceptional divisor
of $\sigma$ and $F$ be the exceptional divisor of $\tau$.  If one
considers the $\ZZ$-bases $(\sigma^* \omega_S, E)$ and
$(\tau^*\omega_{S'},F)$ of $\Pic(X)$, the birational map $\phi$ can be
described by the transformation matrix between these two bases.  We
list all possibilities from \cite[Thm. 2.6]{isko-sarkisov-complete}
in Table~\ref{table:links}.

\begin{table}
\centering
\begin{tabular}{||c||c|c||}
\hline\hline
$\deg(S)$ & $\deg(x)$ & Transformation Matrix \\
\hline
\hline
{\multirow{2}{*}9} & 3 & $M_{9,3} = \left(\begin{array}{cc}
                                          2 & 1 \\ -3 & -2  
                                          \end{array}\right)$\\
\cline{2-3}                                          & 6 & $M_{9,6} = \left(\begin{array}{cc}
                                          5 & 4 \\ -6 & -5  
                                          \end{array}\right)$\\

\hline
8 & 4 & $M_{8,4} = \left(\begin{array}{cc}
                                          3 & 2 \\ -4 & -3  
                                          \end{array}\right)$\\
\hline
{\multirow{2}{*}6} & 2 & $M_{6,2} = \left(\begin{array}{cc}
                                          2 & 1 \\ -3 & -2  
                                          \end{array}\right)$\\
\cline{2-3}                                          & 3 & $M_{6,3} = \left(\begin{array}{cc}
                                          3 & 2 \\ -4 & -3  
                                          \end{array}\right)$\\
\hline\hline
\end{tabular}
\caption{The possible links of type II between nonisomorphic non-$k$-rational Del Pezzo surfaces.
The transformation matrix expresses the base change from $\sigma^* \omega_S, E$ to $\tau^* \omega_{S'},F$.} 
\label{table:links}
\end{table}

In order to understand the behavior of the semiorthogonal
decompositions of $S$ and $S'$ under birational maps, it is enough to
consider the links listed in Table \ref{table:links}.  We will proceed
as follows: given a link $\phi: S \dashrightarrow S'$, we describe the
birational map $\overline{\phi}: S_{k\sep}\dashrightarrow
S'_{k\sep}$. Notice that $\overline{\phi}$ is not a link, since over
$k\sep$, we can factor $\sigma$ into a finite sequence of blow-ups
(actually, $\deg(x)$ of them).

In order to describe $\overline{\phi}$ we will consider the
description of the Picard group of $S_{k\sep}$.  If $\deg(S) =9$, then
$\overline{\phi}$ is described by a linear system on $\PP^2_{k\sep}$,
the so-called \linedef{homaloidal system} of $\overline{\phi}$. If
$\deg(S) = 8$, we find similarly a homaloidal system on the quadric
$S_\ksep \subset \PP^3_{k\sep}$. Finally, if $\deg(S)=6$, we have
to choose models $S_\ksep \to \PP^2_{k\sep}$ and $S'_\ksep
\to \PP^2_{k\sep}$, and describe how $\overline{\phi}$ corresponds to a
homaloidal system on $\PP^2_{k\sep}$.

In general, let us consider a linear system on $\PP^2$ of the form
$G=nH - \sum_{i=1}^r m_i p_i$, where $n>0$ and $m_i>0$ are integers,
$H$ denotes the hyperplane divisor, and $p_i$ are points on $\PP^2$.
Such a linear system defines a birational map $\PP^2 \dashrightarrow \PP^2$ if and
only if $\deg(G)=1$ and the curves in the linear system are
rational. We can resolve the birational map by blowing up the points
$p_i$, and call $X$ the blow up. We obtain then the
following conditions on $n$ and $m_i$:
\begin{equation}\label{eq:conds-on-hom-system}\left\lbrace \begin{array}{ll}
  3n-3 = \sum_{i=1}^r m_i & \text{since } G.K_X=3\\
  n^2 -1  = \sum_{i=1}^r m_i^2 & \text{since } G^2=1 
\end{array}\right.\end{equation}
In order to describe the system, we will extensively use the
Cauchy--Schwartz inequality
$$
\Bigl(\sum_{i=1}^r m_i\Bigr)^2 \leq r \sum_{i=1}^r m_i^2
$$
to bound the possible values of $n$. In particular, we obtain that
$9(n-1) \leq r(n+1)$.  Moreover, it is clear that if $n=1$, then $G=H$
defines an isomorphism. Hence we have that $r \geq 3$. Let us spell
out all the possible birational transforms with $3 \leq r \leq 6$.

If $r=3$, then $n=2$. Conditions \eqref{eq:conds-on-hom-system} give $m_i=1$. These are the standard quadratic transformations
$\phi_2: \PP^2 \dashrightarrow \PP^2$.

If $r=4$, then $n=2$. Conditions \eqref{eq:conds-on-hom-system} give
$\sum_{i=1}^4 m_i^2 = 3$, which is impossible since we only consider
$m_i > 0$. It follows in particular that $n=2$ if and only if $r=3$.

If $r=5$, then $n=3$. Conditions \eqref{eq:conds-on-hom-system} give
$m_1= \ldots = m_4=1$ and $m_5=2$, so that there is only one
possibility (up to renumbering the points). We denote these birational
maps by $\phi_3: \PP^2 \dashrightarrow \PP^2$.

If $r=6$, then $n \leq 5$. Conditions \eqref{eq:conds-on-hom-system}
give two possibilities.  The first one is $n=5$, $m_i=2$. We denote
these birational maps by $\phi_5: \PP^2 \dashrightarrow \PP^2$.  The
second possibility is $n=4$, $m_1=m_2=m_3=1$ and $m_4=m_5=m_6=2$ (up
to renumbering the points).  One can check that the birational map
$\phi_4: \PP^2 \dashrightarrow \PP^2$ is the composition of two
standard quadratic transforms, the first one $\phi_2 : \PP^2
\dashrightarrow \PP^2$ blows up $p_1$, $p_2$, and $p_3$, so that $p_4$,
$p_5$, and $p_6$ belong to the target $\PP^2$. The second standard
quadratic transform blows-up $p_4$, $p_5$, and $p_6$.

With this calculation in mind, we are able to describe the homaloidal
systems of $\PP^2_{k\sep}$ for the links in degree 6, 8, and 9 del Pezzo
surfaces in Table~\ref{table:links}. 

\subsection{Degree 9}\label{subs:Sarkisov-deg9}
If $\deg(S) = 9$, then $S_\ksep \simeq \PP^2_{k\sep}$. Hence we are considering a birational
map $\overline{\phi}: \PP^2_{k\sep} \to \PP^2_{k\sep}$. 

\medskip

The link $M_{9,3}$ base changes to the following diagram (we omit the overlines for ease of notations):
$$\xymatrix{
& X \ar[dl]_{\sigma} \ar[dr]^\tau &\\
\PP^2 \ar@{-->}[rr]^{\phi}& & \PP^2\\
}$$
where $\sigma$ blows up three points $p_1$, $p_2$, and $p_3$.  Hence
$\phi$ is the standard quadratic transformation. Since
$G=2H-L_1-L_2-L_3$, we have that
$\sigma^*\ko(-1) = \tau^*\ko(1) \otimes \omega_X$.

\medskip

The link $M_{9,6}$ base changes to the following diagram (we omit the overlines for ease of notations):
$$\xymatrix{
& X \ar[dl]_{\sigma} \ar[dr]^\tau &\\
\PP^2 \ar@{-->}[rr]^{\phi}& & \PP^2\\
}$$
where $\sigma$ blows up six points $p_1 \ldots, p_6$.  As explained
above, we have two possibilities: $\phi$ is either of type $\phi_5$ or
$\phi_4$.  Checking the action of the matrix $M_{9,6}$ one gets that
$\phi$ is of type $\phi_5$, since all $m_i$ must be equal.  Since
$G=5H-2L_1-2L_2-2L_3-2L_4-2L_5-2L_6$, we have that $\sigma^*\ko(-1) =
\tau^*\ko(1) \otimes \omega_X^{\otimes 2}$.

These considerations lead to a simple proof of Amitsur's theorem in
the case of degree 3 central simple algebras.

\begin{proposition}\label{prop:amitsur-revisited}
Let $S$ be a non-$k$-rational minimal surface of degree $9$, and let
$\cat{T}_S^+$ (resp. $\cat{T}_S^-$) be the category generated by the
descent of a hyperplane section $\ko_{S_\ksep}(1)$
(resp. $\ko_{S_\ksep}(-1)$).  For any birational map $\phi: S
\dasharrow S'$ to a minimal surface $S'$ (which must be of degree 9),
we have either an equivalence $\cat{T}^+_S \simeq \cat{T}^+_{S'}$, or
an equivalence $\cat{T}^+_S \simeq \cat{T}^-_{S'}$.
\end{proposition}

\begin{proof}
It is easy to see that any elementary link $S \dashrightarrow S'$ gives an equivalence between $\cat{T}_S^+$ and $\cat{T}_{S'}^-$,
just by pull-back to $X$ and tensor by either $\omega_X$ or $\omega^{\otimes 2}_X$.
\end{proof}

\subsection{Degree 8}\label{subs:Sarkisov-deg8}
If $\deg(S) = 8$ and $S$ is an involution surface, then $S_\ksep
\subset \PP^3_{k\sep}$ is a quadric surface. We are interested in the
hyperplane section $\ko(1)$ and in $\ko(2)=\omega_{S_\ksep}\dual$,
the anticanonical divisor. The latter is always defined on $S$.  
\medskip

The link $M_{8,4}$ base changes to the following diagram (we omit the overlines for ease of notations):
$$\xymatrix{
& X \ar[dl]_{\sigma} \ar[dr]^\tau &\\
S \ar@{-->}[rr]^{\phi}& & S'\\
}$$
where $\sigma$ blows up four points $p_1, \ldots, p_4$.  Using the
action of the matrix on the Picard group of $X$, we get that
$\sigma^*\ko(1) = \tau^*\ko(1) \otimes \omega_X^{\otimes 2}$.

\bigskip

As a corollary, we can see that involution surfaces with Picard rank
one and no point of degree $\leq 2$ are birationally semirigid.  We
can give a further refinement of this result purely using the
algebraic theory of quadratic forms.  Recall that an involution
surface has Picard rank 2 or 1 depending on whether it has trivial or
nontrivial discriminant, respectively, see Example~\ref{exam:involution}.

\begin{proposition}
\label{prop:dP8_semi-rigid}
Let $S$ and $S'$ be $k$-birational involution surfaces over an
arbitrary field $k$.  Then $S$ and $S'$ are $k$-isomorphic in the
following cases:
\begin{enumerate}
\item\label{prop:dP8_semi-rigid.1} 
$S$ and $S'$ are anisotropic quadrics in $\PP^3$,

\item\label{prop:dP8_semi-rigid.2} 
$S$ (and hence $S'$) has nontrivial discriminant and no rational point,

\item\label{prop:dP8_semi-rigid.3} 
$S$ (and hence $S'$) has index 4.
\end{enumerate}
 \end{proposition}
\begin{proof}
We know that $S$ is rational as soon as it has a rational point, in
which case it can have $k$-birational yet nonisomorphic partners.
Hence we can assume that $S(k)=\varnothing$.  By considering the
Galois action on the the Picard groups, we see that $k$-birational
involution surfaces have the same discriminant.
Now let $S$ and $S'$ be associated to quadratic pairs $(A,\sigma)$ and
$(A',\sigma')$.  

Part \ref{prop:dP8_semi-rigid.1} is a classical result in the theory
of quadratic forms, see \cite[XII.2.2]{lam-book}.
This handles the case when both $A$ and $A'$ are split, hence we may
assume that $A$ is not split.

We recall a result proved by Arason for quadric surfaces and
generalized by Tao~\cite[Thm.~4.8(b)]{tao:involution_variety} to
involution surfaces:
\begin{equation}
\label{eq:tao}
\ker(\Br(k) \to \Br(k(S))) = 
\begin{cases}
\langle A \rangle & \text{if the discriminant is nontrivial} \\
\langle C_0^+,C_0^- \rangle & \text{if the discriminant is trivial}
\end{cases}
\end{equation}
Any $k$-birational isomorphism between $S$ and $S'$ induces a
$k$-isomorphism of function fields $k(S) \isom k(S')$, under which the
Brauer group kernels $\ker(\Br(k) \to \Br(k(S)))$ and $\ker(\Br(k) \to
\Br(k(S')))$ coincide.  We now proceed according to cases.

For part \ref{prop:dP8_semi-rigid.2}, $S$ and $S'$ have nontrivial
discriminant, so \eqref{eq:tao} implies that the cyclic subgroups of
the Brauer group generated by $A$ and $A'$ are the same.  However,
since both algebras carry involutions of the first kind, they are of
period 2 in the Brauer group (we are assuming they are not split).
Thus $A$ and $A'$ are Brauer equivalent, hence are $k$-isomorphic
since they have the same degree.  We will now show that the quadratic
pairs $(A,\sigma)$ and $(A',\sigma')$ become adjoint to anisotropic
quadratic forms over $F = k(\SB(A))\isom k(\SB(A'))$ (and $S_F$ and
$S_F'$ are still $F$-birational), which by case
\ref{prop:dP8_semi-rigid.1} implies that they are isomorphic over $F$,
which in turn, by the following Lemma~\ref{lem:anne}, implies that
they are isomorphic over $k$.  This anisotropicity statement follows,
in the case when $S$ (hence $S'$) has index 4, i.e., that $A \isom A'$
is division, from
Tao~\cite[Prop.~4.18,~Cor.~4.20,~Cor.~4.21]{tao:involution_variety}
(see also Karpenko~\cite[Thm.~5.3]{karpenko:anisotropy_orthogonal} in
greater generality).  In the case when $S$ (hence $S'$) has index 2,
which under our assumptions implies that $A \isom A'$ has index 2, a
result of Karpenko~\cite[Thm.~3.3]{karpenko:isotropy_quadratic_pair},
stating that the Witt index of a quadratic pair $(A,\sigma)$ over $F$
is divisible by the index of the $A$, implies that the quadratic pairs
$\sigma$ and $\sigma'$ are either anisotropic or hyperbolic over $F$
(see also \cite{becher_dolphin}).  The later is impossible, since by
assumption the quadratic pairs have nontrivial discriminant, which
remains nontrivial over $F$.  Another way to see the index 2 case, at
least when the characteristic is not 2, is by writing $A = M_2(H)$,
where $H$ is a quaternion algebra, and interpreting the involution
$(A,\sigma)$ as a $(-1)$-hermitian form of rank 2 over $(H,\tau)$,
where $\tau$ is the standard involution, and similarly for
$(A',\sigma')$, then invoking the result of Parimala, Sridharan, and
Suresh~\cite{parimala_sridharan_suresh:hermitian} that the involutions
remain anisotropic over $k(\SB(H))$, hence over $F$, since
$F/k(\SB(H))$ is purely transcendental.

To finish part \ref{prop:dP8_semi-rigid.3}, we need only deal with the
case of index 4 and trivial discriminant, in which case \eqref{eq:tao}
implies an equality of two Klein four subgroups of the Brauer group.
By the fundamental relations for Clifford algebras
\cite[Thm.~9.14]{book_of_involutions}, we have the equality of Brauer
classes $[A] = [C_0^+] + [C_0^-]$, and we can rule out $[A] =
[C_0^\pm]$ since $\ind(A) = 4$ while $\ind(C_0^\pm) \leq 2$ (in fact,
we see that $\ind(C_0^\pm)=2$), and similarly for $A'$.  We deduce
that $A$ and $A'$ are each the unique element of index 4 in their
respective Klein four Brauer group kernels.  Hence $A$ and $A'$ are
Brauer equivalent, thus are $k$-isomorphic since they have the same
degree.  Also, the unordered pairs of Clifford algebra components
$C_0^+$ and $C_0^-$, associated to $A$ and $A'$, are isomorphic.
Thus, by the classification of quadratic pairs of degree 4 and trivial
discriminant \cite[\S15.B]{book_of_involutions}, both $(A,\sigma)$ and
$(A',\sigma')$ are isomorphic to $(C_0^+,\tau_0^+)\tensor
(C_0^-,\tau_0^-)$, where $\tau_0^\pm$ is the standard involution on
the quaternion algebra $C_0^\pm$.
\end{proof}

The following result, in characteristic $\neq 2$, can be seen as a
consequence of general hyperbolicity results for orthogonal
involutions due to Karpenko \cite{karpenko:hyperbolicity}.  The
following direct argument in the case of degree 4 algebras, using the
results of \cite[\S15.B]{book_of_involutions}, was communicated to us
by Anne Qu\'eguiner-Mathieu and works over any field.

\begin{lemma}
\label{lem:anne}
Let $\sigma_1$ and $\sigma_2$ be quadratic pairs on a central simple
algebra $A$ of degree 4 over a field $k$ and let $F = k(\SB(A))$.  If
$\sigma_1$ and $\sigma_2$ become isomorphic quadratic pairs over $F$
then they are isomorphic over $k$.
\end{lemma}
\begin{proof}
Since $(A,\sigma_1)$ and $(A,\sigma_2)$ become isomorphic over $F$,
their discriminants coincide over $F$, hence coincide over $k$, since
the map $H^1(k,\ZZ/2\ZZ) \to H^1(F,\ZZ/2\ZZ)$ is injective.  Let $K/k$ be
the discriminant extension.  By the low dimension classification of
algebras of degree 4 with quadratic pair
\cite[\S15.B]{book_of_involutions}, we have that $(A,\sigma_i) =
N_{K/k}(H_i,\tau_i)$, where $\tau_i$ is the standard involution on the
quaternion algebra $H_i$ over $K$.

Over $K$, we get $(A,\sigma_i)_K=(H_i,\tau_i)\otimes
({}^\iota H_i, {}^\iota\tau_i)$, where $\iota$ is the the non trivial
automorphism of $K/k$.  Let $KF$ be the compositum of $K$ and $F$.
Over $KF$, we have that $H_i$ and ${}^\iota H_i$ are isomorphic, hence
$A \isom \End(H_i)$ and $\sigma_i$ is adjoint to the quadratic norm
form of $H_i$.  Therefore, if the quadratic pairs $\sigma_i$ are
isomorphic over $F$, then the norm forms of $H_i$ over $KF$ are
isomorphic.  In particular, $H_1 \tensor H_2$ is split over $KF$.
Hence, either $H_1 \tensor H_2$ is already split over $K$, or, by
Amitsur's theorem, $H_1\tensor H_2=H_1\tensor {}^\iota H_1$.  Thus
$H_2$ is either isomorphic to $H_1$ or ${}^\iota H_1$.  In both cases,
we get an isomorphism of quadratic pairs $\sigma_1$ and $\sigma_2$.
\end{proof}

In the remaining cases, we will classify all possible nonisomorphic
birational involution surface partners in
Proposition~\ref{prop:birational_involution_partners}.

\subsection{Degree 6}\label{subs:Sarkisov-deg6}
Let $S$ be a degree $6$ non-$k$-rational del Pezzo surface. Then, as
recalled in Table \ref{table:links}, there are two possible types of
elementary links. Consider $S_{k\sep}$, and recall that there are six
exceptional lines, coming into two triples of non-intersecting
lines. Each of these triples gives a map $S_{k\sep} \to
\PP^2_{k\sep}$, so that we have two ways of blowing down $S_{k\sep}$
to $\PP^2_{k\sep}$. The previous considerations show that $S_{k\sep}$
can be seen as the resolution of a standard quadratic transformation.
Let $H'$ and $H$ denote the pull back of the generic line from each of
the $\PP^2$, and let $\{L_i\}$ and $\{L_i'\}$ be the two sets of
exceptional divisors.

In particular, we have two $\ZZ$-bases for $\Pic(S_{k\sep})$, one
given by $H$ and the $L_i$, and the other given by $H'$ and the
$L_i'$.  We have $H'=2H-L_1-L_2-E$, and $L_i'= H - L_j - L_k$ for
$i\neq j \neq k \neq i$.  If $\phi: S \dashrightarrow S'$ is an
elementary link, we suppose that we have chosen the triple $L_i$
(and hence $L_i'$) to describe $\overline{\phi}$ as coming from a
homaloidal system on $\PP^2_{k\sep}$.

\medskip

The link $M_{6,2}$ base changes to the following diagram (we omit the overlines for ease of notations):
$$
\xymatrix{
& X \ar[dl]_{\sigma} \ar[dr]^\tau &\\
S \ar[d]_{\sigma_0}\ar@{-->}[rr]^{\phi} & & S \ar[d]^{\tau_0} \\
\PP^2 \ar@{-->}[rr]^{\phi_0}& & \PP^2
}
$$
where $\sigma_0$ blows up three points $p_1$, $p_2$, and $p_3$, and
$\sigma$ blows up two points $p_4$ and $p_5$.  Hence $\phi_0$ is
resolved by blowing up $5$ points. In this case, we should calculate
which one of the five points has coefficient 2 in the homaloidal system
of $\phi_0$. Let us then denote $H = \sigma^* \sigma_0^* \ko(1)$,
$G = \tau^* \tau_0^* \ko(1)$, $L_i$ the exceptional divisor
over $p_i$, and $F_i$ the exceptional divisor over $q_i$.

The matrix $M_{6,2}$ in Table~\ref{table:links} is the transformation matrix from $\sigma^* \omega_S, E$ to $\tau^* \omega_{S'}, F$.
Since $\omega_S = -3H + L_1 + L_2 + L_3$ and $\omega_{S'}=-3G + F_1 + F_2 + F_3$, we get the following conditions:
\begin{equation}\label{eq:first-crazy-cont} \left\lbrace \begin{array}{l}
3G - F_1 - F_2 - F_3 = 6H - 2L_1 - 2L_2 - 2L_3 - 3L_4 - 3L_5 \\
F_4 + F_5 = 3H - L_1 - L_2 - L_3 -2L_4 - 2L_5.
             \end{array} \right.
\end{equation}
Since $X$ is a del Pezzo of degree $4$, there are only $16$ exceptional lines on $X$, which can be described, in the
base $H, L_i$ as follows:

\begin{itemize}
 \item The five exceptional lines $L_i$,
 \item The ten strict transforms $L_{i,j}$ of the lines through two of the $p_i$'s. They are of the form
 $H-L_j-L_i$ for any $i \neq j$.
 \item The strict transform $D$ of the conic through the $p_i$'s. It is of the form $2H - \sum_{i=1}^r L_i$.
\end{itemize}
Using the above description (and the fact that $F_i$ is not of type $L_j$), it is easy to check that (up to switching $4$ and $5$) we get that $F_5=D$
and $F_i = L_{i,5}$ for $i \neq 5$. Hence the homaloidal system of $\phi_0$ is $3H - L_1 - \ldots - L_4 - 2L_5$, which is indeed the only
linear system with 5 base points. Our calculation aims at finding the
point with coefficient $-2$.  Notice that while we could have switched
$4$ and $5$, the coefficients of $L_1$, $L_2$, and $L_3$ must be $1$.

\medskip

The link $M_{6,3}$ base changes to the following diagram (we omit the overlines for ease of notations):
$$\xymatrix{
& X \ar[dl]_{\sigma} \ar[dr]^\tau &\\
S \ar[d]_{\sigma_0}\ar@{-->}[rr]^{\phi} & & S \ar[d]^{\tau_0} \\
\PP^2 \ar@{-->}[rr]^{\phi_0}& & \PP^2\\
}$$
where $\sigma_0$ blows-up three points $p_1$, $p_2$ and $p_3$, and $\sigma$ blows-up three points $p_4$, $p_5$ and $p_6$.
Hence $\phi_0$ is resolved by blowing-up $6$ points.
In this case we have two possibilities for the homaloidal system of $\phi_0$, and we appeal to the form of the matrix $M_{6,3}$
to understand which one we are indeed considering. Let us then denote by $H:=\sigma^* \sigma_0^* \ko(1)$, 
and $G:= \tau^* \tau_0^* \ko(1)$, by $L_i$ the exceptional divisor over $p_i$, and by $F_i$ the exceptional divisor over $q_i$.

The matrix $M_{6,3}$ in Table \eqref{table:links} is the
transformation matrix from $(\sigma^* \omega_S, E)$ to $(\tau^*
\omega_{S'}, F)$.  Since $\omega_S = -3H + L_1 + L_2 + L_3$ and
$\omega_{S'}=-3G + F_1 + F_2 + F_3$, we get the following conditions:
\begin{equation}\label{eq:first-crazy-cond} \left\lbrace \begin{array}{l}
3G - F_1 - F_2 - F_3 = 9H - 3L_1 - 3L_2 - 3L_3 - 4L_4 - 4L_5 -4L_6 \\
F_4 + F_5 + F_6 = 6H - 2L_1 - 2L_2 - 2L_3 - 3L_4 - 3L_5 - 3L_6.
             \end{array} \right.
\end{equation}
Since $X$ is a del Pezzo of degree $3$, there are only $27$
exceptional lines on $X$, which can be described, in the basis $(H,
L_i)$, as follows:

\begin{itemize}
 \item The six exceptional lines $L_i$.
 \item The fifteen strict transforms $L_{i,j}$ of the lines through
 two of the $p_i$'s. They are of the form $H-L_j-L_i$ for any $i \neq
 j$.
 \item The six strict transforms $D_j$ of the conic through the $p_i$'s for $i \neq j$. It is of the form $2H - \sum_{i\neq j} L_i$.
\end{itemize}

Now we use that $\phi_0$ is a birational map resolved by a cubic surface,
we have hence two possibilities to write $G$ in the basis $H$, $L_i$. The first one
is $G=4H - 2L_1 - 2L_2 - 2L_3 - L_4 - L_5 - L_6$, in which case $\phi_0$ is the map we called $\phi_4$
above, and we observed that this map is the composition of two standard quadratic tranformations.
It is easy to check that in this case $S$ would not be minimal, because it
would admit a birational morphism onto a Severi--Brauer surfce.

We are left with the case where 
$G= 5H - \sum_{i=1}^6 2L_i$, in which case $\phi_0$ is the map we called
$\phi_5$ above. 
Using the above description of the exceptional lines on the cubic and the action
of the matrix $M_{3,6}$ (and the fact that $F_i$ is not of type
$L_j$), it is easy to check that (up to internal permutations of $1,2,3$ and of
$4,5,6$) we get that $F_1=D_4$, $F_2 = D_5$, and
$F_3=D_6$, while $F_4=D_1$, $F_5=D_2$, and $F_6=D_3$.

Hence, if $S$ is minimal with a point of degree $3$, the birational
map $\phi$ corresponds over $k\sep$ to the homaloidal system
of quintics passing twice to six points in general position in $\PP^2_{k\sep}$,
three of which are Galois-conjugate and correspond to the closed point of degree 3
on $S$.

\section{Links of type I, and minimal del Pezzo surfaces of ``conic bundle type''}\label{subs:links-of-type-1}
Let $S$ be a minimal non-rational del Pezzo surface, which is not
deg-rigid. Then $S$ has either degree $8$ and a point of degree $2$,
or degree $4$ and a point of degree $1$. In both cases, blowing up the
given point gives a conic bundle $S' \to C$ over a conic, of degree
either $6$ or $3$, respectively.

In this appendix, we would like to show how these two special cases
should be thought of, from a derived categorical (or a noncommutative)
point of view as conic bundles instead of del Pezzo surfaces. The main
point is describing a semiorthogonal decomposition of $S$ whose
nonrepresentable components can be seen as the ``natural''
nonrepresentable components of some conic bundle $S'$.  First of all,
let us recall a result of Kuznetsov on the derived category of a conic
bundle, as a special case of a quadric fibration (see
\cite{kuznetquadrics}, and \cite{auel-berna-bolo} for a statement over
any field).

\begin{proposition}\label{prop:semi-deco-for-cbunldes}
Let $\pi:X \to C$ be a conic bundle over a genus zero curve, $A$ the Azumaya algebra associated to
$C$, and $C_0$ the even Clifford algebra associated to $\pi$. Then there is a semiorthogonal
decomposition
$$\Db(X) = \langle \Db(k), \Db(k,A), \Db(C,C_0)\rangle,$$
where the two first components are the pull-back via $\pi$ of the natural semiorthogonal
decomposition of $\Db(C)$. In particular, the first one is generated by $\ko_X=\pi^*\ko_C$ and the
second one by the local form of $\pi^* \ko_C(1)$.
\end{proposition}

In particular, given a conic bundle, we have two natural potentially nonrepresentable components,
one of which is representable in dimension 0 if and only if $C \simeq \PP^1$.

\subsection*{Del Pezzo of degree 4 with a rational point}
Let $S$ be a degree $4$ del Pezzo surface with a rational point. Any
such surface can be realized in $\PP^4$ as an intersection of two
quadrics. In Theorem \ref{thm:A-for-dp4}, we recalled how $\cat{A}_S
\simeq \Db(\PP^1,C_0)$, where $C_0$ is the Clifford algebra of the
quadratic form spanned by the two quadrics. Since $S$ has a point, the
fibration has a regular section, so that it can be reduced by
hyperbolic splitting (see \cite[\S1.3]{auel-berna-bolo}) to a conic
bundle with Clifford algebra $C_0'$ over $\PP^1$, so that
$\Db(\PP^1,C_0) \simeq \Db(\PP^1,C_0')$.

On the other hand, one can blow up the point and obtain a degree 3
surface $S'$, with a structure of conic bundle $S' \to \PP^1$ (see,
e.g., \cite[Thm.~2.6(i)]{isko-sarkisov-complete}). Let us denote by
$B_0$ the Clifford algebra of such a conic bundle.

\begin{theorem}[{\cite[\S4]{auel-berna-bolo}}]\label{thm:cliffords-for-dp4s}
For $S$ a del Pezzo of degree 4 with a rational point, the
$\ko_{\PP^1}$-algebras $C_0$, $C_0'$, and $B_0$ described above are
all Morita equivalent.  In particular, $\cat{A}_S \simeq
\Db(\PP^1,B_0)$.
\end{theorem}

\subsection*{Del Pezzo of degree 8 with a degree 2 point}
In the case where $S$ is involution surface with a point of degree 2, the nonrepresentable
component can also be described by a Clifford algebra over $\PP^1$.

\begin{theorem}
Let $S$ be a minimal nonrational del Pezzo surface of degree $8$, with
a closed point $x$ of degree $2$. Let $S' \to S$ the blow-up of $S$
along $x$ and $\pi: S' \to C$ the associated conic bundle.  Write
$C=\SB(A')$ and $C'_0$ for the even Clifford algebra of
$\pi$.  Recall the semiorthogonal decomposition
\begin{equation}\label{eq:recall-decoquad}
\Db(S)= \langle \Db(k), \Db(k,A), \Db(k,C_0) \rangle,
\end{equation}
where $A$ is the underlying degree 4 central simple algebra defining
$S$, i.e., $S \hookrightarrow \SB(A)$, and $C_0$ is the even Clifford
algebra associated to $S$.  Then $A$ and $A'$ are Brauer equivalent
and there is a semiorthogonal decomposition
$$\Db(C,C'_0) = \langle \Db(l/k), \Db(k,C_0)\rangle,$$
where $l$ is the residual field of $x$.
\end{theorem}

\begin{proof}
Consider the diagram:
$$\xymatrix{
& S' \ar[dl]_\sigma \ar[d]^\pi \\
S \ar@{-->}[r]^{\phi} & C, }$$
where $\phi$ is a rational map that is resolved by the blow up
$\sigma: S' \to S$ of $x$ to the conic bundle $\pi:S' \to C$.  Denote
by $E$ the exceptional divisor of $\sigma$.  Over $\ksep$, we have
that $C \isom \PP^1_\ksep$ and $S \isom \PP^1_\ksep \times
\PP^1_\ksep$ and $x$ decomposes into points $x_1$ and $x_2$.  Let $G =
\sigma^* \ko(1,1)$ and $H = \pi^*\ko(1)$. 

As one can check, comparing parts a) and b) of the case $K^2=8$ of
\cite[Thm.~2.6(i)]{isko-sarkisov-complete}, the rational map $\phi$ is
defined, over $\ksep$, by the linear system $\ko_{S_\ksep}(1,1) -
x_1-x_2$.  In particular, since $\pi$ resolves $\phi$, we have $H = G
- L_1-L_2$ over $\ksep$, where $E=L_1+L_2$. The semiorthogonal
decomposition \eqref{eq:recall-decoquad} can then be written:
$$
\Db(S_\ksep) = \langle \ko, \ko(1,1), \cat{F}\rangle,
$$
where $\cat{F}$ is the block descending to $\Db(k,C_0)$. 
Consider the semiorthogonal decompositions of $\Db(S')$ given
respectively by the blow up and by the conic bundle formulas:
\begin{equation}\label{eq:two-decos-to-compare}
\Db(S'_\ksep) = \sod{\ko_{S'_\ksep}, \ko_{S'_\ksep}(G),
\sigma^*\cat{F}, \ko_{L_1}, \ko_{L_2}}
\end{equation}
Now consider the decomposition \eqref{eq:two-decos-to-compare}, and
mutate $\sigma^*\cat{F}$ to the right with respect to
$\sod{\ko_{L_1},\ko_{L_2}}$, then mutate $\sod{\ko_{L_1},\ko_{L_2}}$
to the left with respect to $\ko_{S'_\ksep}(G)$, then we mutate
$\sod{\ko_{S'_\ksep}(G), \sigma^*\cat{F}}$ to the left with respect to
its left orthogonal, then we mutate $\sigma^*\cat{F}$ to the left with
respect to its left orthogonal, then we mutate $\sigma^*\cat{F}$ to the
right with respect to its right orthogonal to arrive at a decomposition:
\begin{equation}\label{eq:two-decos-to-compare2}
\Db(S'_\ksep) = \sod{\ko_{S'_\ksep}(-G+L_1+L_2), \ko_{S'_\ksep}, \ko_{S'_\ksep}(G-L_1), \ko_{S'_\ksep}(G-L_2),\sigma^*\cat{F}}
\end{equation}
where we have used the evaluation sequence and the fact that
$\omega_{S'_\ksep}=\ko_{S'_\ksep}(-2G+L_1+L_2)$.  Rewriting in terms
of $H$, we have the decomposition
\begin{equation}\label{eq:two-decos-to-compare3}
\Db(S'_\ksep) = \sod{\ko_{S'_\ksep}(-H), \ko_{S'_\ksep}, \ko_{S'_\ksep}(H+L_2), \ko_{S'_\ksep}(H+L_1),\sigma^*\cat{F}}
\end{equation}
in which we can identify $\sod{\ko_{S'_\ksep}(-H), \ko_{S'_\ksep}}$
with $\pi^*\Db(C_\ksep)$.  Since $\ko_{S'_\ksep}(G)$ is the mutation
of $\ko_{S'_\ksep}(-H)$, and these mutations are defined over $k$, we
conclude that they $\sod{\ko_{S'_\ksep}(G)}$ and
$\sod{\ko_{S'_\ksep}(-H)}$ descend to equivalent categories over $k$.
It follows that $\Db(k,A) \simeq \Db(k,A')$, hence by
Corollary~\ref{cor:Caldararu_conj}, $A$ and $A'$ are Brauer
equivalent.  

The block $\sod{\ko_{S'_\ksep}(H+L_2), \ko_{S'_\ksep}(H+L_1)}$ has
tilting bundle $\ko_{S'_\ksep}(H+L_2)\oplus \ko_{S'_\ksep}(H+L_1) =
\ko_{S'_\ksep}(H)\tensor(\ko_{S'_\ksep}(L_2)\oplus
\ko_{S'_\ksep}(L_1))$.  There is a vector bundle $V$ on $C$ of rank 2
such that $V_\ksep = \ko(H)^{\oplus 2}$ and there is a vector bundle
$W$ on $S'$ of rank 2 such that $W_\ksep = \ko(L_1)\oplus \ko(L_2)$.
Then $\pi^*V \tensor W$ is a tilting bundle for a category base
changing to the block $\sod{\ko_{S'_\ksep}(H+L_2),
\ko_{S'_\ksep}(H+L_1)}$.  It follows that the later descends to a
category equivalent to $\Db(k(x)/k,A)$.  Since $k(x)$ is the residue
field of a point of $S \subset \SB(A)$, we have that $k(x)$ splits the
algebra $A$.  When $\Db(k(x)/k,A) \simeq \Db(k(x)/k)$.
In conclusion, we get a decomposition $\pi^*\Db(C,C_0') =
\sod{\Db(l/k),\cat{F}}$ and we recall that $\cat{F} \simeq \Db(k,C_0)$. 
\end{proof}

\section{Links of type II between conic bundles of degree 8}
\label{app:conic_bundles}

In this section, we consider conic bundles of degree $8$ and study their semiorthogonal decompositions under links of type II
and IV.
The upshot is to show that the Griffiths--Kuznetsov component is well defined in these cases, which include in particular all involution
surfaces of Picard rank 2.

Let us first consider a conic bundle $\pi: S \to C$ over a Severi--Brauer curve $C=\SB(A)$. Such a conic bundle
has an associated Clifford algebra $\kc_0$, a locally free sheaf over $C$. Kuznetsov provides a semiorthogonal
decomposition:
$$\Db(S) = \sod{\Db(C), \Db(C,\kc_0)},$$
and shows that there is a root stack structure $\widehat{C}$, obtained by the natural $\ZZ/2\ZZ$-action on points
of $C$ where the fiber is degenerate, and a Brauer class $\beta$ in
$\Br(C)$ such that $\kc_0$ pulls back
to $\widehat{C}$ to an Azumaya algebra with class $\beta$. Recall that a conic bundle over $C$ has degree $8-r$, where
$r$ is the number of degenerate fibers. If $S$ has degree $8$, then it has no degenerate fibers and hence $\widehat{C} \isom C$
since the root stack structure is trivial. Recall moreover that, denoting by $\alpha \in \Br(k)$ the class of $A$,
there is a semiorthogonal decomposition $\Db(C)=\sod{\Db(k),\Db(k,\alpha)}$. We finally obtain a semiorthogonal
decomposition
\begin{equation}
\label{eq:deco-of-cbdeg8}
\Db(S)= \sod{\Db(k), \Db(k,\alpha), \Db(k,\beta), \Db(k,\alpha \otimes \beta)}.
\end{equation}
We can show that the nontrivial components of this decomposition are a birational invariant, which allows
us to conclude that the Griffiths--Kuznetsov component is well defined.

\begin{proposition}\label{prop:GK-for-cb8}
Suppose that $S \to C$ is a conic bundle of degree $8$ and that $S_1 \dashrightarrow S$ is a birational map.
Then there is a semiorthogonal decomposition
$$\cat{A}_{S'} = \sod{ \cat{T}, \Db(k,\alpha), \Db(k,\alpha'),\Db(k, \alpha \otimes \beta)}$$ 
\end{proposition}
\begin{remark}
Notice that if $\alpha=0$ (i.e. $C = \PP^1$), we can include $\Db(k,\alpha)$ in $\cat{T}$, and similarly for $\beta=0$ (i.e. $\pi$ has a section),
we can include $\Db(k,\beta)$ in $\cat{T}$
\end{remark}

\begin{proof}
Notice that, over $\ksep$, the conic bundle $S_{\ksep}$ is isomorphic to a Hirzebruch surface
$\FF_n$, that is, a $\PP^1$-bundle $\overline{\pi}: S_\ksep \to \PP^1_\ksep$. One can check that the semiorthogonal
decomposition \eqref{eq:deco-of-cbdeg8} base-changes to the semiorthogonal decomposition
\begin{equation}
\label{eq:deco-of-Fn}
\Db(S_\ksep)= \sod{\ko_{S_\ksep},\ko_{S_\ksep}(F),\ko_{S_\ksep}(\Sigma),\ko_{S_\ksep}(\Sigma+F)},
\end{equation}
where we denoted by $F$ and $\Sigma$ the fiber and the section of $\overline{\pi}$ respectively.

Suppose first that $S_1$ is minimal, so that $S$ can be decomposed in a series of links of type either II or IV (a link
of type III is a blow-down, so $S_1$ is not minimal).
Let us first consider links of type II, which are described by Iskovskikh \cite[Thm. 2;6 (ii)]{isko-sarkisov-complete}.
Let $\pi: S \to C$ be a conic bundle of degree 8, and 
$$\xymatrix{S \ar[d] & X \ar[l]_\sigma \ar[r]^\tau& S' \ar[d] \\
C & & C' \ar[ll]_\isom}$$
where $\sigma:X \to S$ and $\tau: X \to S'$ are blow-ups in a point $x$ and $x'$ respectively of the same degree $d$, and $S' \to C'$ is a conic bundle. 
We denote by $E$ and $E'$ the exceptional divisors of the blow ups. We first
notice that $S'$ has also degree 8 and that $C'\simeq C$. The degree $d$ of the point can be either $1$, $2$, or $4$.
It is easy to check that,
over $\ksep$, the link is just a composition of $d$ elementary transformations of the Hirzebruch surface $S_{\ksep}$. In particular,
if there is a point of degree $1$, then we obtain a birational map between $S$ and a quadric with a rational point, so that $S$ is already
a Hirzebruch surface and there is nothing to prove. So we assume that $d$ is either $2$ or $4$. Moreover, as one can check over $\ksep$,
we have $-K_S = 2\Sigma - (n-2)F$ and $-K_{S'}= 2\Sigma' - (n+d-2)F'$. 
In the Picard group of $X_{\ksep}$ we have the following relations (we omit the pull-back notation):
\begin{align}
F &= F', \\
E &= d\Sigma' - E',\\
-K_S &= -K_{S'} + dF' + 2E',\\
\Sigma &= \Sigma' - E',
\end{align}
where the last equality, is obtained combining the first and the third one.
In particular, we obtain an identification $\ko(F)$ with $\ko(F')$ and that the the equivalence 
$\otimes \ko(-E')$ of $\Db(X_{\ksep})$ sends the exceptional line bundle $\ko(\Sigma)$
to the exceptional line bundle $\ko(\Sigma')$ and the exceptional line bundle $\ko(\Sigma+F)$
to the exceptional line bundle $\ko(\Sigma'+F')$. We notice that, since $E'$ is defined over $k$,
the latter equivalence is defined over $k$ as well. From this we obtain equivalences
$\Db(k,\alpha)\simeq \Db(k,\alpha')$ (the identity), $\Db(k, \beta) \simeq \Db(k,\beta')$ and
$\Db(k,\alpha\otimes\beta) \simeq \Db(k,\alpha'\otimes\beta')$ (induced by $\otimes\ko(-E')$).

Let us now consider links of type IV. Thanks to Iskovskikh, this is possible only in the cases where
$S=S'=C\times C'$ is the product of two Severi--Brauer curves and the conic bundle structures are
given by the two projections. Let $\alpha$ and $\alpha'$ in $\Br(k)$ be the classes of $C$ and $C'$
respectively. Then there are natural decompositions:
\begin{equation}\label{eq:change-of-cb-strct}
\begin{array}{ll}
\Db(C)=\sod{\Db(k),\Db(k,\alpha)} & \,\,\,\, \Db(C,\kc_0)=\sod{\Db(k,\alpha'),\Db(\alpha\otimes\alpha')}\\
\\
\Db(C')=\sod{\Db(k),\Db(k,\alpha')} & \,\,\,\, \Db(C,\kc'_0)=\sod{\Db(k,\alpha),\Db(\alpha'\otimes\alpha)}
\end{array}
\end{equation}
It follows that the decompositions \eqref{eq:deco-of-cbdeg8} obtained by the conic bundle structures 
are composed by equivalent categories.

This proves the Proposition for $S_1$ minimal. If $S_1$ is not minimal, then just consider a minimal
model $S_1 \to S_0$ and use the blow-up formula.
\end{proof}

We finally classify birational equivalent involution surfaces
based on algebraic methods. This can be seen as the algebraic
interpretation of the classification of all type IV links for such surfaces.

\begin{proposition}
\label{prop:birational_involution_partners}
Let $S$ and $S'$ be nonrational involution varieties over a field
$k$.  Then $S$ and $S'$ are $k$-birational if and only if either they
are isomorphic or $S = \SB(B_1) \times \SB(B_2)$ and $S' = \SB(B_1')
\times \SB(B_2')$ and the Klein four subgroups of the Brauer group
generated by $B_i$ and $B_i'$ coincide and consist solely of
quaternion algebras.
\end{proposition}
\begin{proof}
Let $S$ and $S'$ are be $k$-birational involution varieties associated
to quadratic pairs $(A,\sigma)$ and $(A',\sigma')$.  The only case not
ruled out by Proposition~\ref{prop:dP8_semi-rigid} is that of $S$
(hence $S'$) of index 2, trivial discriminant, and with (at least)
$(A,\sigma)$ nonsplit.
By the classification of degree 4 algebras with orthogonal involution
of trivial discriminant \cite[\S15.B]{book_of_involutions},
$(A,\sigma)$ is isomorphic to $(C_0^+,\tau_0^+)\tensor
(C_0^-,\tau_0^-)$, where $\tau_0^\pm$ is the standard involution on
the quaternion algebra $C_0^\pm$, and similarly for $(A',\sigma')$.
By \eqref{eq:tao} and the considerations in the proof of
Proposition~\ref{prop:dP8_semi-rigid}, the Klein four subgroups of
$\Br(k)$ generated by $C_0^\pm$ and $C_0'{}^\pm$ coincide and consist
solely of quaternion algebras (since $A$ and $A'$ have index at most
2).

Now we verify that if $B_1,B_2$ and $B_1',B_2'$ are Brauer classes
generating the same Klein four subgroup consisting of classes of index
at most 2, then $S=\SB(B_1)\times\SB(B_2)$ and
$S'=\SB(B_1')\times\SB(B_2')$ are $k$-birational.  If the unordered
pairs $(B_1,B_2)$ and $(B_1',B_2')$ are isomorphic then $S \isom S'$.
Otherwise, without loss of generality, we can write $B_1'=B_1$ and
$B_2'=H$, where $B_1\tensor B_2 \isom M_2(H)$.  However, since
$[H]=[B_1]+[B_2] \in \Br(k)$ and $\ker(\Br(k) \to \Br(\SB(B_1))) =
\langle [B_1]\rangle$, the images of $[H]$ and $[B_2]$ coincide in
$\Br(k(\SB(B_1)))$, and thus we conclude that $k(S') =
k(\SB(B_1))\tensor k(\SB(H)) \isom k(\SB(B_1))\tensor k(\SB(B_2)) =
k(S)$ so that $S$ and $S'$ are $k$-birational.
\end{proof}

As a special case, this result says that the involution surfaces $C
\times C$ and $\PP^1 \times C$ are $k$-birational for a any
Severi--Brauer curve $C$ over $k$.

\end{document}